\documentclass{amsart}
\usepackage{amssymb}
\usepackage{amsfonts}
\usepackage{geometry}
\usepackage{color}

\newtheorem{theorem}{Theorem}
\theoremstyle{plain}

\newtheorem{conclusion}[theorem]{Conclusion}

\newtheorem{corollary}[theorem]{Corollary}

\newtheorem{definition}[theorem]{Definition}

\newtheorem{lemma}[theorem]{Lemma}

\newtheorem{remark}[theorem]{Remark}

\numberwithin{equation}{section}
\input{tcilatex}
\begin{document}
\title[Sums of squares II]{Sums of squares II: matrix functions}
\author{Lyudmila Korobenko}
\address{Reed College, Portland, Oregon, USA, korobenko@reed.edu}
\author{Eric Sawyer}
\address{McMaster University, Hamilton, Ontario, Canada, sawyer@mcmaster.ca}
\thanks{The second author is partially supported by NSERC grant number 12409
and the McKay Research Chair grant at McMaster University}

\begin{abstract}
This paper is the second in a series of three papers devoted to sums of
squares and hypoellipticity of infinitely degenerate operators. In the first
paper we established \ a sharp $\omega $-monotonicity criterion for writing
a smooth nonnegative function $f$ that is flat at, and positive away from,
the origin, as a finite sum of squares of $C^{2,\delta }$ functions for some 
$\delta >0$, namely that $f$ is $\omega $-monotone for some H\"{o}lder
modulus of continuity $\omega $. Counterexamples were provided for any
larger modulus of continuity.

In this paper we consider the analogous sum of squares problem for smooth
nonnegative matrix functions $M$ that are flat at, and positive away from,
the origin. We show that such a matrix function $M=\left[ a_{kj}\right]
_{k,j=1}^{n}$ can be written as a finite sum of squares of $C^{2,\delta }$
vector fields if the diagonal entries $a_{kk}$ are $\omega $-monotone for
some H\"{o}lder modulus of continuity $\omega $, and if the off diagonal
entries satisfy certain differential bounds in terms of powers of the
diagonal entries. Examples are given to show that in some cases at least,
these differential inequalities cannot be relaxed.

Various refinements of this result are also given in which one or more of
the diagonal entries need not be assumed to have any monotonicity properties
at all. These sum of squares decompositions will be applied to
hypoellipticity in the infnitely degenerate regime in the third paper in
this series.
\end{abstract}

\maketitle
\tableofcontents

\section{Introduction}

In the theory of partial differential operators, the ability to write a pure
second order real differential operator $L\left( x\right) =\func{trace}\left[
\mathbf{A}\left( x\right) \nabla \otimes \nabla \right] $ as a sum of
squares of vector fields $X_{j}\left( x\right) $, i.e. $\mathbf{A}\left(
x\right) =\sum_{j=1}^{N}X_{j}\left( x\right) X_{j}\left( x\right) ^{\func{tr}%
}$, with some specified smoothness, has proven to be of immense value. See
e. g. work of H\"{o}rmander \cite{Hor}, Rothschild and Stein \cite{RoSt},
Christ \cite{Chr} and Sawyer, Rios and Wheeden \cite[RiSaWh]{RiSaWh} to
mention just a few. If $\mathbf{A}\left( x\right) $ is nonnegative, then the
spectral theorem shows that $\mathbf{A}\left( x\right) =\sum_{j=1}^{n}\left( 
\sqrt{\lambda _{j}\left( x\right) }\mathbf{v}_{j}\left( x\right) \right)
\left( \sqrt{\lambda _{j}\left( x\right) }\mathbf{v}_{j}\left( x\right)
\right) ^{\func{tr}}$ where $\left\{ \mathbf{v}_{j}\left( x\right) \right\}
_{j=1}^{n}$ is an orthonormal set of eigenvectors for $\mathbf{A}\left(
x\right) $, but little can be typically said regarding smoothnes of $\sqrt{%
\lambda _{j}\left( x\right) }\mathbf{v}_{j}\left( x\right) $. Thus it
becomes an important question as to whether or not a given matrix function $%
\mathbf{A}\left( x\right) $ can be represented as a finite sum of squares of
vector functions with preassigned smoothness. We will refer to the rank one
matrix $X_{j}\left( x\right) X_{j}\left( x\right) ^{\func{tr}}=X_{j}\left(
x\right) \otimes X_{j}\left( x\right) $ either as a square of a vector
field, or as a positive dyad\footnote{%
A nontrivial dyad $a\otimes b$ is positive if and only if $a=b$. Indeed, 
\begin{equation*}
\xi ^{\func{tr}}\left[ a\otimes b\right] \xi =\xi ^{\func{tr}}\left\langle
a,\xi \right\rangle b=\left\langle \xi ,b\right\rangle \left\langle a,\xi
\right\rangle \geq 0
\end{equation*}%
for all $\xi $ if and only if $a=b$.}. We point out the obvious fact that a
sum of squares $\sum_{j=1}^{N}X_{j}X_{j}^{\func{tr}}$ is positive
semidefinite and equals $\mathbf{XX}^{\func{tr}}$ where $\mathbf{X}$ is the $%
n\times N$ matrix with columns $X_{j}$.

In this paper we extend the sums of squares results for \emph{scalar}
functions obtained in \cite{KoSa1} to the setting of matrix-valued
functions, where these decompositions will be used in the subsequent and
final paper of this series \cite{KoSa3} to obtain hypoellipticity results in
the infinitely degenerate regime.

In order for an $n\times n$ nonnegative matrix function $A\left( x\right) =%
\left[ a_{ij}\left( x\right) \right] _{1\leq i,j\leq n}$ to be a sum of
squares $\sum_{\ell =1}^{N}\mathbf{v}^{\ell }\left( x\right) \otimes \mathbf{%
v}^{\ell }\left( x\right) $ of vector functions (with some specified
smoothness such as $C^{2,\delta }$), it is of course necessary that the
diagonal entries $a_{ii}\left( x\right) $ themselves be a sum of such
squares, namely $\sum_{\ell =1}^{N}\left( \mathbf{v}^{\ell }\left( x\right)
\right) _{i}^{2}$. A well known and clever construction of Fefferman-Phong 
\cite{FePh} has been used by Tataru \cite{Tat} and Bony \cite{Bon}, see also
Guan \cite{Gua}, to show that every $C^{3,1}$ scalar function $f\left(
x\right) $ can be written as a sum of squares of $C^{1,1}$ functions.
However, this degree of smoothness falls short of what is needed in the
theory of hypoellipticity of sums of squares of infinitely degenerate vector
fields, the appropriate level of smoothness being $C^{2,\delta }$ for some $%
\delta >0$. The theorem proved in \cite{KoSa1} shows that a nonnegative
scalar function $f\left( x\right) $ can be written as a sum of squares of $%
C^{2,\delta }$ functions provided $f$ is \emph{strongly} $C^{4,2\delta }$,
by which we mean $f$ is $C^{4,2\delta }$, positive away from the origin, and
satisfies $\left\vert \nabla ^{2}f\left( x\right) \right\vert ,\left\vert
\nabla ^{4}f\left( x\right) \right\vert \lesssim f\left( x\right) ^{\delta
^{\prime }}$ for some $\delta ^{\prime }>0$. A large class of examples of
such functions is the set of H\"{o}lder monotone functions, those for which
there is $C,s>0$ such that 
\begin{equation*}
f\left( y\right) \leq Cf\left( x\right) ^{s}\text{ for }y\in B\left( \frac{x%
}{2},\frac{\left\vert x\right\vert }{2}\right) .
\end{equation*}%
This class is sharp in the sense that the logarithmic version of this
inequality fails to imply a decomposition into a sum of squares of $%
C^{2,\delta }$ functions, see part (2) of Theorem \ref{sharpness} below.

Neverthess, this latter condition on the diagonal entries of $\mathbf{A}%
\left( x\right) $ is not sufficient for $\mathbf{A}\left( x\right) $ to be a
sum of squares of $C^{2,\delta }$ vector fields, as our counterexamples
below show. Instead, we assume in addition a collection of natural
inequalities on derivatives up to order four of the off diagonal entries $%
a_{ij}\left( x\right) $ having the form:%
\begin{equation*}
\left\vert D^{\alpha }a_{k,j}\left( x\right) \right\vert \lesssim \left(
\min_{1\leq \ell \leq j}a_{j,j}\left( x\right) \right) ^{\frac{1}{2}+\left(
2-\left\vert \mu \right\vert \right) \varepsilon },\ \ \ \ \ \text{fo all }%
k<j\text{, }0\leq \left\vert \mu \right\vert \leq 4\text{, and some }%
\varepsilon \geq \frac{1}{4}.
\end{equation*}%
We also give examples showing that such differential inequalities on off
diagonal entries of the matrix are sharp in some cases.

But first we recall the main results from \cite{KoSa1} that will be used
here.

\begin{definition}
A scalar or matrix function $A:\mathbb{R}^{n}\rightarrow \mathbb{R}^{N\times
N}$ is \emph{elliptical} if the associated quadratic form $Q_{A}\left( x,\xi
\right) \equiv \xi ^{\func{tr}}A\left( x\right) \xi $ is strictly positive
away from the `axes' $\left\{ 0\right\} \times \mathbb{R}^{N}$ and $\mathbb{R%
}^{n}\times \left\{ 0\right\} $ in $\mathbb{R}^{n}\times \mathbb{R}^{N}$.
\end{definition}

\begin{definition}
We say that a scalar, vector or matrix function $g:\mathbb{R}^{n}\rightarrow %
\left[ 0,\infty \right) $ is \emph{regular} if $g\in \dbigcup\limits_{\delta
>0}C^{2,\delta }\left( \mathbb{R}^{n}\right) $, i.e. $g$ is $C^{2,\delta }$
for some $0<\delta <1$.
\end{definition}

Now we recall various notions of monotonicity from \cite{KoSa1}.

\begin{definition}
Let $\omega $ be a modulus of continuity on $\left[ 0,1\right] $, and let $%
f: $ $\mathbb{R}^{n}\rightarrow \left[ 0,\infty \right) $. We say
\end{definition}

\begin{enumerate}
\item $f$ is $\omega $\emph{-monotone} if $f\left( y\right) \leq C\omega
\left( f\left( x\right) \right) $ for $y\in B\left( \frac{x}{2},\frac{%
\left\vert x\right\vert }{2}\right) $ and some positive constant $C$,

\item $\omega _{s}\left( t\right) =\left\{ 
\begin{array}{ccc}
t\left( 1+\ln \frac{1}{t}\right) & \text{ if } & s=1 \\ 
t^{s} & \text{ if } & 0<s<1 \\ 
\frac{1}{1+\ln \frac{1}{t}} & \text{ if } & s=0%
\end{array}%
\right. $, for $0\leq s,t\leq 1$,

\item $f$ is \emph{nearly monotone} if $f$ is $\omega _{s}$-monotone for
every $0\leq s<1$,

\item $f$ is \emph{H\"{o}lder monotone} if $f$ is $\omega _{s}$-monotone for
some $0<s\leq 1$.
\end{enumerate}

Finally we recall the following seminorm from \cite{Bon}, 
\begin{equation}
\left[ h\right] _{\alpha ,\delta }\left( x\right) \equiv
\limsup_{y,z\rightarrow x}\frac{\left\vert D^{\alpha }h\left( y\right)
-D^{\alpha }h\left( z\right) \right\vert }{\left\vert y-z\right\vert
^{\delta }},  \label{def mod D}
\end{equation}%
We can also norm the space $C^{m,\alpha }\left( K\right) $ with $\left\Vert
h\right\Vert _{m,\alpha }\equiv \left\Vert h\right\Vert _{C^{m}\left(
K\right) }+\sum_{\left\vert \alpha \right\vert =m}\sup_{x\in K}\left[
D^{\alpha }h\right] _{\alpha ,\delta }\left( x\right) $ when $K$ is compact.

\begin{theorem}[Theorem 28 from \protect\cite{KoSa1}]
\label{efs eps}Suppose $0<\delta <1$ and that $f$ is a $C^{4,2\delta }$
function on $\mathbb{R}^{n}$. Let%
\begin{equation*}
\rho \left( x\right) =\rho _{f}\left( x\right) \equiv \max \left\{ f\left(
x\right) ^{\frac{1}{4+2\delta }},\left( \sup_{\Theta \in \mathbb{S}^{n-1}}%
\left[ \partial _{\Theta }^{2}f\left( x\right) \right] _{+}\right) ^{\frac{1%
}{2+2\delta }},\left\vert \nabla ^{4}f\left( x\right) \right\vert ^{\frac{1}{%
2\delta }}\right\} ,\ \ \ \ \ x\in \mathbb{R}^{n}.
\end{equation*}

\begin{enumerate}
\item If $f$ satisfies the differential inequalities%
\begin{equation}
\left\vert \nabla ^{4}f\left( x\right) \right\vert \leq Cf\left( x\right) ^{%
\frac{\delta }{2+\delta }}\text{ and }\left\vert \nabla ^{2}f\left( x\right)
\right\vert \leq Cf\left( x\right) ^{\frac{2\delta \left( 1+\delta \right) }{%
2+\delta }},  \label{diff prov}
\end{equation}%
then $f=\sum_{\ell =1}^{N}g_{\ell }^{2}$ can be decomposed as a finite sum
of squares of functions $g_{\ell }\in C^{2+\delta }\left( \mathbb{R}%
^{2}\right) $ where 
\begin{eqnarray}
\left\vert D^{\alpha }g_{\ell }\left( x\right) \right\vert &\leq &C\rho
\left( x\right) ^{2+\delta -\left\vert \alpha \right\vert }\leq Cf\left(
x\right) ^{\frac{\delta }{2+\delta }\left( 2+\delta -\left\vert \alpha
\right\vert \right) },\ \ \ \ \ 0\leq \left\vert \alpha \right\vert \leq 2,
\label{root control} \\
\left[ g_{\ell }\right] _{\alpha ,\delta }\left( x\right) &\leq &C,\ \ \ \ \
\left\vert \alpha \right\vert =2.  \notag
\end{eqnarray}%
and%
\begin{eqnarray*}
\left\vert D^{\alpha }g_{\ell }^{2}\left( x\right) \right\vert &\leq &C\rho
\left( x\right) ^{4+2\delta -\left\vert \alpha \right\vert }\leq Cf\left(
x\right) ^{\frac{\delta }{2+\delta }\left( 4+2\delta -\left\vert \alpha
\right\vert \right) },\ \ \ \ \ 0\leq \left\vert \alpha \right\vert \leq 4,
\\
\left[ g_{\ell }^{2}\right] _{\alpha ,2\delta }\left( x\right) &\leq &C,\ \
\ \ \ \left\vert \alpha \right\vert =4.
\end{eqnarray*}

\item In particular, the inequalities (\ref{diff prov}) hold provided $f$ is
flat, smooth and $\omega _{s}$-monotone for some $s<1$ satisfying%
\begin{equation}
s>\sqrt[4]{\delta }\sqrt[4]{\frac{2}{4+2\delta }},\ \ \ \text{and \ \ \ }%
s\geq \sqrt{\delta }\sqrt{\frac{2+2\delta }{2+\delta }}.  \label{eps delta}
\end{equation}
\end{enumerate}
\end{theorem}

\begin{remark}
The inequalities (\ref{diff prov}) also hold if the smoothness assumption on 
$f$ is relaxed to $f\in C^{k}$, provided that $s$ is replaced by $s-\frac{C}{%
k}$ in (\ref{eps delta}) for a sufficiently large constant $C$ independent
of $k$.
\end{remark}

The next result shows that the H\"{o}lder monotone class comes close to
being sharp for a decomposition into a finite sum of squares of regular
functions.

\begin{theorem}
\label{sharpness}Let $n\geq 5$. There is an elliptical, flat, smooth $\omega
_{0}$-monotone function $f$ that cannot be written as a finite sum of
squares of regular functions. Moreover, $\omega _{0}$ can be replaced by any
modulus of continuity $\omega $ with $\omega _{s}\ll \omega $, i.e. $%
\lim_{t\searrow 0}\frac{\omega _{s}\left( t\right) }{\omega \left( t\right) }%
=0$, for all $0<s<1$.
\end{theorem}

\begin{corollary}
Suppose that $f:\mathbb{R}^{n}\rightarrow \left[ 0,\infty \right) $ is
elliptical, flat and smooth.

\begin{enumerate}
\item Then $f$ can written as a finite sum of squares of regular functions
if $f$ is H\"{o}lder monotone.

\item Conversely, for any modulus of continuity $\omega $ satisfying $\omega
_{s}\ll \omega $ for all $0<s<1$, there is an $\omega $-monotone function $f$
that cannot be written as a finite sum of squares of regular functions.
\end{enumerate}
\end{corollary}

Finally, we recall the following estimates for nearly monotone functions
from \cite{KoSa1}.

\begin{theorem}
\label{main intro}Let $n\geq 1$. Suppose that $f:B\left( 0,a\right)
\rightarrow \left[ 0,\infty \right) $ is an elliptical flat smooth function
on $B\left( 0,a\right) \subset \mathbb{R}^{n}$. Then the first three of the
following four conditions are equivalent. Moreover, the fourth condition,
which holds in particular if $f$ is $\omega _{1}$ monotone, implies the
first three conditions, but not conversely. Finally, for any $0<s<1$, there
is an $\omega _{s}$-monotone function $f$ such that $f^{\frac{1}{s}-1}$ is
not smooth.

\begin{enumerate}
\item There is $\delta >0$ such that $f\left( x\right) ^{\gamma }$ is smooth
on $B\left( 0,a\right) $ for all $0<\gamma <\delta $.

\item For every $m\geq 1$ and $0<s<1$, there is a positive constant $\Gamma
_{n,m,s}$ such that%
\begin{equation*}
\left\vert \nabla ^{m}f\left( x\right) \right\vert \leq \Gamma
_{n,m,s}f\left( x\right) ^{s},\ \ \ \ \ \text{for }x\in B\left( 0,a\right) .
\end{equation*}

\item The functions $f\left( x\right) ^{\gamma }$ are flat smooth functions
on $B\left( 0,a\right) $ for all $\gamma >0$.

\item The function $f$ is nearly monotone.
\end{enumerate}
\end{theorem}

In this paper we will apply the above sums of squares representations for
scalar functions to obtain representations of matrix functions as sums of
squares of $C^{2,\delta }$ vector fields. For the reader's convenience we
include a schematic diagram of connections between theorems. Results in a
double box are logical ends.%
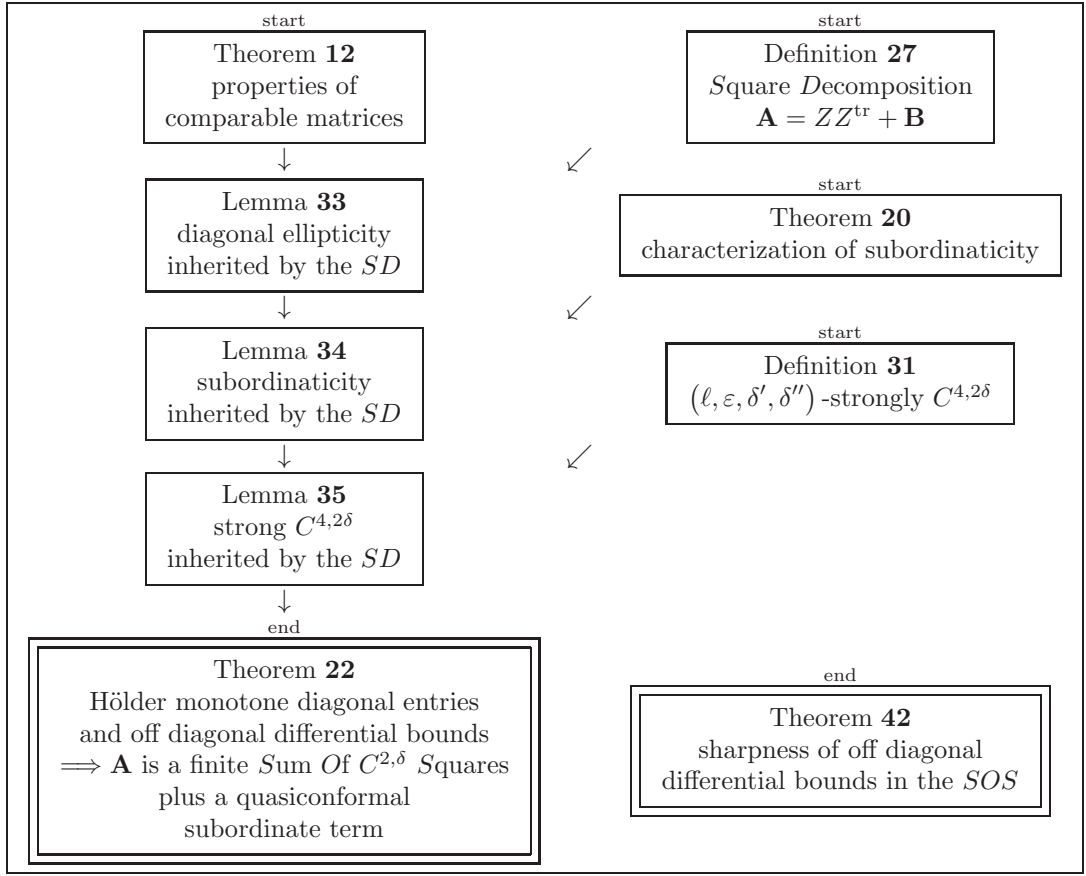
\begin{equation*}
\fbox{$%
\begin{array}{ccc}
\overset{\func{start}}{\fbox{$%
\begin{array}{c}
\text{Theorem}\ \mathbf{12} \\ 
\text{properties of} \\ 
\text{comparable matrices}%
\end{array}%
$}} &  & \overset{\func{start}}{\fbox{$%
\begin{array}{c}
\text{Definition}\ \mathbf{27} \\ 
S\text{quare }D\text{ecomposition} \\ 
\mathbf{A}=ZZ^{\func{tr}}+\mathbf{B}%
\end{array}%
$}} \\ 
\downarrow  & \swarrow  &  \\ 
\fbox{$%
\begin{array}{c}
\text{Lemma}\ \mathbf{33} \\ 
\text{diagonal ellipticity} \\ 
\text{inherited by the }SD%
\end{array}%
$} &  & \overset{\func{start}}{\fbox{$%
\begin{array}{c}
\text{Theorem}\ \mathbf{20} \\ 
\text{characterization of subordinaticity}%
\end{array}%
$}} \\ 
\downarrow  & \swarrow  &  \\ 
\fbox{$%
\begin{array}{c}
\text{Lemma}\ \mathbf{34} \\ 
\text{subordinaticity} \\ 
\text{inherited by the }SD%
\end{array}%
$} &  & \overset{\func{start}}{\fbox{$%
\begin{array}{c}
\text{Definition}\ \mathbf{31} \\ 
\left( \ell ,\varepsilon ,\delta ^{\prime },\delta ^{\prime \prime }\right) 
\text{-strongly }C^{4,2\delta }%
\end{array}%
$}} \\ 
\downarrow  & \swarrow  &  \\ 
\fbox{$%
\begin{array}{c}
\text{Lemma}\ \mathbf{35} \\ 
\text{strong }C^{4,2\delta } \\ 
\text{inherited by the }SD%
\end{array}%
$} &  &  \\ 
\downarrow  &  &  \\ 
\overset{\func{end}}{\fbox{$\fbox{$%
\begin{array}{c}
\text{Theorem}\ \mathbf{22} \\ 
\text{H\"{o}lder monotone diagonal entries} \\ 
\text{and off diagonal differential bounds} \\ 
\Longrightarrow \mathbf{A}\text{ is a finite }S\text{um }O\text{f }%
C^{2,\delta }\text{ }S\text{quares} \\ 
\text{plus a quasiconformal} \\ 
\text{subordinate term}%
\end{array}%
$}$}} &  & \overset{\func{end}}{\fbox{$\fbox{$%
\begin{array}{c}
\text{Theorem}\ \mathbf{42} \\ 
\text{sharpness of off diagonal} \\ 
\text{differential bounds in the }SOS%
\end{array}%
$}$}}%
\end{array}%
$}
\end{equation*}

\section{Statement of main matrix decomposition theorems}

\begin{definition}
Let $A$ and $B$ be real symmetric positive semidefinite $n\times n$
matrices. We define $A\preccurlyeq B$ if $B-A$ is positive semidefinite. Let 
$\beta <\alpha $ be positive constants. A real symmetric positive
semidefinite $n\times n$ matrix $A$ is said to be $\left( \beta ,\alpha
\right) $-\emph{comparable} to a symmetric $n\times n$ matrix $B$, written $%
A\sim _{\beta ,\alpha }B$, if $\beta B\preccurlyeq A\preccurlyeq \alpha B$,
i.e.%
\begin{equation}
\beta \ \xi ^{\limfunc{tr}}B\xi \leq \xi ^{\limfunc{tr}}A\xi \leq \alpha \
\xi ^{\limfunc{tr}}B\xi ,\ \ \ \ \ \text{for all }\xi \in \mathbb{R}^{n}.
\label{pos def}
\end{equation}
\end{definition}

\begin{definition}
Given two symmetric matrix-valued functions $\mathbf{A}\left( x\right) $ and 
$\mathbf{B}\left( x\right) $, we say $\mathbf{A}\left( x\right) $ is \emph{%
comparable} to $\mathbf{B}\left( x\right) $ if there are positive constants $%
0<\beta <\alpha $ \emph{independent} of $x$, such that $\mathbf{A}\left(
x\right) $ is $\left( \beta ,\alpha \right) $-comparable to $\mathbf{B}%
\left( x\right) $ for all $x\in \mathbb{R}^{n}$. In this case we write $%
\mathbf{A}\left( x\right) \sim \mathbf{B}\left( x\right) $.
\end{definition}

Note that if $A$ is \emph{comparable} to $B$, then both $A$ and $B$ are
positive semidefinite. Indeed, both $0\leq \left( \alpha -\beta \right) \xi
^{\limfunc{tr}}B\xi $ and $0\leq \left( \frac{1}{\beta }-\frac{1}{\alpha }%
\right) \xi ^{\limfunc{tr}}A\xi $ hold for all $\xi \in \mathbb{R}^{n}$.
Moreover, for any $x$ we have that $\mathbf{A}\left( x\right) $ is positive
definite if and only if $\mathbf{B}\left( x\right) $ is positive definite.

We first give a simple characterization of when a symmetric positive
semidefinite matrix-valued function is comparable to a diagonal
matrix-valued function that is positive away from the origin. In order to
state this result precisely, we first establish the general fact that if a
matrix $\mathbf{A}$ is $\left( \beta ,\alpha \right) $-comparable to a
diagonal matrix $\mathbf{D}_{\mathbf{\lambda }}$ as above, then $\mathbf{A}$
is $\left( \frac{\beta }{\alpha },\frac{\alpha }{\beta }\right) $-comparable
to its associated diagonal matrix 
\begin{equation*}
\mathbf{A}_{\func{diag}}\equiv \mathbf{D}_{\mathbf{F}}=\left[ 
\begin{array}{cccc}
F_{1} & 0 & \cdots & 0 \\ 
0 & F_{2} & \cdots & 0 \\ 
\vdots & \vdots & \ddots & \vdots \\ 
0 & 0 & \cdots & F_{n}%
\end{array}%
\right] ,
\end{equation*}%
in short $\mathbf{A}\sim _{\beta ,\alpha }\mathbf{D}_{\mathbf{\lambda }%
}\Longrightarrow \mathbf{A}\sim _{\frac{\beta }{\alpha },\frac{\alpha }{%
\beta }}\mathbf{A}_{\func{diag}}$, equivalently 
\begin{equation*}
\beta \mathbf{D}_{\mathbf{\lambda }}\preccurlyeq \mathbf{A}\preccurlyeq
\alpha \mathbf{D}_{\mathbf{\lambda }}\Longrightarrow \frac{\beta }{\alpha }%
\mathbf{A}_{\func{diag}}\preccurlyeq \mathbf{A}\preccurlyeq \frac{\alpha }{%
\beta }\mathbf{A}_{\func{diag}}.
\end{equation*}%
Indeed, it is an easy exercise to show that $\beta \mathbf{D}_{\mathbf{%
\lambda }}\preccurlyeq \mathbf{A}\preccurlyeq \alpha \mathbf{D}_{\mathbf{%
\lambda }}$ implies $\beta \lambda _{k}\leq F_{k}\leq \alpha \lambda _{k}$
implies $\beta \mathbf{D}_{\mathbf{\lambda }}\preccurlyeq \mathbf{D}_{%
\mathbf{F}}\preccurlyeq \alpha \mathbf{D}_{\mathbf{\lambda }}$, and since $%
\mathbf{D}_{\mathbf{F}}=\mathbf{A}_{\func{diag}}$ we then have%
\begin{equation*}
\frac{\beta }{\alpha }\mathbf{A}_{\func{diag}}\preccurlyeq \beta \mathbf{D}_{%
\mathbf{\lambda }}\preccurlyeq \mathbf{A}\preccurlyeq \alpha \mathbf{D}_{%
\mathbf{\lambda }}\preccurlyeq \frac{\alpha }{\beta }\mathbf{A}_{\func{diag}%
}\ .
\end{equation*}

This becomes important when we consider matrix-valued functions $\mathbf{A}%
\left( x\right) $ (of a variable $x\in \mathbb{R}^{m}$) that are $\left(
\beta ,\alpha \right) $-comparable to a diagonal matrix-valued function $%
\mathbf{D}\left( x\right) $, and which are positive definite away from the
origin. For example, the matrix function $\mathbf{A}\left( x\right) \equiv %
\left[ 
\begin{array}{cc}
1 & 1-e^{-\frac{1}{x^{2}}} \\ 
1-e^{-\frac{1}{x^{2}}} & 1%
\end{array}%
\right] $ is positive definite away from the origin, yet is \emph{not} $%
\left( \beta ,\alpha \right) $-comparable to any diagonal matrix-valued
function $\mathbf{D}\left( x\right) $. Indeed, if it were, then $\mathbf{A}%
\left( x\right) $ would be $\left( \beta ,\alpha \right) $-comparable to its
associated diagonal matrix, i.e. the identity matrix $\left[ 
\begin{array}{cc}
1 & 0 \\ 
0 & 1%
\end{array}%
\right] $, which clearly fails for $x$ close to the origin. It turns out
that for matix functions $\mathbf{A}\left( x\right) $ that are both positive
definite away from the origin and $\left( \beta ,\alpha \right) $-comparable
to some diagonal matrix-valued function $\mathbf{D}\left( x\right) $, we can
derive useful consequences on the off diagonal entries of $\mathbf{A}\left(
x\right) $. Since such matrix functions play a major role in the sequel we
make a definition to encompass them.

\begin{definition}
We say that a matrix-valued function $\mathbf{A}\left( x\right) $ is \emph{%
diagonally elliptical} on $\mathbb{R}^{m}$ if $\mathbf{A}\left( x\right) $
is comparable to a diagonal matrix-valued function $\mathbf{D}\left(
x\right) $ for all $x\in \mathbb{R}^{n}$, and if $\mathbf{A}\left( x\right) $
is positive definite away from the origin.
\end{definition}

We see from the above discussion that, if $\mathbf{A}\left( x\right) $ is
diagonally elliptical, then $\mathbf{A}\left( x\right) $ is comparable to
its associated diagonal matrix-valued function $\mathbf{A}_{\func{diag}%
}\left( x\right) $, a fact which will play a key role in deriving useful
properties of diagonally elliptical matrix functions. In the next two
subsections we state and prove a relatively simple property of a diagonally
elliptic matrix function, as well as an easy characterization of when a
matrix function is subordinate. In the third subsection we state our sum of $%
C^{2,\delta }$ squares theorem, which is then proved in the third section,
and finally we demonstrate in Section 4 some results on sharpness.

\subsection{The comparability theorem}

First note that if $\mathbf{A}\left( x\right) \sim \mathbf{B}\left( x\right) 
$, then $\widehat{\mathbf{A}}\left( x\right) \sim \widehat{\mathbf{B}}\left(
x\right) $, with the same comparability constants $0<\beta <\alpha <\infty $%
, where $\widehat{\mathbf{A}}\left( x\right) $ is any principal submatrix of 
$\mathbf{A}\left( x\right) $ and $\widehat{\mathbf{B}}\left( x\right) $ is
the corresponding principal submatrix of $\mathbf{B}\left( x\right) $. Here
is the only other consequence of comparability that we will need.

\begin{theorem}
\label{comp mat}A symmetric positive definite $n\times n$ matrix-valued
function 
\begin{equation*}
\mathbf{A}\left( x\right) =\left[ a_{k,j}\left( x\right) \right]
_{k,j=1}^{n}=\left[ 
\begin{array}{cc}
a_{11}\left( x\right) & b\left( x\right) ^{\func{tr}} \\ 
b\left( x\right) & \mathbf{D}\left( x\right)%
\end{array}%
\right]
\end{equation*}%
is comparable to its associated diagonal matrix-valued function $\mathbf{A}_{%
\func{diag}}\left( x\right) $\emph{\ only if} there is $0<\beta <1$ such that%
\begin{equation}
b\left( x\right) ^{\func{tr}}\left[ \mathbf{D}\left( x\right) -\beta \mathbf{%
D}_{\func{diag}}\left( x\right) \right] ^{-1}b\left( x\right) <\left(
1-\beta \right) a_{11}\left( x\right) ,\ \ \ \ \ \text{ for all }x.
\label{necc cond n}
\end{equation}%
In particular,%
\begin{equation}
\left\vert a_{k,j}\left( x\right) \right\vert <\left( 1-\beta \right) \sqrt{%
a_{k,k}\left( x\right) a_{j,j}\left( x\right) },\ \ \ \ \ \text{for all }%
1\leq k<j\leq n\text{ and all }x.  \label{2 by 2}
\end{equation}
\end{theorem}

For this we will use the following three lemmas.

\begin{lemma}
\label{square root}Let $\mathbf{A}$ be a real symmetric positive
semidefinite $n\times n$ matrix. Then there is a real symmetric positive
semidefinite $n\times n$ real matrix $\sqrt{\mathbf{A}}$ satisfying $\mathbf{%
A}=\sqrt{\mathbf{A}}\sqrt{\mathbf{A}}$.
\end{lemma}

\begin{proof}
Let $\mathbf{P}$ be an orthogonal matrix such that $\mathbf{PAP}^{\limfunc{tr%
}}=\mathbf{D}$ is a diagonal matrix with nonnegative entries $\left\{
\lambda _{i}\right\} _{i=1}^{n}$ along the diagonal. If $\sqrt{\mathbf{D}}$
is the diagonal matrix with entries $\left\{ \sqrt{\lambda _{i}}\right\}
_{i=1}^{n}$ along the diagonal, then the matrix $\sqrt{\mathbf{A}}\equiv 
\mathbf{P}^{\limfunc{tr}}\sqrt{\mathbf{D}}\mathbf{P}$ is symmetric and
positive semidefinite, and satisfies 
\begin{equation*}
\sqrt{\mathbf{A}}\sqrt{\mathbf{A}}=\mathbf{P}^{\limfunc{tr}}\sqrt{\mathbf{D}}%
\mathbf{PP}^{\limfunc{tr}}\sqrt{\mathbf{D}}\mathbf{P}=\mathbf{P}^{\limfunc{tr%
}}\sqrt{\mathbf{D}}\sqrt{\mathbf{D}}\mathbf{P}=\mathbf{P}^{\limfunc{tr}}%
\mathbf{DP}=\mathbf{A}.
\end{equation*}
\end{proof}

\begin{lemma}
\label{deter}Let $n\geq 2$. For any number $\alpha $, $\left( n-1\right) $%
-dimensional vector $v$ and invertible $\left( n-1\right) \times \left(
n-1\right) $ matrix $\mathbf{M}$,\ we have the determinant formula%
\begin{equation*}
\det \left[ 
\begin{array}{cc}
\alpha & v^{\limfunc{tr}} \\ 
v & \mathbf{M}%
\end{array}%
\right] =\alpha \det \mathbf{M}-v^{\limfunc{tr}}\left[ \limfunc{co}\mathbf{M}%
\right] ^{\func{tr}}v=\left\{ \alpha -v^{\limfunc{tr}}\mathbf{M}%
^{-1}v\right\} \det \mathbf{M\ .}
\end{equation*}
\end{lemma}

\begin{proof}
We prove only the case $n=3$, in which case we compute that with $\alpha =d$%
, $v=\left( 
\begin{array}{c}
a \\ 
b%
\end{array}%
\right) $ and $\mathbf{M}=\left[ 
\begin{array}{cc}
e & c \\ 
c & f%
\end{array}%
\right] $ we have%
\begin{eqnarray*}
\det \left[ 
\begin{array}{ccc}
d & a & b \\ 
a & e & c \\ 
b & c & f%
\end{array}%
\right] &=&d\det \left[ 
\begin{array}{cc}
e & c \\ 
c & f%
\end{array}%
\right] -a\det \left[ 
\begin{array}{cc}
a & b \\ 
c & f%
\end{array}%
\right] +b\det \left[ 
\begin{array}{cc}
a & b \\ 
e & c%
\end{array}%
\right] \\
&=&d\det \left[ 
\begin{array}{cc}
e & c \\ 
c & f%
\end{array}%
\right] -a\left( af-bc\right) +b\left( ac-be\right) \\
&=&d\det \left[ 
\begin{array}{cc}
e & c \\ 
c & f%
\end{array}%
\right] -a^{2}f+abc+abc-b^{2}e
\end{eqnarray*}%
and we continue with%
\begin{eqnarray*}
\det \left[ 
\begin{array}{ccc}
d & a & b \\ 
a & e & c \\ 
b & c & f%
\end{array}%
\right] &=&d\det \left[ 
\begin{array}{cc}
e & c \\ 
c & f%
\end{array}%
\right] -\left( 
\begin{array}{cc}
a & b%
\end{array}%
\right) \left[ 
\begin{array}{cc}
f & -c \\ 
-c & e%
\end{array}%
\right] \left( 
\begin{array}{c}
a \\ 
b%
\end{array}%
\right) \\
&=&d\det \left[ 
\begin{array}{cc}
e & c \\ 
c & f%
\end{array}%
\right] -\left( 
\begin{array}{cc}
a & b%
\end{array}%
\right) \limfunc{co}\left[ 
\begin{array}{cc}
e & c \\ 
c & f%
\end{array}%
\right] \left( 
\begin{array}{c}
a \\ 
b%
\end{array}%
\right) \\
&=&\det \left[ 
\begin{array}{cc}
e & c \\ 
c & f%
\end{array}%
\right] \left\{ d-\left( 
\begin{array}{cc}
a & b%
\end{array}%
\right) \left[ 
\begin{array}{cc}
e & c \\ 
c & f%
\end{array}%
\right] ^{-1}\left( 
\begin{array}{c}
a \\ 
b%
\end{array}%
\right) \right\} .
\end{eqnarray*}
\end{proof}

The next lemma will use the following well known characterization of
positive definite matrices.

\begin{theorem}
\label{well known}A symmetric $n\times n$ matix $A$ is positive definite 
\emph{if and only if} the determinant of every right lower principal
submatrix of $A$ is positive.
\end{theorem}

\begin{lemma}
\label{lower bound}Let $\mathbf{f}$ and $\mathbf{F}$ be real symmetric
positive definite $\left( n-1\right) \times \left( n-1\right) $ matrices,
and let $v$ be a real $\left( n-1\right) $-dimensional column vector. Then
for $\beta $ real, 
\begin{equation*}
\beta \left[ 
\begin{array}{cc}
H & 0^{\limfunc{tr}} \\ 
0 & \mathbf{f}%
\end{array}%
\right] \prec \left[ 
\begin{array}{cc}
h^{2} & v^{\limfunc{tr}} \\ 
v & \mathbf{F}%
\end{array}%
\right] \ ,
\end{equation*}%
\emph{if and only if}%
\begin{eqnarray*}
h^{2}-\beta H &>&0, \\
\mathbf{G}_{\beta } &\equiv &\mathbf{F}-\beta \mathbf{f}\text{ is real,
symmetric, and positive definite,}
\end{eqnarray*}%
and%
\begin{equation*}
v^{\limfunc{tr}}\mathbf{G}_{\beta }^{-1}v\leq h^{2}-\beta H.
\end{equation*}
\end{lemma}

\begin{proof}
Theorem \ref{well known} and Lemma \ref{deter}\ show that%
\begin{equation*}
0\prec \left[ 
\begin{array}{cc}
h^{2}-\beta H & v^{\limfunc{tr}} \\ 
v & \mathbf{F}-\beta \mathbf{f}%
\end{array}%
\right] =\left[ 
\begin{array}{cc}
h^{2}-\beta H & v^{\limfunc{tr}} \\ 
v & \mathbf{G}_{\beta }%
\end{array}%
\right]
\end{equation*}%
holds if and only if $h^{2}-\beta H>0$, $\mathbf{G}_{\beta }\succ 0$ and 
\begin{equation*}
0<\det \left[ 
\begin{array}{cc}
h^{2}-\beta H & v^{\limfunc{tr}} \\ 
v & \mathbf{G}_{\beta }%
\end{array}%
\right] =\left( \det \mathbf{G}_{\beta }\right) \left\{ h^{2}-\beta H-v^{%
\limfunc{tr}}\mathbf{G}_{\beta }^{-1}v\right\} ,
\end{equation*}%
which in turn holds if and only if%
\begin{equation*}
v^{\limfunc{tr}}\mathbf{G}_{\beta }^{-1}v<h^{2}-\beta H.
\end{equation*}
\end{proof}

Now we can prove Theorem \ref{comp mat}.

\begin{proof}[Proof of Theorem \protect\ref{comp mat}]
Suppose $\mathbf{A}\sim \mathbf{A}_{\func{diag}}$, say 
\begin{equation*}
\beta \mathbf{A}_{\func{diag}}\preccurlyeq \mathbf{A}\preccurlyeq \alpha 
\mathbf{A}_{\func{diag}},\ \ \ \ \ \text{with }0<\beta <\alpha <\infty ,
\end{equation*}%
where $1$ must belong to $\left[ \beta ,\alpha \right] $ in this case. Now
apply Lemma \ref{lower bound} with $h^{2}=a_{11}$, $\mathbf{f}=\mathbf{D}_{%
\func{diag}}$ and $\mathbf{F}=\mathbf{D}$ to obtain%
\begin{equation*}
b^{\limfunc{tr}}\left[ \mathbf{D}-\beta \mathbf{D}_{\func{diag}}\right]
^{-1}b\leq a_{11}-\beta a_{11}=\left( 1-\beta \right) a_{11}\ ,
\end{equation*}%
which is (\ref{necc cond n}).

To obtain (\ref{2 by 2}), we use the fact noted just before Theorem \ref%
{comp mat}, that every $2\times 2$ principal submatrix $\widehat{\mathbf{A}}=%
\left[ 
\begin{array}{cc}
a_{k,k} & a_{k,j} \\ 
a_{k,j} & a_{j,j}%
\end{array}%
\right] $ of $\mathbf{A}$ satisfies $\widehat{\mathbf{A}}\sim \widehat{%
\mathbf{A}}_{\func{diag}}=\widehat{\mathbf{A}_{\func{diag}}}$ with the \emph{%
same} comparability constants $\beta <\alpha $, and hence by (\ref{necc cond
n}) we conclude that%
\begin{eqnarray*}
a_{k,j}\left[ a_{j,j}-\beta a_{j,j}\right] ^{-1}a_{k,j} &\leq &a_{k,k}-\beta
a_{k,k}=\left( 1-\beta \right) a_{k,k}, \\
\text{i.e. }\left\vert a_{k,j}\right\vert ^{2} &\leq &\left( a_{j,j}-\beta
a_{j,j}\right) \left( 1-\beta \right) a_{k,k}=\left( 1-\beta \right)
^{2}a_{j,j}a_{k,k}\ ,
\end{eqnarray*}%
which is (\ref{2 by 2}).
\end{proof}

\subsection{The subordinaticity theorem}

There is a second concept that will play a significant role in
hypoellipticity theorems, and which we now introduce.

\begin{definition}
A symmetric matrix function $A\left( x\right) $ defined for $x\in \mathbb{R}%
^{M}$ is said to be \emph{subordinate} if there is $\Gamma >0$ such that for
every first order partial derivative $\frac{\partial }{\partial x_{k}}$, the
matrix $\frac{\partial }{\partial x_{k}}A\left( x\right) $ exists and
satisfies%
\begin{equation}
\left\vert \frac{\partial }{\partial x_{k}}A\left( x\right) \xi \right\vert
^{2}\leq \Gamma ^{2}\xi ^{\limfunc{tr}}A\left( x\right) \xi ,\ \ \ \ \ 1\leq
k\leq M.  \label{subordinate}
\end{equation}
\end{definition}

The inequality (\ref{subordinate}) holds for all smooth nonnegative \emph{%
diagonal} matrices $D\left( x\right) $, by the classical Malgrange
inequality for $C^{2}$ scalar functions. However, we note that the
subordination property (\ref{subordinate}) \textbf{fails} miserably for
nondiagonal matrices $A\left( x\right) $ in general, even for $2\times 2$
matrices in one variable $x\in \mathbb{R}$ that are comparable to a smooth
diagonally elliptical matrix $A\left( x\right) $ that is a sum of squares of
smooth vector fields. For example 
\begin{equation*}
A\left( x\right) =\left[ 
\begin{array}{cc}
1 & \gamma f\left( x\right) \\ 
\gamma f\left( x\right) & f\left( x\right) ^{2}%
\end{array}%
\right] ,\ \ \ \ \ \text{where }0<\left\vert \gamma \right\vert <1,
\end{equation*}%
fails to be subordinate since%
\begin{equation*}
\left\vert f^{\prime }\left( x\right) \left[ 
\begin{array}{cc}
0 & \gamma \\ 
\gamma & 2f\left( x\right)%
\end{array}%
\right] \left( 
\begin{array}{c}
\xi _{1} \\ 
\xi _{2}%
\end{array}%
\right) \right\vert ^{2}\leq C\left( 
\begin{array}{cc}
\xi _{1} & \xi _{2}%
\end{array}%
\right) \left[ 
\begin{array}{cc}
1 & \gamma f\left( x\right) \\ 
\gamma f\left( x\right) & f\left( x\right) ^{2}%
\end{array}%
\right] \left( 
\begin{array}{c}
\xi _{1} \\ 
\xi _{2}%
\end{array}%
\right) ,
\end{equation*}%
implies with $\left( 
\begin{array}{c}
\xi _{1} \\ 
\xi _{2}%
\end{array}%
\right) =\left( 
\begin{array}{c}
0 \\ 
1%
\end{array}%
\right) $ that $f^{\prime }\left( x\right) ^{2}\left\vert \left( 
\begin{array}{c}
\gamma \\ 
2f\left( x\right)%
\end{array}%
\right) \right\vert ^{2}\leq Cf\left( x\right) ^{2}$, and hence $\left\vert 
\frac{d}{dx}\ln f\left( x\right) \right\vert ^{2}=\frac{f^{\prime }\left(
x\right) ^{2}}{f\left( x\right) ^{2}}\leq \frac{C}{\gamma ^{2}}$. But $\ln
f\left( x\right) $ cannot be bounded near $0$ for any smooth nonnegative $%
f\left( x\right) $ that vanishes at $0$. However, in the case $f$ is smooth,
then $A\left( x\right) $ is a sum of smooth squares%
\begin{equation*}
A\left( x\right) =\left( 
\begin{array}{c}
1 \\ 
\gamma f\left( x_{1}\right)%
\end{array}%
\right) \otimes \left( 
\begin{array}{c}
1 \\ 
\gamma f\left( x_{1}\right)%
\end{array}%
\right) +\left( 
\begin{array}{c}
0 \\ 
\sqrt{1-\gamma ^{2}}f\left( x_{1}\right)%
\end{array}%
\right) \otimes \left( 
\begin{array}{c}
0 \\ 
\sqrt{1-\gamma ^{2}}f\left( x_{1}\right)%
\end{array}%
\right) ,
\end{equation*}%
illustrating the fact that a sum of smooth squares need not be subordinate.
Conversely, in the final section of this paper, we give an example in
Theorem \ref{C 1 delta} of a diagonally elliptical $3\times 3$ matrix
function $\mathbf{Q}\left( x,y,z\right) $ of three variables that is
subordinate, but not a finite sum of squares of $C^{1,\delta }$ vector
fields for any $\delta >0$.

\begin{definition}
A matrix function $\mathbf{M}$ is said to be $\mathcal{SOS}_{k,\delta }$ if
it can be written as a finite sum of squares of $C^{k,\delta }$ vector
fields.
\end{definition}

\begin{conclusion}
The property that a matrix function $\mathbf{M}$ is $\mathcal{SOS}_{1,\delta
}$ is in general \emph{incomparable} with the property that $\mathbf{M}$ is
subordinate.
\end{conclusion}

\begin{theorem}
\label{sub char}If $A=\left[ a_{ij}\right] _{i,j=1}^{n}$ is an $n\times n$
positive semidefinite $C^{2}$ matrix function that is comparable to a
diagonal matrix function, then $A$ is subordinate \emph{if and only if}%
\begin{equation}
\left\vert \nabla a_{ij}\right\vert ^{2}\lesssim \min \left\{
a_{ii},a_{jj}\right\} ,\ \ \ \ \ 1\leq i,j\leq n.  \label{sub cond n}
\end{equation}%
In particular, a matrix function $A=\left[ a_{ij}\right] _{i,j=1}^{n}$ with $%
a_{ii}\approx \lambda $ for all $1\leq i\leq n$, is subordinate \emph{if and
only if}%
\begin{equation*}
\left\vert \nabla a_{ij}\right\vert ^{2}\lesssim \lambda ,\ \ \ \ \ 1\leq
i,j\leq n.
\end{equation*}%
Note that this latter set of inequalities amount to assuming an extension of
Malgrange's classical inequality $\left\vert \nabla a_{ii}\right\vert
^{2}\lesssim a_{ii}\,\ $(which requires $a_{ii}\in C^{2}$) to the off
diagonal case $i\neq j$.
\end{theorem}

\begin{proof}
Let $f^{\prime }$ denote any of the partial derivatives $\frac{\partial f}{%
\partial x_{k}}$. We begin by noting that $A\left( x\right) $ is subordinate
if and only if%
\begin{eqnarray*}
&&\sum_{i=1}^{n}\left\vert \sum_{j=1}^{n}a_{i,j}^{\prime }\left( x\right)
\xi _{j}\right\vert ^{2}=\left\vert \left[ 
\begin{array}{ccc}
a_{11}^{\prime }\left( x\right) & \cdots & a_{1n}^{\prime }\left( x\right)
\\ 
\vdots & \ddots & \vdots \\ 
a_{n1}^{\prime }\left( x\right) & \cdots & a_{nn}^{\prime }\left( x\right)%
\end{array}%
\right] \left( 
\begin{array}{c}
\xi _{1} \\ 
\vdots \\ 
\xi _{3}%
\end{array}%
\right) \right\vert ^{2} \\
&=&\left\vert A^{\prime }\left( x\right) \mathbf{\xi }\right\vert
^{2}\lesssim \mathbf{\xi }^{\func{tr}}A\left( x\right) \mathbf{\xi }\approx 
\mathbf{\xi }^{\func{tr}}\left[ 
\begin{array}{ccc}
a_{11}\left( x\right) & 0 & 0 \\ 
0 & \ddots & 0 \\ 
0 & 0 & a_{nn}\left( x\right)%
\end{array}%
\right] \mathbf{\xi }=\sum_{i=1}^{n}a_{i,i}\left( x\right) \xi _{i}^{2}.
\end{eqnarray*}%
Taking $\mathbf{\xi }=\mathbf{e}_{1},\mathbf{...},\mathbf{e}_{n}$
respectively shows that the following $n$ conditions are necessary and
sufficient;%
\begin{eqnarray*}
\left\vert a_{1,1}^{\prime }\left( x\right) \right\vert ^{2}+\sum_{j:\ j\neq
1}\left\vert a_{1,j}^{\prime }\left( x\right) \right\vert ^{2} &\lesssim
&a_{1,1}\left( x\right) , \\
\left\vert a_{2,2}^{\prime }\left( x\right) \right\vert ^{2}+\sum_{j:\ j\neq
2}\left\vert a_{2,j}^{\prime }\left( x\right) \right\vert ^{2} &\lesssim
&a_{2,2}\left( x\right) , \\
&&\vdots \\
\left\vert a_{n-1,n-1}^{\prime }\left( x\right) \right\vert ^{2}+\sum_{j:\
j\neq n-1}\left\vert a_{n-1,j}^{\prime }\left( x\right) \right\vert ^{2}
&\lesssim &a_{n,n}\left( x\right) , \\
\left\vert a_{n,n}^{\prime }\left( x\right) \right\vert ^{2}+\sum_{j:\ j\neq
n}\left\vert a_{n,j}^{\prime }\left( x\right) \right\vert ^{2} &\lesssim
&a_{n,n}\left( x\right) .
\end{eqnarray*}%
The inequality of Malgrange already shows that the diagonal entries satisfy $%
\left\vert a_{i,i}^{\prime }\left( x\right) \right\vert ^{2}\lesssim
a_{i,i}\left( x\right) $, and since the derivative $\nabla a_{i,j}\left(
x\right) $ of an off diagonal entry occurs in only two of the above lines,
namely in the $i^{th}$ and $j^{th}$ lines, we see that the conditions in the
display above hold if and only if (\ref{sub cond n}) holds.
\end{proof}

\subsection{The sum of squares theorem}

Here is the sum of squares decomposition for a Grushin type matrix function
with a quasiformal block of order $\left( n-p+1\right) \times \left(
n-p+1\right) $, where $1<p\leq n$.

\begin{definition}
Suppose that $\mathbf{A}\left( x\right) $ is a diagonally elliptical $%
n\times n$ matrix function on $\mathbb{R}^{n}$.

\begin{enumerate}
\item We say that $\mathbf{A}\left( x\right) $ is quasiconformal if the
eigenvalues $\lambda _{i}\left( x\right) $ of $\mathbf{A}\left( x\right) $
are nonnegative and comparable.

\item We say that $\mathbf{A}\left( x\right) $ is of Grushin type if $%
\mathbf{A}\left( x\right) $ is singular exactly on a vector subspace $%
\Lambda $ of $\mathbb{R}^{n}$ having dimension $m<n$, and if the diagonal
entries $a_{k,k}\left( x\right) $ of $\mathbf{A}\left( x\right) $ are each
comparable to a function $\lambda _{k}\left( y\right) $ depending only on $%
y\in \Lambda $.
\end{enumerate}
\end{definition}

Roughly speaking, the next theorem says in particular that if a subordinate
diagonally elliptical $C^{4,2\delta }$ matrix function $\mathbf{A}\left(
x\right) $ of Grushin type has diagonal entries $a_{k,k}\left( x\right) $
that are comparable to the last entry $a_{n,n}$ for $p\leq k\leq n$ (but
unrelated to the first $p-1$ diagonal entries), and finally if the
off-diagonal entries satisfy suitable subordinate type inequalities, then $%
\mathbf{A}$ is a finite sum of squares of $C^{2,\delta }$ vector fields plus
a $C^{4,2\delta }$ block matrix function $\left[ 
\begin{array}{cc}
\mathbf{0} & \mathbf{0} \\ 
\mathbf{0} & \mathbf{Q}_{p}%
\end{array}%
\right] $ where the $\left( n-p+1\right) $-square matrix $\mathbf{Q}_{p}$ is
both subordinate and quasiconformal.

\begin{theorem}
\label{final n Grushin}Let 
\begin{equation*}
1<p\leq n,\ \ \ \frac{1}{4}\leq \varepsilon <1,\ \ \ 0<\delta <\delta
^{\prime },\delta ^{\prime \prime }<1,\ \ \ M\geq 1,
\end{equation*}%
with%
\begin{equation*}
\delta ^{\prime }=\frac{2\delta \left( 1+\delta \right) }{2+\delta }.
\end{equation*}%
Suppose that $\mathbf{A}\left( x\right) $ is a $C^{4,2\delta }$ symmetric $%
n\times n$ matrix function of a variable $x\in \mathbb{R}^{M}$, which is
comparable to a diagonal matrix function $\mathbf{D}\left( x\right) $, hence
comparable to its associated diagonal matrix function $\mathbf{A}_{\func{diag%
}}\left( x\right) $.

\begin{enumerate}
\item Moreover, assume $a_{p,p}\left( x\right) \approx a_{p+1,p+1}\left(
x\right) \approx ...\approx a_{n,n}\left( x\right) $ and that the diagonal
entries $a_{1,1}\left( x\right) ,...,a_{p-1,p-1}\left( x\right) $ satisfy
the following differential estimates up to fourth order,%
\begin{eqnarray}
\left\vert D^{\mu }a_{k,k}\left( x\right) \right\vert &\lesssim
&a_{k,k}\left( x\right) ^{\left[ 1-\left\vert \mu \right\vert \varepsilon %
\right] _{+}+\delta ^{\prime }},\ \ \ \ \ \text{ }1\leq \left\vert \mu
\right\vert \leq 4\text{ and }1\leq k\leq p-1,  \label{diag hyp} \\
\left[ a_{k,k}\right] _{\mu ,2\delta }\left( x\right) &\lesssim &1,\ \ \ \ \ 
\text{ }\left\vert \mu \right\vert =4\text{ and }1\leq k\leq p-1.  \notag
\end{eqnarray}

\item Furthermore, assume the off diagonal entries $a_{k,j}\left( x\right) $
satisfy the following differential estimates up to fourth order, 
\begin{eqnarray}
\left\vert D^{\mu }a_{k,j}\right\vert &\lesssim &\left( \min_{1\leq s\leq
j}a_{s,s}\right) ^{\left[ \frac{1}{2}+\left( 2-\left\vert \mu \right\vert
\right) \varepsilon \right] _{+}+\delta ^{\prime \prime }},\ \ \ \ \ 0\leq
\left\vert \mu \right\vert \leq 4\text{ and }1\leq k<j\leq p-1,
\label{off diag hyp} \\
\left[ a_{k,j}\right] _{\mu ,2\delta } &\lesssim &1,\ \ \ \ \ \left\vert \mu
\right\vert =4\text{ and }1\leq k<j\leq p-1,  \notag \\
\left\vert D^{\mu }a_{k,j}\right\vert &\lesssim &\left( \min_{1\leq s\leq
k}a_{s,s}\right) ^{\left[ \frac{1}{2}+\left( 2-\left\vert \mu \right\vert
\right) \varepsilon \right] _{+}+\delta ^{\prime \prime }},\ \ \ \ \ 0\leq
\left\vert \mu \right\vert \leq 4\text{ and }1\leq k\leq p-1<j\leq n  \notag
\\
\left[ a_{k,j}\right] _{\mu ,2\delta } &\lesssim &1,\ \ \ \ \ \left\vert \mu
\right\vert =4\text{ and }1\leq k\leq p-1<j\leq n.  \notag
\end{eqnarray}

\item Then there is a positive integer $I\in \mathbb{N}$ such that the
matrix function $\mathbf{A}$ can be written as a finite sum of squares of $%
C^{2,\delta }$ vectors $X_{k,j}$, plus a matrix function $\mathbf{A}_{p}$, 
\begin{equation*}
\mathbf{A}\left( x\right) =\sum_{k=1}^{p-1}\sum_{i=1}^{I}X_{k,j}\left(
x\right) X_{k,j}\left( x\right) ^{\func{tr}}+\mathbf{A}_{p}\left( x\right)
,\ \ \ \ \ x\in \mathbb{R}^{M},
\end{equation*}%
where the vectors $X_{k,i}\left( x\right) ,\ 1\leq k\leq p-1,\ 1\leq i\leq I$
are $C^{2,\delta }\left( \mathbb{R}^{M}\right) $, $\mathbf{A}_{p}\left(
x\right) =\left[ 
\begin{array}{cc}
\mathbf{0} & \mathbf{0} \\ 
\mathbf{0} & \mathbf{Q}_{p}\left( x\right)%
\end{array}%
\right] $, and $\mathbf{Q}_{p}\left( x\right) \in C^{4,2\delta }\left( 
\mathbb{R}^{M}\right) $ is \emph{quasiconformal}. Moreover, $Z_{k}\equiv
\sum_{i=1}^{I}X_{k,i}X_{k,i}^{\func{tr}}\in C^{4,2\delta }\left( \mathbb{R}%
^{M}\right) $ and 
\begin{eqnarray}
ca_{k,k}\mathbf{e}_{k}\otimes \mathbf{e}_{k} &\prec &Z_{k}Z_{k}^{\func{tr}%
}+\sum_{m=k+1}^{n}a_{m,m}\mathbf{e}_{m}\otimes \mathbf{e}_{m}\prec
C\sum_{m=k}^{n}a_{m,m}\mathbf{e}_{m}\otimes \mathbf{e}_{m},\ \ \ \ \ 1\leq
k\leq p-1,  \label{more} \\
\mathbf{Q}_{p}\left( x\right) &\sim &a_{p,p}\left( x\right) \mathbb{I}%
_{n-p+1}\ .  \notag
\end{eqnarray}%
Finally, if in addition $A\left( x\right) $ is subordinate, then $\mathbf{Q}%
_{p}\left( x\right) $ is also subordinate.
\end{enumerate}
\end{theorem}

\begin{remark}
If in addition $a_{k,k}\left( x\right) \approx 1$ for $1\leq k\leq m<p$,
then the conditions (\ref{diag hyp}) and (\ref{off diag hyp}) in $\left(
1\right) $ and $\left( 2\right) $ are vacuous for $1\leq k\leq m$, and
moreover the proof shows that the vectors $X_{k,i}$ are actually in $%
C^{4,2\delta }\left( \mathbb{R}^{M}\right) $ for $1\leq k\leq m,1\leq i\leq
I $.
\end{remark}

This remark yields the following corollary in which conditions (\ref{diag
hyp}) and (\ref{off diag hyp}) in $\left( 1\right) $ and $\left( 2\right) $
play no role.

\begin{corollary}
Suppose $\mathbf{A}\left( x\right) $ is a $C^{4,2\delta }\left( \mathbb{R}%
^{M}\right) $ symmetric $n\times n$ matrix function that is comparable to a
diagonal matrix function. In addition suppose that $a_{k,k}\left( x\right)
\approx 1$ for $1\leq k\leq p-1$ and $a_{k,k}\left( x\right) \approx
a_{p,p}\left( x\right) $ for $p\leq k\leq n$. Then%
\begin{equation*}
\mathbf{A}\left( x\right) =\sum_{k=1}^{p-1}X_{k}\left( x\right) X_{k}\left(
x\right) ^{\func{tr}}+\mathbf{Q}_{p}\left( x\right) ,\ \ \ \ \ x\in \mathbb{R%
}^{M},
\end{equation*}%
where $X_{k},\mathbf{Q}_{p}\in C^{4,2\delta }\left( \mathbb{R}^{M}\right) $
and (\ref{more}) holds for $1\leq k\leq p-1$.
\end{corollary}

\begin{remark}
If the diagonal entry $a_{k,k}\left( x\right) $ is smooth and $\omega _{s}$%
-montone on $\mathbb{R}^{n}$ for some $s>1-\varepsilon $, then the diagonal
differential estimates (\ref{diag hyp}) above hold for $a_{k,k}\left(
x\right) $ since by \cite[Theorem 18]{KoSa1} we have $\left\vert D^{\mu
}a_{k,k}\left( x\right) \right\vert \leq C_{s,s^{\prime }}a_{k,k}\left(
x\right) ^{\left( s^{\prime }\right) ^{\left\vert \mu \right\vert }}$ for
any $0<s^{\prime }<s$. Indeed, we then have 
\begin{equation*}
a_{k,k}\left( x\right) ^{\left( s^{\prime }\right) ^{\left\vert \mu
\right\vert }}\lesssim a_{k,k}\left( x\right) ^{\left[ 1-\left\vert \mu
\right\vert \varepsilon \right] _{+}+\delta ^{\prime }},
\end{equation*}%
since $\left( s^{\prime }\right) ^{\left\vert \mu \right\vert }-\left[
1-\left\vert \mu \right\vert \varepsilon \right] _{+}>0$ for $1-\varepsilon
<s^{\prime }<s$, which in turn follows from the fact that $\left(
1-\varepsilon \right) ^{m}+m\varepsilon $ has a strict minimum at $%
\varepsilon =0$.
\end{remark}

\begin{remark}
As $\varepsilon $ increases from $\frac{1}{4}$ to $1$, the diagonal
assumptions (\ref{diag hyp}) become more relaxed, while the off diagonal
assumptions (\ref{off diag hyp}) become more stringent. Thus there is a
tradeoff between assuming more on the diagonal entries or more on the off
diagonal entries.
\end{remark}

\begin{remark}
The positive number $\delta ^{\prime }$ in Theorem \ref{final n Grushin} was
chosen in order to use Theorem \ref{efs eps} at various critical points in
the proof.
\end{remark}

\begin{remark}
Examples in the final section show that we cannot lower the exponent $\left[ 
\frac{1}{2}+\left( 2-\left\vert \mu \right\vert \right) \varepsilon \right]
_{+}$ on the right hand side of (\ref{off diag hyp}).
\end{remark}

\begin{remark}
If in Theorem \ref{final n Grushin}, we drop the hypothesis (\ref{diag hyp})
that the diagonal entries satisfy the differential estimates, and even
slightly weaken the off diagonal hypotheses (\ref{off diag hyp}), then using
the Fefferman-Phong theorem for sums of squares of scalar functions, the
proof of Theorem \ref{final n Grushin} shows that the operator $L=\nabla ^{%
\limfunc{tr}}\mathbf{A}\nabla $ can be written as $L=\sum_{j=1}^{N}X_{j}^{%
\func{tr}}X_{j}$ where the vector fields $X_{j}$ are $C^{1,1}$ for$\
j=1,2,...,N$. However, unlike the situation for scalar functions, the
examples in Theorems \ref{C 1 delta} and \ref{flat counter}\ show that we 
\emph{cannot} dispense with the off diagonal hypotheses (\ref{off diag hyp})
in $\left( 1\right) \left( b\right) $. Moreover, the space $C^{1,1}$ seems
to not be sufficient for gaining a positive degree $\delta $ of smoothness
for solutions to the second order operators we consider, and so this result
will neither be used nor proved here.
\end{remark}

\section{Square Decompositions}

Suppose $\mathbf{A}=\left[ a_{k,j}\right] _{k,j=1}^{n}$ is a symmetric $%
n\times n$ matrix with top left entry $a_{11}>0$. Then we can uniquely
decompose $\mathbf{A}$ into a sum of a positive dyad $ZZ^{\func{tr}}$ and a
matrix $\mathbf{B}$ with zeroes in the first column and first row, namely%
\begin{eqnarray}
\mathbf{A} &\mathbf{=}&ZZ^{\func{tr}}+\mathbf{B}=\left[ \left(
\sum_{j=1}^{n}s_{j}\mathbf{e}_{j}\right) \otimes \left( \sum_{j=1}^{n}s_{j}%
\mathbf{e}_{j}\right) \right] +\mathbf{B},  \label{can dec} \\
\mathbf{B} &=&\left[ 
\begin{array}{cc}
0 & \mathbf{0}_{1\times \left( n-1\right) } \\ 
\mathbf{0}_{\left( n-1\right) \times 1} & \mathbf{Q}%
\end{array}%
\right] ,  \notag
\end{eqnarray}%
where $Z$ is $\frac{1}{\sqrt{a_{11}}}$ times the first column of $\mathbf{A}$%
, i.e. for $1\leq k\leq j\leq n$,%
\begin{eqnarray*}
a_{11} &=&\left( s_{1}\right) ^{2}\text{ and }a_{1j}=s_{1}s_{j}, \\
s_{1} &=&\sqrt{a_{11}}\text{ and }s_{j}=\frac{a_{1j}}{s_{1}}, \\
\mathbf{Q} &=&\left[ q_{k,j}\right] _{j,k=2}^{n}=\left[ a_{k,j}-s_{k}s_{j}%
\right] _{j,k=2}^{n}=\left[ a_{k,j}-\frac{a_{1k}a_{1j}}{a_{11}}\right]
_{j,k=2}^{n}.
\end{eqnarray*}%
As this particular decomposition $\mathbf{A}\mathbf{=}ZZ^{\func{tr}}+\mathbf{%
B}$ will play a pivotal role in our inductive proof of Theorem \ref{final n
Grushin}, we give it a name.

\begin{definition}
The above decomposition $\mathbf{A}\mathbf{=}ZZ^{\func{tr}}+\mathbf{B}$ is
called the $1$\emph{-Square Decomposition} of $A$.
\end{definition}

Note that we are writing the $\left( n-1\right) \times \left( n-1\right) $
matrix $\mathbf{Q}$ as $\left[ q_{k,j}\right] _{k,j=2}^{n}$, where the rows
of $\mathbf{Q}$ are parameterized by $k$ and the columns by $j$ with $2\leq
k,j\leq n$.

Another main ingredient in our inductive proof is the well-known
characterization of positive definiteness given in Theorem \ref{well known}
above. We will also need the well-known determinant formula given in Lemma %
\ref{deter} above,%
\begin{equation*}
\det \left[ 
\begin{array}{cc}
\alpha & v^{\limfunc{tr}} \\ 
v & \mathbf{M}%
\end{array}%
\right] =\alpha \det \mathbf{M}-v^{\limfunc{tr}}\limfunc{co}\mathbf{M}%
v=\left\{ \alpha -v^{\limfunc{tr}}\mathbf{M}^{-1}v\right\} \det \mathbf{M\ .}
\end{equation*}

We have already introduced the concepts of diagonally elliptical and
subordinate matrix functions above, and we now introduce one more concept
relevant to our decomposition, especially to the regularity of the vectors
in the sum of squares. More precisely, we will need to show that the
hypotheses of Theorem \ref{final n Grushin} propogate through the $1$-Square
Decomposition in an appropriate sense. The following definition encodes what
is required.

\begin{definition}
\label{def l strongly}Suppose $\mathbf{A}\left( x\right) \in C^{4,2\delta
}\left( \mathbb{R}^{M}\right) $ is an $n\times n$ matrix function, $1\leq
\ell \leq n$, $\frac{1}{4}\leq \varepsilon <1$ and $0<\delta ^{\prime
},\delta ^{\prime \prime }<1$. Set%
\begin{equation*}
m_{k}\left( x\right) \equiv \min_{1\leq s\leq k}\left\{ a_{s,s}\left(
x\right) \right\} ,\ \ \ \ \ 1\leq k\leq n.
\end{equation*}%
We say $\mathbf{A}\left( x\right) $ is $\left( \ell ,\varepsilon ,\delta
^{\prime },\delta ^{\prime \prime }\right) $\emph{-strongly }$C^{4,2\delta }$
if 
\begin{eqnarray}
\left\vert D^{\mu }a_{k,k}\right\vert &\lesssim &\left\vert
a_{k,k}\right\vert ^{\left[ 1-\left\vert \mu \right\vert \varepsilon \right]
_{+}+\delta ^{\prime }},\ \ \ \ \ 1\leq \left\vert \mu \right\vert \leq 4%
\text{ and }1\leq k\leq \ell ,  \label{akj hyp} \\
\left[ a_{k,k}\right] _{\mu ,2\delta } &\lesssim &1,\ \ \ \ \ \left\vert \mu
\right\vert =4\text{ and }1\leq k\leq \ell ,  \notag \\
\left\vert D^{\mu }a_{k,j}\right\vert &\lesssim &\left( m_{j}\right) ^{ 
\left[ \frac{1}{2}+\left( 2-\left\vert \mu \right\vert \right) \varepsilon %
\right] _{+}+\delta ^{\prime \prime }},\ \ \ \ \ 0\leq \left\vert \mu
\right\vert \leq 4\text{ and }1\leq k<j\leq \ell ,  \notag \\
\left[ a_{k,j}\right] _{\mu ,2\delta } &\lesssim &1,\ \ \ \ \ \left\vert \mu
\right\vert =4\text{ and }1\leq k<j\leq \ell ,  \notag \\
\left\vert D^{\mu }a_{k,j}\right\vert &\lesssim &\left( m_{k}\right) ^{ 
\left[ \frac{1}{2}+\left( 2-\left\vert \mu \right\vert \right) \varepsilon %
\right] _{+}+\delta ^{\prime \prime }},\ \ \ \ \ 0\leq \left\vert \mu
\right\vert \leq 4\text{ and }1\leq k\leq \ell <j\leq n,  \notag \\
\left[ a_{k,j}\right] _{\mu ,2\delta } &\lesssim &1,\ \ \ \ \ \left\vert \mu
\right\vert =4\text{ and }1\leq k\leq \ell <j\leq n.  \notag
\end{eqnarray}
\end{definition}

\begin{remark}
When using this definition in the course of proving Theorem \ref{final n
Grushin}, the assumption that $\delta ^{\prime }>0$ in the first line of (%
\ref{akj hyp}) will be essential in order to apply our result on sums of
squares of $C^{2,\delta }$ functions. The choices $\delta ^{\prime }=\frac{%
2\delta \left( 1+\delta \right) }{2+\delta }$ and $\delta ^{\prime \prime
}>\delta $ made in Theorem \ref{final n Grushin} will be required to show
that the \emph{matrices} being squared are $C^{2,\delta }$.
\end{remark}

In the scalar case $n=1$, the requirements (\ref{akj hyp}) on the scalar
function $a_{1,1}$ are equivalent to the inequalities 
\begin{equation*}
\left\vert D^{\mu }a_{1,1}\right\vert \lesssim \left\vert a_{1,1}\right\vert
^{\delta ^{\prime }},\ \ \ \ \ \left\vert \mu \right\vert =2,4,
\end{equation*}%
by the control of odd derivatives by even derivatives, see e.g. \cite{Tat}
and \cite[Lemma 24]{KoSa1}. In turn, these latter inequalities hold for some 
$\delta ^{\prime }>0$ if and only if both $\left\vert \nabla
^{4}a_{1,1}\right\vert \leq Ca_{1,1}^{\frac{\delta }{2+\delta }}$ and $%
\left\vert \nabla ^{2}a_{1,1}\right\vert \leq Ca_{1,1}^{\frac{2\delta \left(
1+\delta \right) }{2+\delta }}$ hold for some $\delta >0$, and then Theorem %
\ref{efs eps} applies to show that $a_{1,1}$ is a finite sum of squares of $%
C^{2,\delta }$ functions.

We will now show in the next three lemmas that the properties of being
diagonally elliptical, subordinate, and $\left( \ell ,\varepsilon ,\delta
^{\prime },\delta ^{\prime \prime }\right) $-strongly\emph{\ }$C^{4,2\delta }
$, are inherited by the $\left( n-1\right) \times \left( n-1\right) $ matrix 
$\mathbf{Q}\left( x\right) $ in the $1$-Square Decomposition of $\mathbf{A}%
\left( x\right) $, with the proviso that $\ell $ is decreased by $1$ in the
definition of strongly $C^{4,2\delta }$.

\subsection{The three lemmas}

\begin{lemma}
\label{diag ellip}Suppose that $\mathbf{A}\left( x\right) =\left[
a_{k,j}\left( x\right) \right] _{k,j=1}^{n}=\left[ 
\begin{array}{cc}
a_{11}\left( x\right) & b\left( x\right) ^{\func{tr}} \\ 
b\left( x\right) & \mathbf{D}\left( x\right)%
\end{array}%
\right] $ is a diagonally elliptical $n\times n$ matrix function with $1$%
-Square Decomposition 
\begin{equation*}
\mathbf{A}\left( x\right) \mathbf{=}Z\left( x\right) Z\left( x\right) ^{%
\func{tr}}+\mathbf{B}\left( x\right) ,\ \ \ \ \ \text{where }\mathbf{B}%
\left( x\right) =\left[ 
\begin{array}{cc}
0 & 0^{\func{tr}} \\ 
0 & \mathbf{Q}\left( x\right)%
\end{array}%
\right] .
\end{equation*}%
Then $\mathbf{Q}\left( x\right) =\left[ q_{k,j}\left( x\right) \right]
_{k,j=1}^{n}$ is a diagonally elliptical $\left( n-1\right) \times \left(
n-1\right) $ matrix function, and moreover%
\begin{equation}
q_{i,i}\left( x\right) \approx a_{i,i}\left( x\right) ,\ \ \ \ \ 2\leq i\leq
n,  \label{q-a}
\end{equation}%
and%
\begin{eqnarray}
Z_{1}\left( x\right) Z_{1}\left( x\right) ^{\func{tr}} &\prec &C\mathbf{A}%
\left( x\right) ,  \label{Z1 comp} \\
ca_{1,1}\mathbf{e}_{1}\otimes \mathbf{e}_{1} &\prec &Z_{1}Z_{1}^{\func{tr}%
}+\sum_{k=2}^{n}a_{k,k}\mathbf{e}_{k}\otimes \mathbf{e}_{k}\prec C\mathbf{A}.
\notag
\end{eqnarray}
\end{lemma}

\begin{proof}
Note that 
\begin{equation*}
Z=\left( 
\begin{array}{c}
\sqrt{a_{1,1}} \\ 
\frac{1}{\sqrt{a_{1,1}}}a_{1,2} \\ 
\vdots \\ 
\frac{1}{\sqrt{a_{1,1}}}a_{1,n}%
\end{array}%
\right) =\left( 
\begin{array}{c}
\sqrt{a_{1,1}} \\ 
\frac{b_{2}}{\sqrt{a_{1,1}}} \\ 
\vdots \\ 
\frac{b_{n}}{\sqrt{a_{1,1}}}%
\end{array}%
\right) =\left( 
\begin{array}{c}
\sqrt{a_{1,1}} \\ 
\frac{b}{\sqrt{a_{1,1}}}%
\end{array}%
\right) .
\end{equation*}%
Since $\mathbf{A}\left( x\right) $ is comparable to $\mathbf{A}_{\func{diag}%
}\left( x\right) $, say%
\begin{equation*}
\beta _{0}\mathbf{A}_{\func{diag}}\preccurlyeq \mathbf{A}\preccurlyeq \alpha
_{0}\mathbf{A}_{\func{diag}},
\end{equation*}%
we also have%
\begin{equation}
\beta _{0}\mathbf{D}_{\func{diag}}\preccurlyeq \mathbf{D}\preccurlyeq \alpha
_{0}\mathbf{D}_{\func{diag}}  \label{also have}
\end{equation}%
and Theorem \ref{comp mat} now shows that for $0<\beta \leq \beta _{0}$ we
have%
\begin{equation*}
b^{\func{tr}}\mathbf{D}_{\beta }b\leq \left( 1-\beta \right) a_{11}\ ,
\end{equation*}%
where $\mathbf{D}_{\beta }\equiv \left[ \mathbf{D}-\beta \mathbf{D}_{\func{%
diag}}\right] ^{-1}$.

Since $\mathbf{D}_{\beta }\left( x\right) $ is positive semidefinite, it has
a positive semidefinite square root $\mathbf{D}_{\beta }\left( x\right) ^{%
\frac{1}{2}}$, and so we have%
\begin{equation*}
\left( \mathbf{D}_{\beta }\left( x\right) ^{\frac{1}{2}}b\left( x\right)
\right) ^{\func{tr}}\mathbf{D}_{\beta }\left( x\right) ^{\frac{1}{2}}b\left(
x\right) =b\left( x\right) ^{\func{tr}}\mathbf{D}_{\beta }\left( x\right)
b\left( x\right) \leq \left( 1-\beta \right) a_{11}\left( x\right) \ ,
\end{equation*}%
which implies that%
\begin{equation*}
\left\vert \mathbf{D}_{\beta }\left( x\right) ^{\frac{1}{2}}b\left( x\right)
\right\vert \leq \sqrt{1-\beta }\sqrt{a_{11}\left( x\right) }.
\end{equation*}

But then we have that%
\begin{eqnarray*}
&&\mathbf{D}_{\beta }\left( x\right) ^{\frac{1}{2}}\mathbf{Q}\left( x\right) 
\mathbf{D}_{\beta }\left( x\right) ^{\frac{1}{2}}=\mathbf{D}_{\beta }\left(
x\right) ^{\frac{1}{2}}\left[ \mathbf{D}\left( x\right) -\frac{1}{%
a_{11}\left( x\right) }b\left( x\right) b\left( x\right) ^{\func{tr}}\right] 
\mathbf{D}_{\beta }\left( x\right) ^{\frac{1}{2}} \\
&=&\mathbf{D}_{\beta }\left( x\right) ^{\frac{1}{2}}\mathbf{D}\left(
x\right) \mathbf{D}_{\beta }\left( x\right) ^{\frac{1}{2}}-\left( \frac{1}{%
\sqrt{a_{11}\left( x\right) }}\mathbf{D}_{\beta }\left( x\right) ^{\frac{1}{2%
}}b\left( x\right) \right) \left( \frac{1}{\sqrt{a_{11}\left( x\right) }}%
\mathbf{D}_{\beta }\left( x\right) ^{\frac{1}{2}}b\left( x\right) \right) ^{%
\func{tr}}
\end{eqnarray*}%
where%
\begin{eqnarray*}
\mathbf{D}_{\beta }\left( x\right) ^{\frac{1}{2}}\mathbf{D}\left( x\right) 
\mathbf{D}_{\beta }\left( x\right) ^{\frac{1}{2}} &=&\mathbf{D}_{\beta
}\left( x\right) ^{\frac{1}{2}}\left[ \mathbf{D}_{\beta }\left( x\right)
^{-1}+\beta \mathbf{D}_{\func{diag}}\left( x\right) \right] \mathbf{D}%
_{\beta }\left( x\right) ^{\frac{1}{2}} \\
&=&\mathbb{I}_{n-1}+\beta \mathbf{D}_{\beta }\left( x\right) ^{\frac{1}{2}}%
\mathbf{D}_{\func{diag}}\left( x\right) \mathbf{D}_{\beta }\left( x\right) ^{%
\frac{1}{2}},
\end{eqnarray*}%
and hence%
\begin{eqnarray*}
\mathbf{D}_{\beta }\left( x\right) ^{\frac{1}{2}}\mathbf{Q}\left( x\right) 
\mathbf{D}_{\beta }\left( x\right) ^{\frac{1}{2}} &=&\mathbb{I}_{n-1}+\beta 
\mathbf{D}_{\beta }\left( x\right) ^{\frac{1}{2}}\mathbf{D}_{\func{diag}%
}\left( x\right) \mathbf{D}_{\beta }\left( x\right) ^{\frac{1}{2}} \\
&&-\left( \frac{1}{\sqrt{a_{11}\left( x\right) }}\mathbf{D}_{\beta }\left(
x\right) ^{\frac{1}{2}}b\left( x\right) \right) \left( \frac{1}{\sqrt{%
a_{11}\left( x\right) }}\mathbf{D}_{\beta }\left( x\right) ^{\frac{1}{2}%
}b\left( x\right) \right) ^{\func{tr}}.
\end{eqnarray*}%
Thus $\mathbf{D}_{\beta }\left( x\right) ^{\frac{1}{2}}\mathbf{Q}\left(
x\right) \mathbf{D}_{\beta }\left( x\right) ^{\frac{1}{2}}$ is positive
definite since $\mathbf{D}_{\func{diag}}\left( x\right) $ is and 
\begin{eqnarray*}
\xi ^{\func{tr}}\mathbf{D}_{\beta }\left( x\right) ^{\frac{1}{2}}\mathbf{Q}%
\left( x\right) \mathbf{D}_{\beta }\left( x\right) ^{\frac{1}{2}}\xi
&=&\left\vert \xi \right\vert ^{2}+\beta \left( \mathbf{D}_{\beta }\left(
x\right) ^{\frac{1}{2}}\xi \right) ^{\func{tr}}\mathbf{D}_{\func{diag}%
}\left( x\right) \left( \mathbf{D}_{\beta }\left( x\right) ^{\frac{1}{2}}\xi
\right) \\
&&-\left\vert \frac{1}{\sqrt{a_{11}\left( x\right) }}\xi \cdot \mathbf{D}%
_{\beta }\left( x\right) ^{\frac{1}{2}}b\left( x\right) \right\vert ^{2} \\
&\geq &\left\vert \xi \right\vert ^{2}-\left\vert \xi \right\vert ^{2}\frac{1%
}{a_{11}\left( x\right) }\left\vert \mathbf{D}_{\beta }\left( x\right) ^{%
\frac{1}{2}}b\left( x\right) \right\vert ^{2} \\
&\geq &\left\vert \xi \right\vert ^{2}-\left\vert \xi \right\vert ^{2}\left(
1-\beta \right) =\beta \left\vert \xi \right\vert ^{2}.
\end{eqnarray*}

Thus altogether we have%
\begin{eqnarray*}
\mathbf{D}_{\beta }\left( x\right) ^{\frac{1}{2}}\mathbf{Q}\left( x\right) 
\mathbf{D}_{\beta }\left( x\right) ^{\frac{1}{2}}-\beta \mathbb{I}_{n-1}
&\succcurlyeq &0, \\
\text{implies }\mathbf{Q}\left( x\right) -\beta \mathbf{D}_{\beta }\left(
x\right) ^{-1} &\succcurlyeq &0,
\end{eqnarray*}%
and hence%
\begin{eqnarray*}
\mathbf{Q}\left( x\right) &\succ &\beta \mathbf{D}_{\beta }\left( x\right)
^{-1}=\beta \left( \mathbf{D}\left( x\right) -\beta \mathbf{D}_{\func{diag}%
}\left( x\right) \right) \\
&\succcurlyeq &\beta \left( \mathbf{D}\left( x\right) -\frac{\beta }{\beta
_{0}}\mathbf{D}\left( x\right) \right) =\beta \left( 1-\frac{\beta }{\beta
_{0}}\right) \mathbf{D}\left( x\right) ,
\end{eqnarray*}%
where $\beta _{0}$ is as in (\ref{also have}) above. If we now choose $\beta
=\frac{\beta _{0}}{2}$ we obtain $\mathbf{Q}\left( x\right) \succ \frac{%
\beta _{0}}{4}\mathbf{D}\left( x\right) $. Since we trivially have $\mathbf{Q%
}\left( x\right) \preccurlyeq \mathbf{D}\left( x\right) $, it follows that $%
\mathbf{Q}\left( x\right) $ is positive semidefinite and comparable to the
matrix $\mathbf{D}\left( x\right) $, and hence also that $\mathbf{Q}\left(
x\right) _{\func{diag}}\sim \mathbf{D}_{\func{diag}}\left( x\right) $. Since 
$\mathbf{D}\left( x\right) $ is comparable to its diagonal matrix $\mathbf{D}%
_{\func{diag}}\left( x\right) $ (because $\mathbf{A}\left( x\right) $ is
comparabe to $\mathbf{A}_{\func{diag}}\left( x\right) $), we now conclude
that $\mathbf{Q}\left( x\right) $ is comparable to its associated diagonal
matrix $\mathbf{Q}_{\func{diag}}\left( x\right) $. Thus $\mathbf{Q}\left(
x\right) $ is diagonally elliptical by definition, and since $\mathbf{Q}_{%
\func{diag}}\left( x\right) \sim \mathbf{D}_{\func{diag}}\left( x\right) $
as mentioned above, we conclude that $q_{i,i}\left( x\right) \approx
a_{i,i}\left( x\right) $ for $2\leq i\leq n$, which is (\ref{q-a}).

Finally,%
\begin{equation*}
\sum_{k=2}^{n}a_{k,k}\mathbf{e}_{k}\otimes \mathbf{e}_{k}\prec \mathbf{A}_{%
\func{diag}}\left( x\right) \prec C\mathbf{A}\left( x\right) ,
\end{equation*}%
and 
\begin{eqnarray*}
\xi ^{\func{tr}}Z_{1}\left( x\right) Z_{1}\left( x\right) ^{\func{tr}}\xi
&=&\left( Z_{1}\left( x\right) \cdot \xi \right) ^{2}=\left( \sqrt{a_{1,1}}%
\xi _{1}+\sum_{k=2}^{n}\frac{a_{1,k}}{\sqrt{a_{1,1}}}\xi _{k}\right) ^{2} \\
&\leq &\left( \sqrt{a_{1,1}}\xi _{1}+\sum_{k=2}^{n}\frac{\gamma \sqrt{%
a_{1,1}a_{k,k}}}{\sqrt{a_{1,1}}}\xi _{k}\right) ^{2}=\left( \sqrt{a_{1,1}}%
\xi _{1}+\sum_{k=2}^{n}\gamma \sqrt{a_{k,k}}\xi _{k}\right) ^{2} \\
&\leq &C\left( a_{1,1}\xi _{1}^{2}+\sum_{k=2}^{n}a_{k,k}\xi _{k}^{2}\right)
=C\xi ^{\func{tr}}\mathbf{A}_{\func{diag}}\left( x\right) \xi \leq C\xi ^{%
\func{tr}}\mathbf{A}\left( x\right) \xi ,
\end{eqnarray*}%
shows that $Z_{1}\left( x\right) Z_{1}\left( x\right) ^{\func{tr}}\prec C%
\mathbf{A}\left( x\right) $. On the other hand, we have 
\begin{equation*}
\mathbf{Q}\left( x\right) \sim \mathbf{Q}_{\func{diag}}\left( x\right) \sim
\sum_{k=2}^{n}\left( a_{k,k}-\frac{\left( a_{1,k}\right) ^{2}}{a_{1,1}}%
\right) \mathbf{e}_{k}\otimes \mathbf{e}_{k}\sim \sum_{k=2}^{n}a_{k,k}%
\mathbf{e}_{k}\otimes \mathbf{e}_{k}\ ,
\end{equation*}%
since $\left( a_{1,k}\right) ^{2}\leq \gamma ^{2}a_{1,1}a_{k,k}$ for some $%
\gamma ^{2}<1$ by diagonal positivity of $\mathbf{A}\left( x\right) $. Thus
we have%
\begin{eqnarray*}
a_{1,1}\mathbf{e}_{1}\otimes \mathbf{e}_{1} &\prec &\sum_{k=1}^{n}a_{k,k}%
\mathbf{e}_{k}\otimes \mathbf{e}_{k}=\mathbf{A}_{\func{diag}}\left( x\right)
\sim \mathbf{A}\left( x\right) \\
&=&Z_{1}Z_{1}^{\func{tr}}+\left[ 
\begin{array}{cc}
0 & 0 \\ 
0 & \mathbf{Q}%
\end{array}%
\right] \sim Z_{1}Z_{1}^{\func{tr}}+\sum_{k=2}^{n}a_{k,k}\mathbf{e}%
_{k}\otimes \mathbf{e}_{k}\ ,
\end{eqnarray*}%
which proves (\ref{Z1 comp}).
\end{proof}

\begin{lemma}
\label{comp and sub}Suppose that the matrix function $\mathbf{A}\left(
x\right) $ is diagonally elliptical and subordinate. Then with notation as
in the previous lemma, $\mathbf{Q}\left( x\right) $ is subordinate.
\end{lemma}

\begin{proof}
We have by definition,%
\begin{equation*}
\mathbf{Q}=\left[ a_{k,j}-\frac{a_{1,k}a_{1,j}}{a_{1,1}}\right] _{k,j=2}^{n},
\end{equation*}%
and by the comparability and subordinaticity of the matrix $\mathbf{A}$, we
have the estimates $q_{j,j}\approx a_{j,j}$ and%
\begin{equation*}
\left\vert a_{k,j}\right\vert \leq \left( a_{k,k}a_{j,j}\right) ^{\frac{1}{2}%
}\text{ and }\left\vert \nabla a_{k,j}\right\vert \leq \left( a_{j,j}\right)
^{\frac{1}{2}},\ \ \ \ \ 1\leq k\leq j\leq n.
\end{equation*}%
Thus we compute%
\begin{eqnarray*}
\left\vert \nabla \left( a_{k,j}-\frac{a_{1,k}a_{1,j}}{a_{1,1}}\right)
\right\vert &=&\left\vert \nabla a_{k,j}-\frac{\left( \nabla a_{1,k}\right)
a_{1,j}}{a_{1,1}}-\frac{a_{1,k}\left( \nabla a_{1,j}\right) }{a_{1,1}}+\frac{%
a_{1,k}a_{1,j}\nabla E_{1}}{\left( a_{1,1}\right) ^{2}}\right\vert \\
&\leq &\left( a_{j,j}\right) ^{\frac{1}{2}}+\frac{\left(
a_{1,1}a_{k,k}\right) ^{\frac{1}{2}}\left( a_{j,j}\right) ^{\frac{1}{2}}}{%
a_{1,1}}+\frac{\left( a_{k,k}\right) ^{\frac{1}{2}}\left(
a_{1,1}a_{j,j}\right) ^{\frac{1}{2}}}{a_{1,1}}+\frac{\left(
a_{1,1}a_{k,k}\right) ^{\frac{1}{2}}\left( a_{1,1}a_{j,j}\right) ^{\frac{1}{2%
}}\left( a_{1,1}\right) ^{\frac{1}{2}}}{\left( a_{1,1}\right) ^{2}} \\
&\approx &\left( q_{j,j}\right) ^{\frac{1}{2}}\left\{ 1+\frac{\left(
q_{1,1}q_{k,k}\right) ^{\frac{1}{2}}}{q_{1,1}}+\frac{\left(
q_{1,1}qa_{k,k}\right) ^{\frac{1}{2}}}{q_{1,1}}+\frac{\left( q_{k,k}\right)
^{\frac{1}{2}}\left( q_{j,j}\right) ^{\frac{1}{2}}}{\left( q_{1,1}\right) ^{%
\frac{1}{2}}}\right\} \lesssim \left( q_{j,j}\right) ^{\frac{1}{2}},
\end{eqnarray*}%
which shows that $\mathbf{Q}$ is subordinate by Theorem \ref{sub char}.
\end{proof}

An induction argument using Lemma \ref{diag ellip} shows in particular that
we can iterate the $1$-Square Decomposition starting with a diagonally
elliptical matrix function $\mathbf{A}\left( x\right) $ to produce a $\left(
p-1\right) $-Square Decomposition 
\begin{equation*}
\mathbf{A}\left( x\right) =\sum_{k=1}^{p-1}X_{k}X_{k}^{\func{tr}}+\mathbf{A}%
_{p}\left( x\right)
\end{equation*}%
where $\mathbf{A}_{p}\left( x\right) $ is an $\left( n-p+1\right) \times
\left( n-p+1\right) $ diagonally elliptical matrix function. This iterated
construction will be used to prove Theorem \ref{final n Grushin}. The next
lemma shows that the property of being $\left( \ell ,\varepsilon ,\delta
^{\prime },\delta ^{\prime \prime }\right) $-strongly\emph{\ }$C^{4,2\delta
} $ also propagates through the one step square decomposition $1$-Square
Decomposition upon decreasing $\ell $ by $1$.

\begin{lemma}
\label{off diagonal}Suppose $\mathbf{A}\left( x\right) \in C^{4,2\delta
}\left( \mathbb{R}^{M}\right) $ is $\left( \ell ,\varepsilon ,\delta
^{\prime },\delta ^{\prime \prime }\right) $-strongly $C^{4,2\delta }$ for
some $1\leq \ell \leq n$ and $\frac{1}{4}\leq \varepsilon <1$. Then we have
the corresponding inequalities for $\mathbf{Q}\left( x\right) $ but with $%
k\geq 2$: 
\begin{eqnarray}
\left\vert D^{\mu }q_{k,k}\right\vert &\lesssim &\left\vert
q_{k,k}\right\vert ^{\left[ 1-\left\vert \mu \right\vert \varepsilon \right]
_{+}+\delta ^{\prime }},\ \ \ \ \ 1\leq \left\vert \mu \right\vert \leq 4%
\text{ and }2\leq k\leq \ell ,  \label{akj con} \\
\left[ q_{k,k}\right] _{\mu ,2\delta } &\lesssim &1,\ \ \ \ \ \left\vert \mu
\right\vert =4\text{ and }2\leq k\leq \ell ,  \notag \\
\left\vert D^{\mu }q_{k,j}\right\vert &\lesssim &\left( \min_{1\leq s\leq
j}q_{s,s}\right) ^{\left[ \frac{1}{2}+\left( 2-\left\vert \mu \right\vert
\right) \varepsilon \right] _{+}+\delta ^{\prime \prime }},\ \ \ \ \ 0\leq
\left\vert \mu \right\vert \leq 4\text{ and }2\leq k\leq \ell <j\leq p-1, 
\notag \\
\left[ q_{k,j}\right] _{\mu ,2\delta } &\lesssim &1,\ \ \ \ \ \left\vert \mu
\right\vert =4\text{ and }2\leq k\leq \ell <j\leq p-1,  \notag \\
\left\vert D^{\mu }q_{k,j}\right\vert &\lesssim &\left( \min_{1\leq s\leq
k}q_{s,s}\right) ^{\left[ \frac{1}{2}+\left( 2-\left\vert \mu \right\vert
\right) \varepsilon \right] _{+}+\delta ^{\prime \prime }},\ \ \ \ \ 0\leq
\left\vert \mu \right\vert \leq 4\text{ and }2\leq k\leq \ell \leq j\leq
p-1<j\leq n,  \notag \\
\left[ q_{k,j}\right] _{\mu ,2\delta } &\lesssim &1,\ \ \ \ \ \left\vert \mu
\right\vert =4\text{ and }2\leq k\leq \ell \leq p-1<j\leq n.  \notag
\end{eqnarray}%
In particular $\mathbf{Q}\in C^{4,2\delta }$, and if $\ell \geq 2$, then $%
\mathbf{Q}$ is $\left( \ell -1,\varepsilon ,\delta ^{\prime },\delta
^{\prime \prime }\right) $-strongly $C^{4,2\delta }$.
\end{lemma}

In order to prove the lemma, we will use the trivial inequality $\left[ a+b%
\right] _{+}\leq \left[ a\right] _{+}+\left[ b\right] _{+}$ together with (%
\ref{q-a}), the product formula%
\begin{equation}
D^{\mu }\left( fgh\right) =\sum_{\alpha +\beta +\gamma =\mu }c_{\alpha
,\beta ,\gamma }^{\mu }\left( D^{\alpha }f\right) \left( D^{\beta }g\right)
\left( D^{\gamma }h\right) ,  \label{pdt form}
\end{equation}%
and the composition formula,%
\begin{equation}
\nabla ^{M}\left( \psi \circ h\right) =\sum_{m=1}^{M}\left( \psi ^{\left(
m\right) }\circ h\right) \left( \sum_{\substack{ \eta =\left( \eta
_{1},...,\eta _{M}\right) \in \mathbb{Z}_{+}^{M}  \\ \eta _{1}+\eta
_{2}+...+\eta _{M}=m  \\ \eta _{1}+2\eta _{2}+...+M\eta _{M}=M}}\left[ 
\begin{array}{c}
M \\ 
\eta%
\end{array}%
\right] \left( \nabla h\right) ^{\eta _{1}}...\left( \nabla ^{M}h\right)
^{\eta _{M}}\right) ,  \label{comp fla}
\end{equation}%
where $\left[ 
\begin{array}{c}
M \\ 
\alpha%
\end{array}%
\right] $ is defined for $\alpha =\left( \alpha _{1},...,\alpha _{M}\right)
\in \mathbb{Z}_{+}^{M}$ satisfying $\alpha _{1}+2\alpha _{2}+...+M\alpha
_{M}=M$. See e.g. \cite{MaSaUrVu}\ or \cite[(3.1)]{KoSa1}. We will use the
reciprocal function $\psi \left( t\right) =\frac{1}{t}$ in (\ref{comp fla})
so that $\psi ^{\left( m\right) }\left( t\right) =\left( -1\right)
^{m}m!t^{-m-1}$.

\begin{proof}
We first derive an auxiliary estimate that will be used throughout the
proof. Namely, for any multiindex $\gamma $ with $1\leq |\gamma |\leq 4$ we
have using (\ref{akj hyp}) and (\ref{comp fla}), 
\begin{eqnarray*}
\left\vert D^{\gamma }\frac{1}{a_{1,1}}\right\vert &=&\left\vert
\sum_{m=1}^{\left\vert \gamma \right\vert }\left( \frac{1}{a_{11}}\right)
^{m+1}\sum_{\substack{ \eta =\left( \eta _{1},...,\eta _{M}\right) \in 
\mathbb{Z}_{+}^{\left\vert \gamma \right\vert }  \\ \eta _{1}+\eta
_{2}+...+\eta _{\left\vert \gamma \right\vert }=m  \\ \eta _{1}+2\eta
_{2}+...+\left\vert \gamma \right\vert \eta _{\left\vert \gamma \right\vert
}=\left\vert \gamma \right\vert }}\left[ 
\begin{array}{c}
\gamma \\ 
\eta%
\end{array}%
\right] \left( \nabla a_{11}\right) ^{\eta _{1}}...\left( \nabla
^{\left\vert \gamma \right\vert }a_{11}\right) ^{\eta _{\left\vert \gamma
\right\vert }}\right\vert \\
&\lesssim &\sum_{m=1}^{\left\vert \gamma \right\vert }\left( \frac{1}{a_{11}}%
\right) ^{m+1}\sum_{\substack{ \eta =\left( \eta _{1},...,\eta _{M}\right)
\in \mathbb{Z}_{+}^{\left\vert \gamma \right\vert }  \\ \eta _{1}+\eta
_{2}+...+\eta _{\left\vert \gamma \right\vert }=m  \\ \eta _{1}+2\eta
_{2}+...+\left\vert \gamma \right\vert \eta _{\left\vert \gamma \right\vert
}=\left\vert \gamma \right\vert }}a_{11}^{\left( \left[ 1-\varepsilon \right]
_{+}+\delta ^{\prime }\right) \eta _{1}+\left( \left[ 1-2\varepsilon \right]
_{+}+\delta ^{\prime }\right) \eta _{2}+\dots +\left( \left[ 1-\left\vert
\gamma \right\vert \varepsilon \right] _{+}+\delta ^{\prime }\right) \eta
_{\left\vert \gamma \right\vert }} \\
&\leq &\left( \frac{1}{a_{11}}\right) ^{m+1}\sum_{\substack{ \eta =\left(
\eta _{1},...,\eta _{M}\right) \in \mathbb{Z}_{+}^{\left\vert \gamma
\right\vert }  \\ \eta _{1}+\eta _{2}+...+\eta _{\left\vert \gamma
\right\vert }=m  \\ \eta _{1}+2\eta _{2}+...+\left\vert \gamma \right\vert
\eta _{\left\vert \gamma \right\vert }=\left\vert \gamma \right\vert }}%
a_{11}^{\left( 1-\varepsilon +\delta ^{\prime }\right) \eta _{1}+\left(
1-2\varepsilon +\delta ^{\prime }\right) \eta _{2}+\dots +\left(
1-\left\vert \gamma \right\vert \varepsilon +\delta ^{\prime }\right) \eta
_{\left\vert \gamma \right\vert }} \\
&=&\sum_{m=1}^{\left\vert \gamma \right\vert }\left( \frac{1}{a_{11}}\right)
^{m+1}\sum_{\substack{ \eta =\left( \eta _{1},...,\eta _{M}\right) \in 
\mathbb{Z}_{+}^{\left\vert \gamma \right\vert }  \\ \eta _{1}+\eta
_{2}+...+\eta _{\left\vert \gamma \right\vert }=m  \\ \eta _{1}+2\eta
_{2}+...+\left\vert \gamma \right\vert \eta _{\left\vert \gamma \right\vert
}=\left\vert \gamma \right\vert }}a_{11}^{m-\left\vert \gamma \right\vert
\varepsilon +m\delta ^{\prime }}.
\end{eqnarray*}%
Note\ that this implies the estimate%
\begin{equation}
\left\vert D^{\gamma }\frac{1}{a_{1,1}}\right\vert \leq
Ca_{11}^{-1-\left\vert \gamma \right\vert \varepsilon +\delta ^{\prime }},\
\ \ \ \ 1\leq \left\vert \gamma \right\vert \leq 4,  \label{aux}
\end{equation}

The product formula (\ref{pdt form}), hypothesis (\ref{akj hyp}) and the
estimate (\ref{aux}) above imply that for $k\geq 2$ and $M=\left\vert \mu
\right\vert $, 
\begin{eqnarray*}
\left\vert \nabla ^{M}q_{k,k}\right\vert &=&\left\vert \nabla ^{M}\left(
a_{k,k}-\frac{\left( a_{1,k}\right) ^{2}}{a_{1,1}}\right) \right\vert \leq
\left\vert \nabla ^{M}a_{k,k}\right\vert +\left\vert \nabla ^{M}\frac{\left(
a_{1,k}\right) ^{2}}{a_{1,1}}\right\vert \\
&\lesssim &\left\vert a_{k,k}\right\vert ^{\left[ 1-\left\vert \mu
\right\vert \varepsilon \right] _{+}+\delta ^{\prime }} \\
&&+\sum_{\alpha +\beta +\gamma =\mu }\left( \min_{1\leq s\leq k}\left\{
a_{s,s}\left( x\right) \right\} \right) ^{\left[ \frac{1}{2}+\left(
2-\left\vert \alpha \right\vert \right) \varepsilon \right] _{+}+\delta
^{\prime \prime }}\left( \min_{1\leq s\leq k}\left\{ a_{s,s}\left( x\right)
\right\} \right) ^{\left[ \frac{1}{2}+\left( 2-\left\vert \beta \right\vert
\right) \varepsilon \right] _{+}+\delta ^{\prime \prime }}\left\vert
D^{\gamma }\frac{1}{a_{1,1}}\right\vert \\
&\lesssim &\left\vert a_{k,k}\right\vert ^{\left[ 1-\left\vert \mu
\right\vert \varepsilon \right] _{+}+\delta ^{\prime }}+\sum_{\alpha +\beta
+\gamma =\mu }\left( \min_{1\leq s\leq k}\left\{ a_{s,s}\left( x\right)
\right\} \right) ^{\left[ 1+\left( 4-\left( \left\vert \alpha \right\vert
+\left\vert \beta \right\vert \right) \right) \varepsilon \right]
_{+}+2\delta ^{\prime \prime }}a_{11}^{-1-\left\vert \gamma \right\vert
\varepsilon +\delta ^{\prime }} \\
&\lesssim &\left\vert a_{k,k}\right\vert ^{\left[ 1-\left\vert \mu
\right\vert \varepsilon \right] _{+}+\delta ^{\prime }}+\sum_{\alpha +\beta
+\gamma =\mu }\left( \min_{1\leq s\leq k}\left\{ a_{s,s}\left( x\right)
\right\} \right) ^{\left[ 1+\left( 4-\left( \left\vert \alpha \right\vert
+\left\vert \beta \right\vert \right) \right) \varepsilon \right]
_{+}+2\delta ^{\prime \prime }-1-\left\vert \gamma \right\vert \varepsilon
+\delta ^{\prime }} \\
&\lesssim &\left\vert a_{k,k}\right\vert ^{\left[ 1-\left\vert \mu
\right\vert \varepsilon \right] _{+}+\delta ^{\prime }},
\end{eqnarray*}%
where the last line above follows from the fact that%
\begin{eqnarray*}
&&\left[ 1+\left( 4-\left( \left\vert \alpha \right\vert +\left\vert \beta
\right\vert \right) \right) \varepsilon \right] _{+}+2\delta ^{\prime \prime
}-1-\left\vert \gamma \right\vert \varepsilon +\delta ^{\prime } \\
&\geq &1+\left( 4-\left( \left\vert \alpha \right\vert +\left\vert \beta
\right\vert \right) \right) \varepsilon +2\delta ^{\prime }-1-\left\vert
\gamma \right\vert \varepsilon +\delta ^{\prime } \\
&=&\left( 4-\left\vert \mu \right\vert \right) \varepsilon +3\delta ^{\prime
}\geq \left[ 1-\left\vert \mu \right\vert \varepsilon \right] _{+}+\delta
^{\prime },
\end{eqnarray*}%
if $\left\vert \mu \right\vert \leq 4$ and $\varepsilon \geq \frac{1}{4}$.
Since $a_{k,k}\approx q_{k,k}$ by (\ref{q-a}), we have established the first
line in (\ref{akj con}).

The second line is obtained similarly using the subproduct and subchain
rules for the seminorm (\ref{def mod D}) as in \cite{Bon} and \cite[(3.7)
and (3.8) and subsequent proofs]{KoSa1}. Indeed, if $\left\vert \mu
\right\vert =4$, then 
\begin{eqnarray*}
\left[ \frac{\left( a_{1,k}\right) ^{2}}{a_{1,1}}\right] _{\mu ,2\delta }
&=&\left\vert \sum_{\alpha +\beta +\gamma =\mu }c_{\alpha ,\beta ,\gamma
}^{\mu }\left( D^{\alpha }a_{1,k}\right) \left( D^{\beta }a_{1,k}\right)
\left( D^{\gamma }\frac{1}{a_{1,1}}\right) \right\vert ^{2\delta } \\
&\lesssim &\sum_{\alpha +\beta +\gamma =\mu }\left[ a_{1,k}\right] _{\alpha
,2\delta }\left\vert D^{\beta }a_{1,k}\right\vert \left\vert D^{\gamma }%
\frac{1}{a_{1,1}}\right\vert \\
&&+\sum_{\alpha +\beta +\gamma =\mu }\left\vert D^{\alpha
}a_{1,k}\right\vert \left[ a_{1,k}\right] _{\beta ,2\delta }\left\vert
D^{\gamma }\frac{1}{a_{1,1}}\right\vert \\
&&+\sum_{\alpha +\beta +\gamma =\mu }\left\vert D^{\alpha
}a_{1,k}\right\vert \left\vert D^{\beta }a_{1,k}\right\vert \left[ \frac{1}{%
a_{1,1}}\right] _{\gamma ,2\delta }.
\end{eqnarray*}%
Now use that if $\left\vert \alpha \right\vert <4$, then%
\begin{eqnarray*}
\left[ a_{1,k}\right] _{\alpha ,2\delta } &=&\lim \sup_{y,z\rightarrow x}%
\frac{\left\vert D^{\alpha }a_{1,k}\left( y\right) -D^{\alpha }a_{1,k}\left(
z\right) \right\vert }{\left\vert y-z\right\vert ^{2\delta }} \\
&\leq &\lim \sup_{y,z\rightarrow x}\frac{\left\vert D^{\alpha }a_{1,k}\left(
y\right) -D^{\alpha }a_{1,k}\left( z\right) \right\vert }{\left\vert
y-z\right\vert }\left\vert y-z\right\vert ^{1-2\delta }=0,
\end{eqnarray*}%
while if $\left\vert \alpha \right\vert =4$, then%
\begin{equation*}
\left[ a_{1,k}\right] _{\alpha ,2\delta }\leq \left\Vert a_{1,k}\right\Vert
_{C^{4,2\delta }}\leq C.
\end{equation*}

Now we turn to proving the third line in (\ref{akj con}). In order to
simplify notation, set%
\begin{equation*}
m_{k}\left( x\right) \equiv \min_{1\leq s\leq k}\left\{ q_{s,s}\left(
x\right) \right\} ,\ \ \ \ \ 1\leq k\leq n,
\end{equation*}%
and note that by (\ref{q-a}) we have $m_{k}\left( x\right) \approx
\min_{1\leq s\leq k}\left\{ a_{s,s}\left( x\right) \right\} $. The case $%
\left\vert \mu \right\vert =0$ is immediate from the identity $%
q_{k,j}=a_{k,j}-\frac{a_{1,k}a_{1,j}}{a_{1,1}}$, and the estimate%
\begin{equation*}
\left\vert \frac{a_{1,k}a_{1,j}}{a_{1,1}}\right\vert \lesssim \frac{\left(
m_{k}\right) ^{\frac{1}{2}+2\varepsilon +\delta ^{\prime \prime }}\left(
m_{j}\right) ^{\frac{1}{2}+2\varepsilon +\delta ^{\prime \prime }}}{a_{1,1}}%
\leq \left( a_{1,1}\right) ^{-\frac{1}{2}+2\varepsilon +\delta ^{\prime
\prime }}\left( m_{j}\right) ^{\frac{1}{2}+2\varepsilon +\delta ^{\prime
\prime }}\leq \left( m_{j}\right) ^{\left[ \frac{1}{2}+\left( 2-\left\vert
\mu \right\vert \right) \varepsilon \right] _{+}+\delta ^{\prime \prime }},
\end{equation*}%
since $-\frac{1}{2}+2\varepsilon +\delta ^{\prime \prime }\geq 0$.

To prove the cases $\left\vert \mu \right\vert \geq 1$, we continue to use (%
\ref{akj hyp}), (\ref{aux}) and (\ref{q-a}), to obtain the estimate%
\begin{eqnarray*}
&&\left\vert D^{\mu }\frac{a_{1,k}a_{1,j}}{a_{1,1}}\right\vert =\left\vert
\sum_{\alpha +\beta +\gamma =\mu }c_{\alpha ,\beta ,\gamma }^{\mu }\left(
D^{\alpha }a_{1,k}\right) \left( D^{\beta }a_{1,j}\right) \left( D^{\gamma }%
\frac{1}{a_{1,1}}\right) \right\vert \\
&\lesssim &\sum_{\alpha +\beta +\gamma =\mu }\left( m_{k}\right) ^{\left[ 
\frac{1}{2}+\left( 2-\left\vert \alpha \right\vert \right) \varepsilon %
\right] _{+}+\delta ^{\prime \prime }}\left( m_{j}\right) ^{\left[ \frac{1}{2%
}+\left( 2-\left\vert \beta \right\vert \right) \varepsilon \right]
_{+}+\delta ^{\prime \prime }}\left( a_{1,1}\right) ^{-1-\left\vert \gamma
\right\vert \varepsilon +\delta ^{\prime }} \\
&\lesssim &\left( q_{1,1}\right) ^{-1-\left\vert \gamma \right\vert
\varepsilon +\delta ^{\prime }}\ \left( m_{k}\right) ^{\left[ \frac{1}{2}%
+\left( 2-\left\vert \alpha \right\vert \right) \varepsilon \right]
_{+}+\delta ^{\prime \prime }}\ \left( m_{j}\right) ^{\left[ \frac{1}{2}%
+\left( 2-\left\vert \beta \right\vert \right) \varepsilon \right]
_{+}+\delta ^{\prime \prime }} \\
&\lesssim &\left( m_{k}\right) ^{\left[ \frac{1}{2}+\left( 2-\left\vert
\alpha \right\vert \right) \varepsilon \right] _{+}+\delta ^{\prime \prime
}-1-\left\vert \gamma \right\vert \varepsilon +\delta ^{\prime }}\ \left(
m_{j}\right) ^{\left[ \frac{1}{2}+\left( 2-\left\vert \beta \right\vert
\right) \varepsilon \right] _{+}+\delta ^{\prime \prime }}.
\end{eqnarray*}

Note that placing a derivative in position $\beta $ in the last line above
causes the most damage up to order $2$ because $m_{j}\leq m_{k}\leq
m_{1}=q_{1,1}$. After that, placing a third derivative in position $\alpha $
equivalently $\gamma $ is the next most damaging, or depending on the
precise relation between $m_{j}$ and $m_{k}$, repeating the position $\beta $
once more might be worse. Indeed, placing a third derivative in position $%
\alpha $ or $\gamma $ loses $m_{k}^{-\varepsilon }$, while placing a third
derivative in position $\beta \,\,$loses $m_{j}^{-\frac{1}{2}}$, and since $%
\varepsilon \geq \frac{1}{4}$, either of these could be largest under the
restriction $m_{j}\leq m_{k}$. However after three derivatives have been
assigned as above, the fourth derivative does the most damage when applied
to position $\alpha $ equivalently $\gamma $. As a consequence we need only
consider the cases where $\beta $ is filled up to order $2$, and then $%
\alpha $ or $\beta $ is filled thereafter. This will become clear as the
reader progresses through the proof.

In the case $\left\vert \mu \right\vert =1$, $\left\vert D^{\mu }\frac{%
a_{1,k}a_{1,j}}{a_{1,1}}\right\vert $ is bounded by%
\begin{eqnarray*}
&&\left( m_{k}\right) ^{\left[ \frac{1}{2}+\left( 2-\left\vert \alpha
\right\vert \right) \varepsilon \right] _{+}+\delta ^{\prime \prime
}-1-\left\vert \gamma \right\vert \varepsilon +\delta ^{\prime }}\ \left(
m_{j}\right) ^{\left[ \frac{1}{2}+\left( 2-\left\vert \beta \right\vert
\right) \varepsilon \right] _{+}+\delta ^{\prime \prime }} \\
&\leq &\left( m_{k}\right) ^{\left[ \frac{1}{2}+2\varepsilon \right]
_{+}+\delta ^{\prime \prime }-1+\delta ^{\prime }}\ \left( m_{j}\right) ^{%
\left[ \frac{1}{2}+\varepsilon \right] _{+}+\delta ^{\prime \prime }} \\
&=&\left( m_{k}\right) ^{\frac{1}{2}+2\varepsilon +\delta ^{\prime }+\delta
^{\prime \prime }-1}\ \left( m_{j}\right) ^{\frac{1}{2}+\varepsilon +\delta
^{\prime \prime }}\leq \left( m_{j}\right) ^{\frac{1}{2}+\varepsilon +\delta
^{\prime \prime }}=\left( m_{j}\right) ^{\left[ \frac{1}{2}+\left(
2-\left\vert \mu \right\vert \right) \varepsilon \right] _{+}+\delta
^{\prime \prime }},
\end{eqnarray*}%
since $\frac{1}{2}+2\varepsilon +\delta ^{\prime }+\delta ^{\prime \prime
}-1\geq \delta ^{\prime }+\delta ^{\prime \prime }>0$.

In the case $\left\vert \mu \right\vert =2$, $\left\vert D^{\mu }\frac{%
a_{1,k}a_{1,j}}{a_{1,1}}\right\vert $ is bounded by 
\begin{eqnarray*}
&&\left( m_{k}\right) ^{\left[ \frac{1}{2}+\left( 2-\left\vert \alpha
\right\vert \right) \varepsilon \right] _{+}+\delta ^{\prime \prime
}-1-\left\vert \gamma \right\vert \varepsilon +\delta ^{\prime }}\ \left(
m_{j}\right) ^{\left[ \frac{1}{2}+\left( 2-\left\vert \beta \right\vert
\right) \varepsilon \right] _{+}+\delta ^{\prime \prime }} \\
&\leq &\left( m_{k}\right) ^{\left[ \frac{1}{2}+2\varepsilon \right]
_{+}+\delta ^{\prime \prime }-1+\delta ^{\prime }}\ \left( m_{j}\right) ^{%
\frac{1}{2}+\delta ^{\prime \prime }}=\left( m_{k}\right) ^{\frac{1}{2}%
+2\varepsilon -1+\delta ^{\prime }+\delta ^{\prime \prime }}\ \left(
m_{j}\right) ^{\frac{1}{2}+\delta ^{\prime \prime }} \\
&\leq &\left( m_{j}\right) ^{\frac{1}{2}+\delta ^{\prime \prime }}=\left(
m_{j}\right) ^{\left[ \frac{1}{2}+\left( 2-\left\vert \mu \right\vert
\right) \varepsilon \right] _{+}+\delta ^{\prime \prime }},
\end{eqnarray*}%
since $\frac{1}{2}+2\varepsilon -1+\delta ^{\prime }+\delta ^{\prime \prime
}\geq \delta ^{\prime }+\delta ^{\prime \prime }>0$.

In the case $\left\vert \mu \right\vert =3$, the most damage is caused
either when $\left\vert \beta \right\vert =2$ and $\left\vert \alpha
\right\vert =1$, or when $\left\vert \beta \right\vert =3$. In the former
instance, $\left\vert D^{\mu }\frac{a_{1,k}a_{1,j}}{a_{1,1}}\right\vert $ is
bounded by%
\begin{equation*}
\left( m_{k}\right) ^{\left[ \frac{1}{2}+\left( 2-\left\vert \alpha
\right\vert \right) \varepsilon \right] _{+}+\delta ^{\prime \prime
}-1-\left\vert \gamma \right\vert \varepsilon +\delta ^{\prime }}\ \left(
m_{j}\right) ^{\left[ \frac{1}{2}+\left( 2-\left\vert \beta \right\vert
\right) \varepsilon \right] _{+}+\delta ^{\prime \prime }}\leq \left(
m_{k}\right) ^{-\frac{1}{2}+\varepsilon +\delta ^{\prime }+\delta ^{\prime
\prime }}\ \left( m_{j}\right) ^{\frac{1}{2}+\delta ^{\prime \prime }}.
\end{equation*}%
In the case $\varepsilon \geq \frac{1}{2}$, this is bounded by%
\begin{equation*}
\left( m_{k}\right) ^{\delta ^{\prime }+\delta ^{\prime \prime }}\ \left(
m_{j}\right) ^{\frac{1}{2}+\delta ^{\prime \prime }}\leq \left( m_{j}\right)
^{\frac{1}{2}+\delta ^{\prime \prime }}=\left( m_{j}\right) ^{\left[ \frac{1%
}{2}+\left( 2-\left\vert \mu \right\vert \right) \varepsilon \right]
_{+}+\delta ^{\prime \prime }}.
\end{equation*}%
If $\frac{1}{4}\leq \varepsilon <\frac{1}{2}$, then this is instead bounded
by%
\begin{equation*}
\left( m_{k}\right) ^{-\frac{1}{2}+2\varepsilon +\delta ^{\prime }+\delta
^{\prime \prime }}\ \left( m_{j}\right) ^{\frac{1}{2}-\varepsilon +\delta
^{\prime \prime }}\leq \left( m_{j}\right) ^{\frac{1}{2}-\varepsilon +\delta
^{\prime \prime }}=\left( m_{j}\right) ^{\left[ \frac{1}{2}+\left(
2-\left\vert \mu \right\vert \right) \varepsilon \right] _{+}+\delta
^{\prime }+\delta ^{\prime \prime }}.
\end{equation*}%
since $-\frac{1}{2}+2\varepsilon +\delta ^{\prime }+\delta ^{\prime \prime
}\geq \delta ^{\prime }+\delta ^{\prime \prime }>0$. In the latter instance,
when $\left\vert \beta \right\vert =3$, $\left\vert D^{\mu }\frac{%
a_{1,k}a_{1,j}}{a_{1,1}}\right\vert $ is bounded by%
\begin{eqnarray*}
&&\left( m_{k}\right) ^{\left[ \frac{1}{2}+\left( 2-\left\vert \alpha
\right\vert \right) \varepsilon \right] _{+}+\delta ^{\prime \prime
}-1-\left\vert \gamma \right\vert \varepsilon +\delta ^{\prime }}\ \left(
m_{j}\right) ^{\left[ \frac{1}{2}+\left( 2-\left\vert \beta \right\vert
\right) \varepsilon \right] _{+}+\delta ^{\prime \prime }} \\
&\leq &\left( m_{k}\right) ^{\left[ \frac{1}{2}+2\varepsilon \right]
_{+}+\delta ^{\prime \prime }-1+\delta ^{\prime }}\ \left( m_{j}\right) ^{%
\left[ \frac{1}{2}-\varepsilon \right] _{+}+\delta ^{\prime \prime }}\leq
\left( m_{j}\right) ^{\left[ \frac{1}{2}-\varepsilon \right] _{+}+\delta
^{\prime \prime }}=\left( m_{j}\right) ^{\left[ \frac{1}{2}+\left(
2-\left\vert \mu \right\vert \right) \varepsilon \right] _{+}+\delta
^{\prime \prime }},
\end{eqnarray*}%
since $2\varepsilon -\frac{1}{2}+\delta ^{\prime }+\delta ^{\prime \prime
}\geq \delta ^{\prime }+\delta ^{\prime \prime }>0$.

In the case $\left\vert \mu \right\vert =4$, we again consider the two most
damaging cases. When $\left\vert \beta \right\vert =3$ and $\left\vert
\alpha \right\vert =1$, we have that $\left\vert D^{\mu }\frac{a_{1,k}a_{1,j}%
}{a_{1,1}}\right\vert $ is bounded by%
\begin{eqnarray*}
&&\left( m_{k}\right) ^{\left[ \frac{1}{2}+\left( 2-\left\vert \alpha
\right\vert \right) \varepsilon \right] _{+}+\delta ^{\prime \prime
}-1-\left\vert \gamma \right\vert \varepsilon +\delta ^{\prime }}\ \left(
m_{j}\right) ^{\left[ \frac{1}{2}+\left( 2-\left\vert \beta \right\vert
\right) \varepsilon \right] _{+}+\delta ^{\prime \prime }} \\
&=&\left( m_{k}\right) ^{-\frac{1}{2}+\varepsilon +\delta ^{\prime }+\delta
^{\prime \prime }}\ \left( m_{j}\right) ^{\left[ \frac{1}{2}-\varepsilon %
\right] _{+}+\delta ^{\prime \prime }}.
\end{eqnarray*}%
In the case $\varepsilon \geq \frac{1}{2}$ this is bounded by $m_{j}^{\delta
^{\prime \prime }}=\left( m_{j}\right) ^{\left[ \frac{1}{2}+\left(
2-\left\vert \mu \right\vert \right) \varepsilon \right] _{+}+\delta
^{\prime \prime }}$, while in the case $\frac{1}{4}\leq \varepsilon <\frac{1%
}{2}$, this is bounded by 
\begin{equation*}
\left( m_{k}\right) ^{-\frac{1}{2}+\varepsilon +\delta ^{\prime }+\delta
^{\prime \prime }}\left( m_{j}\right) ^{\frac{1}{2}-\varepsilon +\delta
^{\prime \prime }}\leq \left( m_{k}\right) ^{\delta ^{\prime }+\delta
^{\prime \prime }}m_{j}^{\delta ^{\prime \prime }}\leq m_{j}^{\delta
^{\prime \prime }}=\left( m_{j}\right) ^{\left[ \frac{1}{2}+\left(
2-\left\vert \mu \right\vert \right) \varepsilon \right] _{+}+\delta
^{\prime \prime }},
\end{equation*}%
since $\delta ^{\prime }+\delta ^{\prime \prime }>0$. When $\left\vert
\alpha \right\vert =\left\vert \beta \right\vert =2$, then $\left\vert
D^{\mu }\frac{a_{1,k}a_{1,j}}{a_{1,1}}\right\vert $ is bounded by%
\begin{eqnarray*}
&&\left( m_{k}\right) ^{\left[ \frac{1}{2}+\left( 2-\left\vert \alpha
\right\vert \right) \varepsilon \right] _{+}+\delta ^{\prime \prime
}-1-\left\vert \gamma \right\vert \varepsilon +\delta ^{\prime }}\ \left(
m_{j}\right) ^{\left[ \frac{1}{2}+\left( 2-\left\vert \beta \right\vert
\right) \varepsilon \right] _{+}+\delta ^{\prime \prime }} \\
&\leq &\left( m_{k}\right) ^{\left[ \frac{1}{2}\right] _{+}+\delta ^{\prime
\prime }-1+\delta ^{\prime }}\ \left( m_{j}\right) ^{\left[ \frac{1}{2}%
\right] _{+}+\delta ^{\prime \prime }}\leq \left( m_{j}\right) ^{\delta
^{\prime \prime }}=\left( m_{j}\right) ^{\left[ \frac{1}{2}+\left(
2-\left\vert \mu \right\vert \right) \varepsilon \right] _{+}+\delta
^{\prime \prime }},
\end{eqnarray*}%
since $\delta ^{\prime }+\delta ^{\prime \prime }>0$.

This completes the proof of the third line in (\ref{akj con}), and just as
before, the fourth line is obtained similarly using the subproduct and
subchain rules for the H\"{o}lder expression $\left[ \cdot \right] _{\mu
,2\delta }$. Indeed, if $\left\vert \mu \right\vert =4$, then 
\begin{eqnarray*}
\left[ \frac{a_{1,k}a_{1,j}}{a_{1,1}}\right] _{\mu ,2\delta } &\lesssim
&\sum_{\alpha +\beta +\gamma =\mu }\left[ a_{1,k}\right] _{\alpha ,2\delta
}\left\vert D^{\beta }a_{1,j}\right\vert \left\vert D^{\gamma }\frac{1}{%
a_{1,1}}\right\vert \\
&&+\sum_{\alpha +\beta +\gamma =\mu }\left\vert D^{\alpha
}a_{1,k}\right\vert \left[ a_{1,j}\right] _{\beta ,2\delta }\left\vert
D^{\gamma }\frac{1}{a_{1,1}}\right\vert \\
&&+\sum_{\alpha +\beta +\gamma =\mu }\left\vert D^{\alpha
}a_{1,k}\right\vert \left\vert D^{\beta }a_{1,j}\right\vert \left[ \frac{1}{%
a_{1,1}}\right] _{\gamma ,2\delta },
\end{eqnarray*}%
and if $\left\vert \alpha \right\vert <4$, then%
\begin{eqnarray*}
\left[ a_{1,k}\right] _{\alpha ,2\delta } &=&\lim \sup_{y,z\rightarrow x}%
\frac{\left\vert D^{\alpha }a_{1,k}\left( y\right) -D^{\alpha }a_{1,k}\left(
z\right) \right\vert }{\left\vert y-z\right\vert ^{2\delta }} \\
&\leq &\lim \sup_{y,z\rightarrow x}\frac{\left\vert D^{\alpha }a_{1,k}\left(
y\right) -D^{\alpha }a_{1,k}\left( z\right) \right\vert }{\left\vert
y-z\right\vert }\left\vert y-z\right\vert ^{1-2\delta }=0,
\end{eqnarray*}%
while if $\left\vert \alpha \right\vert =4$, then%
\begin{equation*}
\left[ a_{1,k}\right] _{\alpha ,2\delta }\leq \left\Vert a_{1,k}\right\Vert
_{C^{4,2\delta }}\leq C.
\end{equation*}

Finally, we note that it is an easy matter to check that when $\ell <j\leq n$%
, we can replace $q_{j,j}$ with the larger quantity $q_{k,k}$ on the right
hand side of the estimates in the fifth and sixth lines. This completes the
proof of Lemma \ref{off diagonal}.
\end{proof}

\subsection{Proof of the main decomposition theorem}

At this point we can apply induction together with Lemmas \ref{diag ellip}, %
\ref{comp and sub}\ and \ref{off diagonal}, and Theorem \ref{sub char}, to
prove Theorem \ref{final n Grushin}.

\begin{proof}[Proof of Theorem \protect\ref{final n Grushin}]
Set $\mathbf{Q}_{1}\left( x\right) \equiv \mathbf{A}\left( x\right) $ and
suppose that $\mathbf{Q}_{1}\left( x\right) $ is diagonally elliptical and $%
\left( p-1,\varepsilon \right) $-strongly $C^{4,2\delta }$. Let $\mathbf{Q}%
_{1}\left( x\right) =Y_{1}\left( x\right) Y_{1}\left( x\right) ^{\func{tr}}+%
\mathbf{B}_{1}\left( x\right) $ be the $1$-Square Decomposition of $\mathbf{Q%
}_{1}\left( x\right) $ with $\mathbf{B}_{1}\left( x\right) =\left[ 
\begin{array}{cc}
0 & \mathbf{0}_{1\times \left( n-1\right) } \\ 
\mathbf{0}_{\left( n-1\right) \times 1} & \mathbf{Q}_{2}\left( x\right)%
\end{array}%
\right] $. Then Lemmas \ref{diag ellip}\ and \ref{off diagonal} show that $%
\mathbf{Q}_{2}\left( x\right) $ is diagonally elliptical and satisfies (\ref%
{Z1 comp}), and is $\left( p-2\right) $-strongly $C^{4,2\delta }$ so long as 
$p\geq 3$. Now we apply the above reasoning to the $\left( n-1\right) \times
\left( n-1\right) $ matrix $\mathbf{Q}_{2}\left( x\right) $ to obtain the $1$%
-Square Decomposition of $\mathbf{Q}_{2}\left( x\right) =Y_{2}\left(
x\right) Y_{2}\left( x\right) ^{\func{tr}}+\mathbf{B}_{2}\left( x\right) $
with $\mathbf{B}_{2}\left( x\right) =\left[ 
\begin{array}{cc}
0 & \mathbf{0}_{1\times \left( n-2\right) } \\ 
\mathbf{0}_{\left( n-2\right) \times 1} & \mathbf{Q}_{3}\left( x\right)%
\end{array}%
\right] $, and we see that $\mathbf{Q}_{3}\left( x\right) $ is diagonally
elliptical and satisfies (\ref{Z1 comp}) relative to $\mathbf{Q}_{2}\left(
x\right) $, and is $\left( p-3\right) $-strongly $C^{4,2}$ so long as $p\geq
4$. By induction we obtain that%
\begin{equation*}
\mathbf{Q}_{p-2}\left( x\right) =Y_{p-2}\left( x\right) Y_{p-2}\left(
x\right) ^{\func{tr}}+\mathbf{B}_{p-2}\left( x\right)
\end{equation*}%
where $\mathbf{Q}_{p-1}\left( x\right) $ is diagonally elliptical and
satisfies (\ref{Z1 comp}) relative to $\mathbf{Q}_{p-2}\left( x\right) $,
and is $1$-strongly $C^{4,2\delta }$. One more application of Lemmas \ref%
{diag ellip}\ and \ref{off diagonal} yields that 
\begin{equation*}
\mathbf{Q}_{p-1}\left( x\right) =Y_{p-1}\left( x\right) Y_{p-1}\left(
x\right) ^{\func{tr}}+\mathbf{B}_{p-1}\left( x\right)
\end{equation*}%
where $\mathbf{Q}_{p}\left( x\right) $ is diagonally elliptical. While we
cannot now assert that $\mathbf{Q}_{p}\left( x\right) $ is $1$-strongly $%
C^{4,2\delta }$, we do have that $\mathbf{Q}_{p-1}=\left[ q_{k,j}\right]
_{k,j=1}^{n-p+1}$ is $1$-strongly $C^{4,2\delta }$, and hence 
\begin{equation*}
\mathbf{Q}_{p}=\left[ q_{k,j}-\frac{q_{1k}q_{1j}}{q_{11}}\right]
_{j,k=2}^{n}\in C^{4,2\delta },
\end{equation*}%
upon estimating derivatives of $\frac{q_{1k}q_{1j}}{q_{11}}$ as above.

Now let $Z_{k}$ be the $n$-vector whose final $n-k+1$ entries are the
entries of $Y_{k}$, with zeroes elsewhere, and similarly let $\mathbf{A}_{p}$
be the $n\times n$ matrix whose bottom right $\left( n-p+1\right) \times
\left( n-p+1\right) $ block is $\mathbf{Q}_{p}$, with zeroes elsewhere. Then
we have 
\begin{equation*}
\mathbf{A}\left( x\right) =\sum_{k=1}^{p-1}Z_{k}Z_{k}^{\func{tr}}+\mathbf{A}%
_{p}\left( x\right) ,
\end{equation*}%
which is the claimed formula. Moreover we have the following extension of (%
\ref{Z1 comp}):%
\begin{equation}
ca_{k,k}\mathbf{e}_{k}\otimes \mathbf{e}_{k}\prec Z_{k}\left( x\right)
Z_{k}\left( x\right) ^{\func{tr}}+\sum_{m=k+1}^{n}a_{m,m}\mathbf{e}%
_{m}\otimes \mathbf{e}_{m}\prec C\sum_{m=k}^{n}a_{m,m}\mathbf{e}_{m}\otimes 
\mathbf{e}_{m},\ \ \ \ \ 1\leq k\leq p-1.  \label{Zk comp}
\end{equation}

It remains now to prove that we can further decompose each of the postive
dyads $Z_{k}\left( x\right) Z_{k}\left( x\right) ^{\func{tr}}$ into a finite
sum of squares of $C^{2,\delta }$ vector fields. Let $\mathbf{Q}_{\ell
}\left( x\right) =\left[ q_{k,j}\right] _{k,j=1}^{n}$ and $E_{\ell }\equiv
q_{\ell ,\ell }$. To obtain the sum of squares of $C^{2,\delta }$ vector
fields decomposition of the positive dyad $Z_{k}\left( x\right) Z_{k}\left(
x\right) ^{\func{tr}}$, we will begin with the conclusion of Lemma \ref{off
diagonal}, namely that $E_{k}=q_{k,k}\in C^{4,2\delta }$ and satisfies%
\begin{equation*}
\left\vert D^{\mu }E_{k}\right\vert \lesssim \left( E_{k}\right) ^{\left[
1-\left\vert \mu \right\vert \varepsilon \right] _{+}+\delta ^{\prime }}\leq
\left( E_{k}\right) ^{\delta ^{\prime }}=\left( E_{k}\right) ^{\frac{2\delta
\left( 1+\delta \right) }{2+\delta }},\ \ \ \ \ 1\leq \left\vert \mu
\right\vert \leq 4,\ 2\leq k\leq \ell ,
\end{equation*}%
together with the scalar sum of squares Theorem \ref{efs eps}, and we will
show that there exists a vector function $\mathbf{t}_{k}^{k}\left( x\right)
=\left\{ t_{k,i}^{k}\left( x\right) \right\} _{i=1}^{I}\in C^{2,\delta }$
such that%
\begin{eqnarray}
E_{k}\left( x\right) &=&\left\vert \mathbf{t}_{k}^{k}\left( x\right)
\right\vert ^{2},  \label{st} \\
\left\vert t_{k,i}^{k}\left( x\right) \right\vert &\leq &\left( E_{k}\right)
^{\frac{1}{2}},\ \ \ \left\vert \nabla t_{k,i}^{k}\left( x\right)
\right\vert \lesssim \left( E_{k}\right) ^{\frac{1}{2}\left( \left[
1-2\varepsilon \right] _{+}+\delta ^{\prime \prime \prime }\right) },\ \ \
\left\vert \nabla ^{2}t_{k,i}^{k}\left( x\right) \right\vert \lesssim \left(
E_{k}\right) ^{\frac{\delta ^{2}}{2+\delta }},  \notag
\end{eqnarray}%
where $\delta ^{\prime \prime \prime }=\frac{\delta }{4+2\delta }$.

Now for $1\leq k\leq p-1$, define vector functions $\mathbf{t}_{j}^{k}\left(
x\right) =\left\{ t_{j,i}^{k}\left( x\right) \right\} _{i=1}^{I}$ by%
\begin{equation*}
t_{j,i}^{k}\left( x\right) \equiv t_{k,i}^{k}\left( x\right) \frac{a_{k,j}}{%
E_{k}},\ \ \ \ \ k<j\leq n,
\end{equation*}%
where the functions $t_{k,i}^{k}\left( x\right) $ are given by applying
Theorem \ref{efs eps} to $E_{k}\left( x\right) =\left\vert \mathbf{t}%
_{k}^{k}\left( x\right) \right\vert ^{2}$, and then set%
\begin{equation*}
X_{k,i}\left( x\right) \equiv \sum_{j=k}^{n}t_{j,i}^{k}\left( x\right) 
\mathbf{e}_{j}
\end{equation*}%
so that%
\begin{eqnarray*}
\sum_{i=1}^{I}t_{k,i}^{k}\left( x\right) t_{j,i}^{k}\left( x\right)
&=&\sum_{i=1}^{I}t_{k,i}^{k}\left( x\right) t_{k,i}^{k}\left( x\right) \frac{%
a_{k,j}}{E_{k}}=\frac{\left\vert \mathbf{t}_{k}^{k}\left( x\right)
\right\vert ^{2}}{E_{k}}a_{k,j}=\frac{a_{k,k}a_{k,j}}{E_{k}}, \\
\text{i.e. }\sum_{i=1}^{I}X_{k,i}\left( x\right) \otimes X_{k,i}\left(
x\right) &=&\left( \sum_{j=k}^{n}\frac{a_{k,j}}{\sqrt{E_{k}}}\mathbf{e}%
_{j}\right) \otimes \left( \sum_{j=k}^{n}\frac{a_{k,j}}{\sqrt{E_{k}}}\mathbf{%
e}_{j}\right) =Z_{k}\otimes Z_{k}.
\end{eqnarray*}

At this point we have obtained our decomposition%
\begin{equation*}
\mathbf{A}=\sum_{k=1}^{p-1}\sum_{i=1}^{I}X_{k,i}X_{k,i}^{\func{tr}}+\mathbf{A%
}_{p}
\end{equation*}%
where $\mathbf{A}_{p}=\left[ 
\begin{array}{cc}
\mathbf{0} & \mathbf{0} \\ 
\mathbf{0} & \mathbf{Q}_{p}%
\end{array}%
\right] $ and $\mathbf{Q}_{p}\in C^{4,2\delta }\left( \mathbb{R}^{M}\right) $
is quasiconformal and $\mathbf{Q}_{p}\sim a_{p,p}\mathbb{I}_{\left(
n-p+1\right) \times \left( n-p+1\right) }$. We also have 
\begin{eqnarray*}
Z_{k}\otimes Z_{k} &=&\sum_{i=1}^{I}X_{k,i}X_{k,i}^{\func{tr}}\in
C^{4,2\delta }\left( \mathbb{R}^{M}\right) ,\ \ \ \ \ 1\leq k\leq p-1 \\
ca_{k,k}\mathbf{e}_{k}\otimes \mathbf{e}_{k} &\prec &Z_{k}Z_{k}^{\func{tr}%
}+\sum_{m=k}^{n}a_{m,m}\mathbf{e}_{m}\otimes \mathbf{e}_{m}\prec C.,\ \ \ \
\ 1\leq k\leq p-1.
\end{eqnarray*}

Thus it remains only to show the inequalities in the second line of (\ref{st}%
), and then that $X_{k,i}\in C^{2,\delta }\left( \mathbb{R}^{M}\right) $.

First, $\left\vert t_{k,i}^{k}\left( x\right) \right\vert \leq \left(
E_{k}\right) ^{\frac{1}{2}}$ follows from the definition of $%
t_{k,i}^{k}\left( x\right) $. Second, from part (1) of Theorem \ref{efs eps}
with $\left\vert \alpha \right\vert =2$ we have%
\begin{equation*}
\left\vert \nabla ^{2}t_{k,i}^{k}\left( x\right) \right\vert \lesssim \left(
E_{k}\right) ^{\frac{\delta ^{2}}{2+\delta }}.
\end{equation*}%
Third, from (\ref{diag hyp}) we have $\left\vert \nabla ^{2}E_{k}\right\vert
\lesssim \left( E_{k}\right) ^{\left[ 1-2\varepsilon \right] _{+}+\delta
^{\prime }}$ and so from (\ref{root control}) we conclude%
\begin{eqnarray*}
\left\vert \nabla t_{k,i}^{k}\left( x\right) \right\vert &\lesssim &\rho
_{E_{k}}^{1+\delta }\leq \max \left\{ \left( E_{k}\right) ^{\frac{1+\delta }{%
4+2\delta }},\left\vert \nabla ^{2}E_{k}\right\vert ^{\frac{1+\delta }{%
2+2\delta }}\right\} \\
&\leq &\max \left\{ \left( E_{k}\right) ^{\frac{1+\delta }{4+2\delta }%
},\left( E_{k}\right) ^{\left( \left[ 1-2\varepsilon \right] _{+}+\delta
^{\prime }\right) \frac{1+\delta }{2+2\delta }}\right\} \leq \left(
E_{k}\right) ^{\frac{1}{2}\left( \left[ 1-2\varepsilon \right] _{+}+\delta
^{\prime \prime \prime }\right) },
\end{eqnarray*}%
where $\delta ^{\prime \prime \prime }=\frac{\delta }{4+2\delta }$, and this
completes the proof of the second line in (\ref{st}).

Now we can show that $X_{k,i}\in C^{2,\delta }\left( \mathbb{R}^{M}\right) $
using the product formula (\ref{pdt form}) together with the inequalities in
the second line of (\ref{st}) that we just proved. Indeed, we have%
\begin{equation*}
D^{\mu }t_{j,i}^{k}=D^{\mu }\left( t_{k,i}^{k}a_{k,j}\frac{1}{E_{k}}\right)
=\sum_{\alpha +\beta +\gamma =\mu }c_{\alpha ,\beta ,\gamma }^{\mu }\left(
D^{\alpha }t_{k,i}^{k}\right) \left( D^{\beta }a_{k,j}\right) \left(
D^{\gamma }\frac{1}{E_{k}}\right) ,
\end{equation*}%
where 
\begin{eqnarray*}
\left\vert t_{k,i}^{k}\left( x\right) \right\vert &\leq &\left( E_{k}\right)
^{\frac{1}{2}},\ \ \ \left\vert \nabla t_{k,i}^{k}\left( x\right)
\right\vert \lesssim \left( E_{k}\right) ^{\frac{1}{2}\left( \left[
1-2\varepsilon \right] _{+}+\delta ^{\prime \prime \prime }\right) },\ \ \
\left\vert \nabla ^{2}t_{k,i}^{k}\left( x\right) \right\vert \lesssim \left(
E_{k}\right) ^{\frac{\delta ^{2}}{2+\delta }} \\
\left\vert D^{\beta }a_{k,j}\right\vert &\lesssim &\left( E_{j}\right) ^{ 
\left[ \frac{1}{2}+\left( 2-\left\vert \beta \right\vert \right) \varepsilon %
\right] _{+}+\delta ^{\prime \prime }},\ \ \ \ \ 1\leq k<j\leq p-1, \\
\left\vert D^{\beta }a_{k,j}\right\vert &\lesssim &\left( E_{k}\right) ^{ 
\left[ \frac{1}{2}+\left( 2-\left\vert \beta \right\vert \right) \varepsilon %
\right] _{+}+\delta ^{\prime \prime }},\ \ \ \ \ 1\leq k\leq p-1<j, \\
\left\vert D^{\gamma }\frac{1}{E_{k}}\right\vert &\lesssim &\left\vert
E_{k}\right\vert ^{-1-\left\vert \gamma \right\vert \varepsilon +\left\vert
\gamma \right\vert \delta ^{\prime }},
\end{eqnarray*}%
for $0\leq \left\vert \alpha \right\vert ,\left\vert \beta \right\vert
,\left\vert \gamma \right\vert \leq 2$. We now have the following estimates
when $1\leq k<j\leq p-1$, where we treat the cases $\left\vert \alpha
\right\vert =0,1,2$ separately due to the unorthodox form of the estimates
for $\left\vert \nabla ^{2}t_{k,i}^{k}\left( x\right) \right\vert $. The
presence of $\delta ^{\prime \prime }>0$ will now play a crucial role.

When $\left\vert \alpha \right\vert =0$ we have 
\begin{eqnarray*}
\left\vert \left( D^{\alpha }t_{k,i}^{k}\right) \left( D^{\beta
}a_{k,j}\right) \left( D^{\gamma }\frac{1}{E_{k}}\right) \right\vert
&\lesssim &\left( E_{k}\right) ^{\frac{1}{2}}\ \left( E_{j}\right) ^{\left[ 
\frac{1}{2}+\left( 2-\left\vert \beta \right\vert \right) \varepsilon \right]
_{+}+\delta ^{\prime \prime }}\ \left\vert E_{k}\right\vert ^{-1-\left\vert
\gamma \right\vert \varepsilon +\left\vert \gamma \right\vert \delta
^{\prime }} \\
&\lesssim &\left( E_{k}\right) ^{\frac{1}{2}+\left[ \frac{1}{2}+\left(
2-\left\vert \beta \right\vert \right) \varepsilon \right] _{+}+\delta
^{\prime \prime }-1-\left\vert \gamma \right\vert \varepsilon +\left\vert
\gamma \right\vert \delta ^{\prime }},
\end{eqnarray*}%
which is bounded because the exponent of $E_{k}$ is 
\begin{eqnarray*}
&&\frac{1}{2}+\frac{1}{2}+\left( 2-\left\vert \beta \right\vert \right)
\varepsilon +\delta ^{\prime \prime }-1-\left\vert \gamma \right\vert
\varepsilon +\left\vert \gamma \right\vert \delta ^{\prime } \\
&=&\left( 2-\left\vert \beta \right\vert -\left\vert \gamma \right\vert
\right) \varepsilon +\left\vert \gamma \right\vert \delta ^{\prime }+\delta
^{\prime \prime }=\delta ^{\prime \prime }+\left\vert \gamma \right\vert
\delta ^{\prime }\geq \delta ^{\prime \prime }.
\end{eqnarray*}

When $\left\vert \alpha \right\vert =1$ we have the estimate%
\begin{equation*}
\left\vert \left( D^{\alpha }t_{k,i}^{k}\right) \left( D^{\beta
}a_{k,j}\right) \left( D^{\gamma }\frac{1}{E_{k}}\right) \right\vert
\lesssim \left( E_{k}\right) ^{\left( \left[ 1-2\varepsilon \right]
_{+}+\delta ^{\prime \prime \prime }\right) \frac{1+\delta }{2+2\delta }}\
\left( E_{k}\right) ^{\left[ \frac{1}{2}+\varepsilon \right] _{+}+\delta
^{\prime \prime }}\ \left( E_{k}\right) ^{-1},
\end{equation*}%
since the worst case is when $\left\vert \beta \right\vert =1$. But this is
bounded since%
\begin{equation*}
\left( \left[ 1-2\varepsilon \right] _{+}+\delta ^{\prime \prime \prime
}\right) \frac{1}{2}+\left[ \frac{1}{2}+\varepsilon \right] _{+}+\delta
^{\prime \prime }-1\geq \delta ^{\prime \prime }+\frac{\delta ^{\prime
\prime \prime }}{2},\ \ \ \ \ \text{for }\frac{1}{4}\leq \varepsilon <1.
\end{equation*}%
Indeed, this is clear when $\varepsilon \geq \frac{1}{2}$, and when $\frac{1%
}{4}\leq \varepsilon <\frac{1}{2}$, we have%
\begin{eqnarray*}
&&\left( \left[ 1-2\varepsilon \right] _{+}+\delta ^{\prime \prime \prime
}\right) \frac{1}{2}+\left[ \frac{1}{2}+\varepsilon \right] _{+}+\delta
^{\prime \prime }-1 \\
&=&\delta ^{\prime \prime }+\left( 1-2\varepsilon +\delta ^{\prime \prime
\prime }\right) \frac{1}{2}+\frac{1}{2}+\varepsilon -1=\delta ^{\prime
\prime }+\frac{\delta ^{\prime \prime \prime }}{2}.
\end{eqnarray*}

When $\left\vert \alpha \right\vert =2$ we have the estimate%
\begin{eqnarray*}
\left\vert \left( D^{\alpha }t_{k,i}^{k}\right) \left( D^{\beta
}a_{k,j}\right) \left( D^{\gamma }\frac{1}{E_{k}}\right) \right\vert
&\lesssim &\left( E_{k}\right) ^{\frac{\delta ^{2}}{2+\delta }}\left(
E_{k}\right) ^{\left[ \frac{1}{2}+2\varepsilon \right] _{+}+\delta ^{\prime
\prime }}\left( E_{k}\right) ^{-1} \\
&=&\left( E_{k}\right) ^{\frac{\delta ^{2}}{2+\delta }+\frac{1}{2}%
+2\varepsilon +\delta ^{\prime \prime }-1}\leq \left( E_{k}\right) ^{\delta
^{\prime \prime }+\frac{\delta ^{2}}{2+\delta }},
\end{eqnarray*}%
when $\varepsilon \geq \frac{1}{4}$.

Using the subproduct and subchain rules for the functional $\left[ h\right]
_{\mu ,\delta ^{\prime \prime }}$ in $\mathbb{R}^{M}$ as in (\ref{def mod D}%
) above, we claim that%
\begin{equation*}
\left[ t_{j,i}^{k}\right] _{\mu ,\delta ^{\prime \prime }}\lesssim 1\ \ \ \
\ \text{for }\left\vert \mu \right\vert =2.
\end{equation*}%
To see this, consider the "worst" expression above, 
\begin{equation*}
H\equiv \left( t_{k,i}^{k}\right) \left( D^{\beta }a_{k,j}\right) \frac{1}{%
E_{k}},\ \ \ \ \text{where }\left\vert \beta \right\vert =2,
\end{equation*}%
in which the exponent of $E_{k}$ in the estimate for $H$ vanishes if $\delta
^{\prime \prime }=0$.

\textbf{Case }$E_{k}\left( y\right) \geq E_{k}\left( z\right) $: In this
case we write%
\begin{eqnarray*}
H\left( y\right) -H\left( z\right) &=&\frac{t_{k,i}^{k}\left( y\right) }{%
E_{k}\left( y\right) }D^{\beta }a_{k,j}\left( y\right) -\frac{%
t_{k,i}^{k}\left( z\right) }{E_{k}\left( z\right) }D^{\beta }a_{k,j}\left(
z\right) \\
&=&\frac{t_{k,i}^{k}\left( y\right) }{E_{k}\left( y\right) }\left( D^{\beta
}a_{k,j}\left( y\right) -D^{\beta }a_{k,j}\left( z\right) \right) +\left( 
\frac{t_{k,i}^{k}\left( y\right) }{E_{k}\left( y\right) }-\frac{%
t_{k,i}^{k}\left( z\right) }{E_{k}\left( z\right) }\right) D^{\beta
}a_{k,j}\left( z\right) \\
&\equiv &I\left( y,z\right) +II\left( y,z\right) .
\end{eqnarray*}%
We estimate term $I\left( y,z\right) $ by considering two cases separately
depending on the separation between $y$ and $z$.

\textbf{Subcase }$\left\vert y-z\right\vert ^{\delta }\geq E_{k}\left(
y\right) ^{\delta ^{\prime \prime }}$: We estimate crudely in this case to
obtain%
\begin{equation*}
\frac{\left\vert I\left( y,z\right) \right\vert }{\left\vert y-z\right\vert
^{\delta }}\lesssim \frac{E_{k}\left( y\right) ^{\frac{1}{2}}}{E_{k}\left(
y\right) }\frac{E_{k}\left( y\right) ^{\frac{1}{2}+\delta ^{\prime \prime
}}+E_{k}\left( z\right) ^{\frac{1}{2}+\delta ^{\prime \prime }}}{E_{k}\left(
y\right) ^{\delta ^{\prime \prime }}}\lesssim \frac{E_{k}\left( y\right)
^{\delta ^{\prime \prime }}}{E_{k}\left( y\right) ^{\delta ^{\prime \prime }}%
}=1.
\end{equation*}%
On the other hand,%
\begin{eqnarray*}
\frac{\left\vert II\left( y,z\right) \right\vert }{\left\vert y-z\right\vert
^{\delta }} &\lesssim &\left( \left\vert \frac{t_{k,i}^{k}\left( y\right) }{%
E_{k}\left( y\right) }\right\vert +\left\vert \frac{t_{k,i}^{k}\left(
z\right) }{E_{k}\left( z\right) }\right\vert \right) \frac{\left\vert
D^{\beta }a_{k,j}\left( z\right) \right\vert }{E_{k}\left( y\right) ^{\delta
^{\prime \prime }}} \\
&\lesssim &\left( \frac{1}{E_{k}\left( y\right) ^{\frac{1}{2}}}+\frac{1}{%
E_{k}\left( z\right) ^{\frac{1}{2}}}\right) \frac{E_{k}\left( z\right) ^{%
\frac{1}{2}+\delta ^{\prime \prime }}}{E_{k}\left( y\right) ^{\delta
^{\prime \prime }}} \\
&\lesssim &\left( \frac{1}{E_{k}\left( z\right) ^{\frac{1}{2}}}\right) \frac{%
E_{k}\left( z\right) ^{\frac{1}{2}+\delta ^{\prime \prime }}}{E_{k}\left(
y\right) ^{\delta ^{\prime \prime }}}\leq 1.
\end{eqnarray*}

\textbf{Subcase }$\left\vert y-z\right\vert ^{\delta }<E_{k}\left( y\right)
^{\delta ^{\prime \prime }}$: In this case we apply the submean value
theorem to the difference $D^{\beta }a_{k,j}\left( y\right) -D^{\beta
}a_{k,j}\left( z\right) $ to obtain the estimate%
\begin{equation*}
\left\vert D^{\beta }a_{k,j}\left( y\right) -D^{\beta }a_{k,j}\left(
z\right) \right\vert \leq \left\vert D^{\gamma }a_{k,j}\left( \left(
1-\theta \right) y+\theta z\right) \right\vert \left\vert y-z\right\vert
\lesssim \left\vert y-z\right\vert ,
\end{equation*}%
since $\left\vert \gamma \right\vert =3$ implies $\left\vert D^{\gamma
}a_{k,j}\right\vert $ is bounded. Moreover, the submean value theorem
applied to $E_{k}$ yields%
\begin{eqnarray*}
\left\vert E_{k}\left( y\right) -E_{k}\left( z\right) \right\vert &\leq
&\left\vert DE_{k}\left( \left( 1-\theta \right) y+\theta z\right)
\right\vert \left\vert y-z\right\vert \lesssim \left\vert y-z\right\vert \\
&\leq &E_{k}\left( y\right) ^{\frac{\delta ^{\prime \prime }}{\delta }},
\end{eqnarray*}%
and since $\delta ^{\prime \prime }>\delta $, we conclude that $E_{k}\left(
y\right) \approx E_{k}\left( z\right) $ for $y$ and $z$ sufficiently close
to the origin, depending on the ratio $\frac{\delta ^{\prime \prime }}{%
\delta }>1$. Plugging all of this into the estimate for $\frac{\left\vert
I\left( y,z\right) \right\vert }{\left\vert y-z\right\vert ^{\delta }}$ gives%
\begin{equation*}
\frac{\left\vert I\left( y,z\right) \right\vert }{\left\vert y-z\right\vert
^{\delta }}\lesssim \frac{E_{k}\left( y\right) ^{\frac{1}{2}}}{E_{k}\left(
y\right) }\left\vert y-z\right\vert ^{1-\delta }\lesssim \frac{E_{k}\left(
y\right) ^{\frac{1}{2}}}{E_{k}\left( y\right) }\left( E_{k}\left( y\right)
^{\delta ^{\prime \prime }}\right) ^{\frac{1-\delta }{\delta }}\lesssim
E_{k}\left( y\right) ^{\delta ^{\prime \prime }\frac{1-\delta }{\delta }-%
\frac{1}{2}},
\end{equation*}%
which is bounded because $\delta ^{\prime \prime }>\delta $ implies $\delta
^{\prime \prime }\frac{1-\delta }{\delta }-\frac{1}{2}>1-\delta -\frac{1}{2}%
\geq 0$ if $0<\delta \leq \frac{1}{2}$. On the other hand,%
\begin{eqnarray*}
\frac{\left\vert II\left( y,z\right) \right\vert }{\left\vert y-z\right\vert
^{\delta }} &\lesssim &\left\vert \frac{t_{k,i}^{k}\left( y\right) }{%
E_{k}\left( y\right) }-\frac{t_{k,i}^{k}\left( z\right) }{E_{k}\left(
z\right) }\right\vert \frac{\left\vert D^{\beta }a_{k,j}\left( z\right)
\right\vert }{\left\vert y-z\right\vert ^{\delta }} \\
&\lesssim &\left\vert \frac{E_{k}\left( z\right) t_{k,i}^{k}\left( y\right)
\pm E_{k}\left( z\right) t_{k,i}^{k}\left( z\right) -t_{k,i}^{k}\left(
z\right) E_{k}\left( y\right) }{E_{k}\left( y\right) E_{k}\left( z\right) }%
\right\vert \frac{E_{k}\left( z\right) ^{\frac{1}{2}+\delta ^{\prime \prime
}}}{\left\vert y-z\right\vert ^{\delta }} \\
&\lesssim &\left( \frac{E_{k}\left( z\right) \left\vert t_{k,i}^{k}\left(
y\right) -t_{k,i}^{k}\left( z\right) \right\vert +\left\vert E_{k}\left(
z\right) -E_{k}\left( y\right) \right\vert \left\vert t_{k,i}^{k}\left(
z\right) \right\vert }{E_{k}\left( y\right) E_{k}\left( z\right) }\right) 
\frac{E_{k}\left( z\right) ^{\frac{1}{2}+\delta ^{\prime \prime }}}{%
\left\vert y-z\right\vert ^{\delta }}
\end{eqnarray*}%
and so using the mean value theorem, 
\begin{eqnarray*}
\frac{\left\vert II\left( y,z\right) \right\vert }{\left\vert y-z\right\vert
^{\delta }} &\lesssim &\left( \frac{E_{k}\left( z\right) \left\vert
y-z\right\vert +\left\vert y-z\right\vert E_{k}\left( z\right) ^{\frac{1}{2}}%
}{E_{k}\left( y\right) E_{k}\left( z\right) }\right) \frac{E_{k}\left(
z\right) ^{\frac{1}{2}+\delta ^{\prime \prime }}}{\left\vert y-z\right\vert
^{\delta }} \\
&\lesssim &\left( \frac{\left\vert y-z\right\vert E_{k}\left( z\right) ^{%
\frac{1}{2}}}{E_{k}\left( y\right) E_{k}\left( z\right) }\right) \frac{%
E_{k}\left( z\right) ^{\frac{1}{2}+\delta ^{\prime \prime }}}{\left\vert
y-z\right\vert ^{\delta }}\approx \left\vert y-z\right\vert ^{1-\delta
}E_{k}\left( y\right) ^{\delta ^{\prime \prime }-1} \\
&=&\left\vert y-z\right\vert ^{\delta \left( \frac{1-\delta }{\delta }%
\right) }E_{k}\left( y\right) ^{\delta ^{\prime \prime }-1}<E_{k}\left(
y\right) ^{\delta ^{\prime \prime }\left( \frac{1-\delta }{\delta }\right)
}E_{k}\left( y\right) ^{\delta ^{\prime \prime }-1}\lesssim 1,
\end{eqnarray*}%
since $\left\vert y-z\right\vert ^{\delta }<E_{k}\left( y\right) ^{\delta
^{\prime \prime }}$ and $\delta ^{\prime \prime }>\delta $.

\textbf{Case }$E_{k}\left( z\right) ^{\delta ^{\prime \prime }}\geq
E_{k}\left( y\right) ^{\delta ^{\prime \prime }}$: This case is similar to
the previous case.

Finally, if in addition $\mathbf{A}$ is subordinate, then $\mathbf{Q}_{p}$
is subordinate by Lemma \ref{comp and sub}. This completes the proof of
Theorem \ref{final n Grushin}.
\end{proof}

\section{Counterexamples for sums of squares of matrix functions}

Consider the $3\times 3$ matrix of quadratic homogeneous polynomials in
three variables as in \cite[Example 6]{HiNi},%
\begin{equation*}
\mathbf{Q}\left( x,y,z\right) \equiv \left[ 
\begin{array}{ccc}
x^{2}+2z^{2} & -xy & -xz \\ 
-xy & y^{2}+2x^{2} & -yz \\ 
-xz & -yz & z^{2}+2y^{2}%
\end{array}%
\right] .
\end{equation*}%
The dehomogenization of this form is%
\begin{equation*}
\mathbf{Q}\left( x,y\right) \equiv \left[ 
\begin{array}{ccc}
x^{2}+2 & -xy & -x \\ 
-xy & y^{2}+2x^{2} & -y \\ 
-x & -y & 1+2y^{2}%
\end{array}%
\right] .
\end{equation*}%
It was shown in \cite{Cho} that $\mathbf{Q}$ is not a sum of squares of
polynomials. However, the quadratic matrix form $\mathbf{Q}\left(
x,y,z\right) $ is not elliptical since its determinant vanishes on the union
of the three coordinate axes, 
\begin{equation*}
\det \mathbf{Q}\left( x,y,z\right) =4\left(
x^{4}y^{2}+y^{4}z^{2}+z^{4}x^{2}+x^{2}y^{2}z^{2}\right) .
\end{equation*}%
Here we modify $\mathbf{Q}\left( x,y,z\right) $ so that it is diagonally
elliptical, i.e. its determinant vanishes only at the origin, yet still
cannot be written as a sum of squares.

\subsection{A positive quadratic matrix form that is not a sum of squares of
forms}

We prove the following theorem.

\begin{theorem}
\label{Cho's thm}If $0<\lambda <\frac{2}{81}$, then the quadratic matrix
form 
\begin{equation*}
\mathbf{Q}_{\lambda }\left( x,y,z\right) \equiv \left[ 
\begin{array}{ccc}
x^{2}+\lambda y^{2}+2z^{2} & -xy & -xz \\ 
-xy & y^{2}+\lambda z^{2}+2x^{2} & -yz \\ 
-xz & -yz & z^{2}+\lambda x^{2}+2y^{2}%
\end{array}%
\right]
\end{equation*}%
is both positive definite away from the origin, and \emph{not} a sum of
squares of linear matrix polynomials.
\end{theorem}

\begin{proof}
Suppose $0<\lambda <\frac{2}{81}$. The quadratic matrix form $\mathbf{Q}%
_{\lambda }\left( x,y,z\right) $ is positive definite for all $\left(
x,y,z\right) \neq \left( 0,0,0\right) $. Indeed, for $\lambda >0$, the top
left entry of $\mathbf{Q}_{\lambda }\left( x,y,z\right) $ satisfies 
\begin{equation*}
x^{2}+\lambda y^{2}+2z^{2}\geq \min \left\{ \lambda ,1\right\} \left(
x^{2}+y^{2}+z^{2}\right) ,
\end{equation*}%
the top left $2\times 2$ minor of $\mathbf{Q}_{\lambda }\left( x,y,z\right) $
satisfies%
\begin{eqnarray*}
&&\det \left[ 
\begin{array}{cc}
x^{2}+\lambda y^{2}+2z^{2} & -xy \\ 
-xy & y^{2}+\lambda z^{2}+2x^{2}%
\end{array}%
\right] \\
&=&\lambda ^{2}\left\{ y^{2}+z^{2}\right\} +\lambda \left\{ z^{2}\left(
x^{2}+2z^{2}\right) +y^{2}\left( y^{2}+2x^{2}\right) \right\} \\
&&+\left( x^{2}+2z^{2}\right) \left( y^{2}+2x^{2}\right) -x^{2}y^{2} \\
&\geq &2x^{4}+\lambda \left( y^{4}+2z^{4}\right) \geq \min \left\{ \lambda
,2\right\} \left( x^{4}+y^{4}+z^{4}\right) ,
\end{eqnarray*}%
and the determinant of $\mathbf{Q}_{\lambda }\left( x,y,z\right) $ satisfies%
\begin{eqnarray*}
\det \mathbf{Q}_{\lambda }\left( x,y,z\right) &=&\lambda ^{3}\left(
x^{2}y^{2}z^{2}\right) +\lambda ^{2}\left\{ 2\left(
x^{2}z^{4}+z^{2}y^{4}+y^{2}x^{4}\right)
+x^{2}y^{4}+y^{2}z^{4}+z^{2}x^{4}\right\} \\
&&+2\lambda \left\{ x^{6}+y^{6}+z^{6}+2\left(
x^{2}y^{4}+y^{2}z^{4}+z^{2}x^{4}\right) +3x^{2}y^{2}z^{2}\right\} \\
&&+4\left( x^{2}z^{4}+z^{2}y^{4}+y^{2}x^{4}+x^{2}y^{2}z^{2}\right) \\
&\geq &2\lambda \left( x^{6}+y^{6}+z^{6}\right) .
\end{eqnarray*}

Now assume, in order to derive a contradiction, that the dehomogenization 
\begin{equation*}
\mathbf{Q}_{\lambda }\left( x,y\right) \equiv \mathbf{Q}_{\lambda }\left(
x,y,1\right) =\left[ 
\begin{array}{ccc}
x^{2}+\lambda y^{2}+2 & -xy & -x \\ 
-xy & y^{2}+\lambda +2x^{2} & -y \\ 
-x & -y & 1+\lambda x^{2}+2y^{2}%
\end{array}%
\right]
\end{equation*}%
is a sum of squares of dyads of degree one, i.e.%
\begin{equation*}
\mathbf{Q}_{\lambda }\left( x,y\right) =\sum_{\ell =1}^{L}\mathbf{v}^{\ell
}\left( x,y\right) \otimes \mathbf{v}^{\ell }\left( x,y\right) =\sum_{\ell
=1}^{L}\left[ 
\begin{array}{ccc}
m_{11}^{\ell } & m_{12}^{\ell } & m_{13}^{\ell } \\ 
m_{21}^{\ell } & m_{22}^{\ell } & m_{23}^{\ell } \\ 
m_{31}^{\ell } & m_{32}^{\ell } & m_{33}^{\ell }%
\end{array}%
\right] \left( 
\begin{array}{c}
x \\ 
y \\ 
1%
\end{array}%
\right) \left( 
\begin{array}{ccc}
x & y & 1%
\end{array}%
\right) \left[ 
\begin{array}{ccc}
m_{11}^{\ell } & m_{21}^{\ell } & m_{31}^{\ell } \\ 
m_{12}^{\ell } & m_{22}^{\ell } & m_{32}^{\ell } \\ 
m_{13}^{\ell } & m_{23}^{\ell } & m_{33}^{\ell }%
\end{array}%
\right] .
\end{equation*}%
Then with $\mathbf{m}_{ij}\equiv \left( m_{ij}^{\ell }\right) _{\ell =1}^{L}$%
, we have upon equating the two formulas for $\mathbf{Q}_{\lambda }\left(
x,y\right) $ that%
\begin{eqnarray*}
\left\vert \mathbf{m}_{11}\right\vert ^{2} &=&\left\vert \mathbf{m}%
_{22}\right\vert ^{2}=\left\vert \mathbf{m}_{33}\right\vert ^{2}=1, \\
\left\vert \mathbf{m}_{13}\right\vert ^{2} &=&\left\vert \mathbf{m}%
_{21}\right\vert ^{2}=\left\vert \mathbf{m}_{32}\right\vert ^{2}=2, \\
\left\vert \mathbf{m}_{12}\right\vert ^{2} &=&\left\vert \mathbf{m}%
_{23}\right\vert ^{2}=\left\vert \mathbf{m}_{31}\right\vert ^{2}=\lambda ,
\end{eqnarray*}%
and%
\begin{equation*}
\mathbf{m}_{22}\cdot \mathbf{m}_{11}+\mathbf{m}_{21}\cdot \mathbf{m}_{12}=%
\mathbf{m}_{33}\cdot \mathbf{m}_{11}+\mathbf{m}_{31}\cdot \mathbf{m}_{13}=%
\mathbf{m}_{33}\cdot \mathbf{m}_{22}+\mathbf{m}_{32}\cdot \mathbf{m}_{23}=-1.
\end{equation*}%
Thus we conclude that%
\begin{eqnarray*}
\left\vert \mathbf{m}_{11}+\mathbf{m}_{22}\right\vert ^{2} &=&\left\vert 
\mathbf{m}_{11}\right\vert ^{2}+2\mathbf{m}_{11}\cdot \mathbf{m}%
_{22}+\left\vert \mathbf{m}_{22}\right\vert ^{2}=2+2\left( -1-\mathbf{m}%
_{21}\cdot \mathbf{m}_{12}\right) =-2\mathbf{m}_{21}\cdot \mathbf{m}_{12}, \\
\left\vert \mathbf{m}_{22}+\mathbf{m}_{33}\right\vert ^{2} &=&\left\vert 
\mathbf{m}_{22}\right\vert ^{2}+2\mathbf{m}_{22}\cdot \mathbf{m}%
_{33}+\left\vert \mathbf{m}_{33}\right\vert ^{2}=2+2\left( -1-\mathbf{m}%
_{32}\cdot \mathbf{m}_{23}\right) =-2\mathbf{m}_{32}\cdot \mathbf{m}_{23}, \\
\left\vert \mathbf{m}_{33}+\mathbf{m}_{11}\right\vert ^{2} &=&\left\vert 
\mathbf{m}_{33}\right\vert ^{2}+2\mathbf{m}_{33}\cdot \mathbf{m}%
_{11}+\left\vert \mathbf{m}_{11}\right\vert ^{2}=2+2\left( -1-\mathbf{m}%
_{31}\cdot \mathbf{m}_{13}\right) =-2\mathbf{m}_{31}\cdot \mathbf{m}_{13},
\end{eqnarray*}%
and hence%
\begin{eqnarray*}
4 &=&\left\vert 2\mathbf{m}_{11}\right\vert ^{2}=\left\vert \left( \mathbf{m}%
_{11}+\mathbf{m}_{22}\right) -\left( \mathbf{m}_{22}+\mathbf{m}_{33}\right)
+\left( \mathbf{m}_{33}+\mathbf{m}_{11}\right) \right\vert ^{2} \\
&\leq &3\left( \left\vert \mathbf{m}_{11}+\mathbf{m}_{22}\right\vert
^{2}+\left\vert \mathbf{m}_{22}+\mathbf{m}_{33}\right\vert ^{2}+\left\vert 
\mathbf{m}_{33}+\mathbf{m}_{11}\right\vert ^{2}\right) \\
&=&-6\left( \mathbf{m}_{21}\cdot \mathbf{m}_{12}+\mathbf{m}_{32}\cdot 
\mathbf{m}_{23}+\mathbf{m}_{31}\cdot \mathbf{m}_{13}\right) \\
&\leq &6\left( \sqrt{2}\sqrt{\lambda }+\sqrt{2}\sqrt{\lambda }+\sqrt{\lambda 
}\sqrt{2}\right) =18\sqrt{2\lambda }<4,
\end{eqnarray*}%
if $0<\lambda <\frac{2}{81}$, which is the desired contradiction.
\end{proof}

\subsection{A matrix-valued smooth function not a finite sum of vector $C^{1,%
\protect\alpha }$ squares}

Now suppose that $\mathbf{P}_{\lambda }\left( x,y,z\right) =\mathbf{Q}%
_{\lambda }\left( x,y,z\right) +O\left( r^{2+\alpha }\right) $, where $r=%
\sqrt{x^{2}+y^{2}+z^{2}}$. Then if $\mathbf{P}_{\lambda }$ is a sum of $%
C^{1,\alpha }$ squares,%
\begin{equation*}
\mathbf{P}_{\lambda }\left( x,y,z\right) =\sum_{\ell =1}^{L}\mathbf{u}^{\ell
}\left( x,y,z\right) \otimes \mathbf{u}^{\ell }\left( x,y,z\right) ,\ \ \ \
\ \text{where }\mathbf{u}^{\ell }\left( x,y,z\right) \in C^{1,\alpha },
\end{equation*}%
Taylor's theorem shows that%
\begin{equation*}
\mathbf{u}^{\ell }\left( x,y,z\right) =\mathbf{v}^{\ell }\left( x,y,z\right)
+O\left( r^{1+\alpha }\right) ,\ \ \ \ \ \text{where }\mathbf{v}^{\ell }%
\text{ is a linear form},
\end{equation*}%
and so 
\begin{eqnarray*}
\mathbf{Q}_{\lambda }\left( x,y,z\right) +O\left( r^{2+\alpha }\right) &=&%
\mathbf{P}_{\lambda }\left( x,y,z\right) \\
&=&\sum_{\ell =1}^{L}\left( \mathbf{v}^{\ell }\left( x,y,z\right) +O\left(
r^{1+\alpha }\right) \right) \otimes \left( \mathbf{v}^{\ell }\left(
x,y,z\right) +O\left( r^{1+\alpha }\right) \right) \\
&=&\sum_{\ell =1}^{L}\mathbf{v}^{\ell }\left( x,y,z\right) \otimes \mathbf{v}%
^{\ell }\left( x,y,z\right) +O\left( r^{2+\alpha }\right) ,\ \ \ \ \ \text{%
since }\mathbf{v}^{\ell }\left( x,y,z\right) =O\left( r\right) ,
\end{eqnarray*}%
implies that%
\begin{equation*}
\mathbf{Q}_{\lambda }\left( x,y,z\right) =\sum_{\ell =1}^{L}\mathbf{v}^{\ell
}\left( x,y,z\right) \otimes \mathbf{v}^{\ell }\left( x,y,z\right) ,
\end{equation*}%
the desired contradiction. Since $\mathbf{Q}_{\lambda }\left( x,y,z\right) $
is obviously subordinate, we have established the following.

\begin{theorem}
\label{C 1 delta}There is a subordinate, diagonally elliptical $3\times 3$
matrix-valued quadratic polynomial function of three variables, e.g. $%
\mathbf{Q}_{\lambda }\left( x,y,z\right) $ with $0<\lambda <\frac{2}{81}$,
that cannot be written as a finite sum of squares of $C^{1,\delta }$ vector
functions for any $\delta >0$.
\end{theorem}

This conclusion shows in particular that a smooth matrix-valued function,
comparable to the identity matrix, need not have even a $C^{1,\delta }$ sum
of squares representation, in stark contrast to the scalar case where the
Fefferman-Phong theorem shows that a $C^{1,1}$ sum of squares representation
always holds.

However, the hypotheses $\left\vert D^{\mu }a_{k,k}\left( w\right)
\right\vert \lesssim a_{k,k}\left( w\right) ^{\left[ 1-\left\vert \mu
\right\vert \varepsilon \right] _{+}+\delta ^{\prime }}$ for some $\delta
^{\prime }>0$, on the \emph{diagonal} entries $a_{k,k}\left( w\right) $ in
Theorem \ref{final n Grushin} with $n=3$ and $p=4$, are \emph{not} satisfied
by the matrix function $\mathbf{Q}_{\lambda }\left( x,y,z\right) =\mathbf{Q}%
_{\lambda }\left( w\right) $ with $w=\left( x,y,z\right) $, since $\frac{%
\partial ^{2}}{\partial w_{k}^{2}}a_{k,k}\left( w\right) =2$ and $%
a_{k,k}\left( w\right) ^{\left[ 1-\left\vert \mu \right\vert \varepsilon %
\right] _{+}+\delta ^{\prime }}\lesssim \left\vert w\right\vert ^{\delta
^{\prime }}$. We now turn to constructing a counterexample in Theorem \ref%
{flat counter} below, where only the off diagonal inequalities fail to hold,
which will show that in order to conclude that there is a representation as
a sum of squares of $C^{2,\delta }$ vectors, it is necessary to impose
additional conditions on the \emph{off-diagonal} entries, such as we have
done in Theorem \ref{final n Grushin}.

\subsection{The flat elliptical case}

We will prove in Lemma \ref{failure} below that there is a positive constant 
$C_{\beta }$ such that if 
\begin{equation*}
\psi \left( t\right) \leq C_{\beta }\varphi \left( t\right) ^{\frac{2}{\beta 
}}t^{\frac{4}{\beta }},\ \ \ \ \text{for all sufficiently small }\left\vert
t\right\vert ,
\end{equation*}%
then $\mathbf{F}_{\varphi ,\psi }$ \emph{cannot} be written as a finite sum
of squares of $C^{1,\beta }$ vector fields.

To match notation with that used in the paper \cite{KoSa1}, we fix $%
0<\lambda <\frac{2}{81}$, and for $W=\left( x,y,z\right) \in \mathbb{R}^{3}$%
, set 
\begin{equation*}
\mathbf{L}\left( W\right) =\mathbf{L}\left( x,y,z\right) \equiv \mathbf{Q}%
_{\lambda }\left( x,y,z\right) =\left[ 
\begin{array}{ccc}
x^{2}+\lambda y^{2}+2z^{2} & -xy & -xz \\ 
-xy & y^{2}+\lambda z^{2}+2x^{2} & -yz \\ 
-xz & -yz & z^{2}+\lambda x^{2}+2y^{2}%
\end{array}%
\right] ,
\end{equation*}%
so that $\mathbf{L}\left( W\right) \sim \left\vert W\right\vert ^{2}\mathbb{I%
}_{3}$. We now recall some of the constructions in \cite{KoSa1}, but in the
context of $\mathbb{R}^{3}$ here, rather than $\mathbb{R}^{4}$ as was done
in \cite{KoSa1}. For a modulus of continuity $\omega $, and $h$ defined on
the unit ball $B_{\mathbb{R}^{3}}\left( 0,1\right) $ in $\mathbb{R}^{4}$, 
\begin{equation*}
\left\Vert h\right\Vert _{C^{1,\omega }\left( B_{\mathbb{R}^{3}}\left(
0,1\right) \right) }\equiv \left\Vert h\right\Vert _{C^{0}\left( B_{\mathbb{R%
}^{3}}\left( 0,1\right) \right) }+\left\Vert \nabla h\right\Vert
_{C^{0}\left( B_{\mathbb{R}^{3}}\left( 0,1\right) \right)
}+\sup_{W,W^{\prime }\in B_{\mathbb{R}^{3}}\left( 0,1\right) }\frac{%
\left\vert \nabla h\left( W\right) -\nabla h\left( W^{\prime }\right)
\right\vert }{\omega \left( \left\vert W-W^{\prime }\right\vert \right) }.
\end{equation*}

Given $\tau >0$, define%
\begin{equation*}
\mathcal{C}_{1,\omega }^{\nu }\left( \tau \right) \equiv \inf \left\{
\left\Vert \mathcal{G}\right\Vert _{1,\omega }:\mathcal{G}\in \oplus ^{\nu
}C^{1,\omega }\left( B_{\mathbb{R}^{3}}\left( 0,1\right) \right) \text{ and }%
\mathbf{L}\left( W\right) +\tau =\sum_{\ell =1}^{\nu }G_{\ell }\left(
W\right) G_{\ell }\left( W\right) ^{\func{tr}}\right\} .
\end{equation*}%
Note that this expression differs from that in \cite{KoSa1} by using the
smoothness space $C^{1,\omega }$ in place of $C^{2,\omega }$, as was used in 
\cite{KoSa1}. The reason is that the matrix form $\mathbf{L}\left(
x,y,z\right) $ is homogenous of degree $2$ here, whereas the function form $%
L\left( w,x,y,z\right) $ was homogenous of degree $4$ in \cite{KoSa1}. We
will need a crucial lower bound in the case $\omega \left( s\right) =\omega
_{\beta }\left( s\right) =s^{\beta }$. For this we define%
\begin{equation}
\delta _{\nu }^{2}\equiv \inf_{\left\{ S_{\ell }\right\} _{\ell =1}^{\nu
}}\inf_{W\in \mathbb{S}^{2}}\left( \mathbf{L}\left( W\right) -\sum_{\ell
=1}^{\nu }S_{\ell }\left( W\right) ^{2}\right) ^{2}.  \label{def delta}
\end{equation}%
Here the infimum is taken over all collections $\left\{ S_{\ell }\right\}
_{\ell =1}^{\nu }$ of linear forms $S_{\ell }\left( W\right) =f_{\ell
,1}x+f_{\ell ,2}y+f_{\ell ,3}z$ with $W\in \mathbb{S}^{2}$ and coefficients $%
f_{\ell ,j}$ of modulus at most $C_{0}$, which will be determined in (\ref{R
bound}) below. Since the infimum is taken over a compact set, it is
achieved, and must then be positive since $\mathbf{L}$ cannot be written as
a sum of squares of linear forms by Theorem \ref{Cho's thm}.

\begin{lemma}
\label{crucial}With notation as above, there is a positive constant $C$ such
that 
\begin{equation}
\mathcal{C}_{1,\omega _{\beta }}^{\nu }\left( \tau \right) \geq \left( \frac{%
\delta _{\nu }}{2C\tau ^{\frac{\beta }{2}}}\right) ^{\frac{1}{2-\beta }%
}=\left( \frac{\delta _{\nu }}{2C}\right) ^{\frac{1}{2-\beta }}\tau ^{-\frac{%
\beta }{4-2\beta }}.  \label{crucial low}
\end{equation}
\end{lemma}

\begin{proof}
We claim the inequalities%
\begin{eqnarray}
\delta _{\nu } &\leq &\mathbf{L}\left( \frac{W}{\left\vert W\right\vert }%
\right) -\sum_{\ell =1}^{\nu }S_{\ell }\left( \frac{W}{\left\vert
W\right\vert }\right) ^{2}=\left\vert \frac{\mathbf{L}\left( W\right)
-\sum_{\ell =1}^{\nu }S_{\ell }\left( W\right) ^{2}}{\left\vert W\right\vert
^{2}}\right\vert  \label{claimed ineq} \\
&\leq &2C\left\Vert \mathcal{G}\right\Vert _{1,\omega }\omega \left(
\left\vert W\right\vert \right) =2C\left\Vert \mathcal{G}\right\Vert
_{1,\omega }^{2}\omega \left( \frac{\sqrt{\tau }}{\left\Vert \mathcal{G}%
\right\Vert _{1,\omega }}\right) ,  \notag
\end{eqnarray}%
which then lead directly to (\ref{crucial low}) with $\omega =\omega _{\beta
}$ upon using the definition of $\mathcal{C}_{1,\omega }^{\nu }\left( \tau
\right) $. To see this claim we write 
\begin{eqnarray*}
&&\mathbf{L}\left( W\right) +\tau =\sum_{\ell =1}^{\nu }G_{\ell }\left(
W\right) ^{2}, \\
&&\text{where }G_{\ell }\left( W\right) =a_{\ell }+S_{\ell }\left( W\right)
+R_{\ell }\left( W\right) , \\
&&a_{\ell }\text{ is a constant, }S_{\ell }\left( W\right) \text{ is linear
and }R_{\ell }\left( W\right) \text{ is }o\left( \left\vert W\right\vert
\right) .
\end{eqnarray*}%
Then setting $W=0$ in the equation gives 
\begin{equation*}
\tau =\sum_{\ell =1}^{\nu }a_{\ell }^{2},
\end{equation*}%
and so%
\begin{eqnarray*}
\mathbf{L}\left( W\right) &=&\sum_{\ell =1}^{\nu }\left[ a_{\ell }+S_{\ell
}\left( W\right) +R_{\ell }\left( W\right) \right] ^{2}-\tau \\
&=&\left( \sum_{\ell =1}^{\nu }a_{\ell }^{2}\right) -\tau +\sum_{\ell
=1}^{\nu }2a_{\ell }S_{\ell }\left( W\right) \\
&&+\sum_{\ell =1}^{\nu }2a_{\ell }R_{\ell }\left( W\right) +\sum_{\ell
=1}^{\nu }\left[ S_{\ell }\left( W\right) +R_{\ell }\left( W\right) \right]
^{2}.
\end{eqnarray*}%
Now the sum of terms in the middle line vanishes identically since it is a
linear polynomial, and all of the remaining terms in the final line of the
identity vanish to order greater than $1$ at the origin (simply
differentiate this identity and then evaluate at $W=0$, using that $\nabla 
\mathbf{L}\left( 0\right) $ and $\nabla R_{\ell }\left( 0\right) $ vanish).
Thus we conclude that%
\begin{equation}
\mathbf{L}\left( W\right) -\sum_{\ell =1}^{\nu }\left[ S_{\ell }\left(
W\right) +R_{\ell }\left( W\right) \right] ^{2}=\sum_{\ell =1}^{\nu
}2a_{\ell }R_{\ell }\left( W\right) ,  \label{conclude''}
\end{equation}%
where%
\begin{eqnarray}
\sum_{\ell =1}^{\nu }a_{\ell }^{2} &=&\tau ,  \label{terms_est'} \\
\sum_{\ell =1}^{\nu }|S_{\ell }\left( W\right) | &\leq &\left\Vert \mathbf{G}%
\right\Vert _{1,\omega }|W|,  \notag \\
\sum_{\ell =1}^{\nu }|R_{\ell }\left( W\right) | &\leq &\left\Vert \mathbf{G}%
\right\Vert _{1,\omega }|W|\omega (W).  \notag
\end{eqnarray}

From (\ref{conclude''}) we have%
\begin{eqnarray*}
\mathbf{L}\left( W\right) -\sum_{\ell =1}^{\nu }S_{\ell }\left( W\right)
^{2} &=&\mathbf{L}\left( W\right) -\sum_{\ell =1}^{\nu }\left[ S_{\ell
}\left( W\right) +R_{\ell }\left( W\right) \right] ^{2}+\sum_{\ell =1}^{\nu }%
\left[ 2S_{\ell }\left( W\right) +R_{\ell }\left( W\right) \right] R_{\ell
}\left( W\right) \\
&=&h_{1}\left( W\right) +h_{2}\left( W\right) \equiv h\left( W\right) ,
\end{eqnarray*}%
where 
\begin{eqnarray*}
h_{1}\left( W\right) &\equiv &\sum_{\ell =1}^{\nu }2a_{\ell }R_{\ell }\left(
W\right) , \\
h_{2}\left( W\right) &\equiv &\sum_{\ell =1}^{\nu }\left[ 2S_{\ell }\left(
W\right) +R_{\ell }\left( W\right) \right] R_{\ell }\left( W\right) .
\end{eqnarray*}%
Using (\ref{terms_est'}) we obtain 
\begin{align*}
\left\vert h_{1}\left( W\right) \right\vert & \leq C\sqrt{\tau }\left\Vert 
\mathcal{G}\right\Vert _{1,\omega }\left\vert W\right\vert \omega \left(
\left\vert W\right\vert \right) \\
\left\vert h_{2}\left( W\right) \right\vert & \leq C\left\Vert \mathcal{G}%
\right\Vert _{1,\omega }\left\vert W\right\vert \left\vert R_{\ell }\left(
W\right) \right\vert \leq C\left\Vert \mathcal{G}\right\Vert _{1,\omega
}\left\vert W\right\vert ^{2}\omega \left( \left\vert W\right\vert \right) .
\end{align*}%
So altogether we have%
\begin{equation*}
\left\vert h\left( W\right) \right\vert \leq \left\vert h_{1}\left( W\right)
\right\vert +\left\vert h_{2}\left( W\right) \right\vert \leq C\sqrt{\tau }%
\left\Vert \mathcal{G}\right\Vert _{1,\omega }\left\vert W\right\vert \omega
\left( \left\vert W\right\vert \right) +C\left\Vert \mathcal{G}\right\Vert
_{1,\omega }^{2}\left\vert W\right\vert ^{2}\omega \left( \left\vert
W\right\vert \right) ,
\end{equation*}%
provided $\left\vert W\right\vert \leq 1$.

Now note that we may assume without loss of generality that $\frac{\sqrt{%
\tau }}{\left\Vert \mathcal{G}\right\Vert _{1,\omega }}\leq 1$, since
otherwise (\ref{crucial low}) holds trivially. Thus if $\left\vert
W\right\vert =\frac{\sqrt{\tau }}{\left\Vert \mathbf{G}\right\Vert
_{1,\omega }}$, then we have 
\begin{equation*}
\left\vert h\left( W\right) \right\vert \leq C\tau \omega \left( \frac{\sqrt{%
\tau }}{\left\Vert \mathcal{G}\right\Vert _{1,\omega }}\right) +C\tau \omega
\left( \frac{\sqrt{\tau }}{\left\Vert \mathcal{G}\right\Vert _{1,\omega }}%
\right) .
\end{equation*}%
Consequently we conclude%
\begin{equation*}
\frac{\tau }{\left\Vert \mathcal{G}\right\Vert _{1,\omega }^{2}}\left\vert 
\mathbf{L}\left( \frac{W}{\left\vert W\right\vert }\right) -\sum_{\ell
=1}^{\nu }S_{\ell }\left( \frac{W}{\left\vert W\right\vert }\right)
^{2}\right\vert =\left\vert \mathbf{L}\left( W\right) -\sum_{\ell =1}^{\nu
}S_{\ell }\left( W\right) ^{2}\right\vert =\left\vert h\left( W\right)
\right\vert \leq C\tau \omega \left( \frac{\sqrt{\tau }}{\left\Vert \mathcal{%
G}\right\Vert _{1,\omega }}\right) .
\end{equation*}%
Also note that from $\sum_{\ell =1}^{\nu }\left\vert S_{\ell }\left(
W\right) \right\vert \leq C\sqrt{\mathbf{L}\left( W\right) +\tau }$, we
obtain for $0<\tau <1$ that 
\begin{equation}
\left\vert f_{\ell ,\alpha }\right\vert \leq C_{0}\equiv C\sqrt{\mathbf{L}%
\left( W\right) +1}.  \label{R bound}
\end{equation}%
This completes the proof of our claimed inequality (\ref{claimed ineq}), and
hence also that of Lemma \ref{crucial}.
\end{proof}

For a positive integer $\nu \in \mathbb{N}$ and a modulus of continuity $%
\omega $, we say that a smooth nonnegative\ matrix function $\mathbf{F}%
\left( W,t\right) $ has the property $\mathcal{SOS}_{1,\omega }^{\nu }$ if
there exists a finite collection $\mathcal{G}=\left\{ G_{\ell }\left(
W,t\right) \right\} _{\ell =1}^{\nu }\in \oplus ^{\nu }C^{1,\omega }\left(
\Omega \right) $ of vector fields $G_{\ell }\left( W,t\right) \in
C^{1,\omega }\left( \Omega \right) $ such that%
\begin{equation*}
\mathbf{F}\left( W,t\right) =\sum_{\ell =1}^{\nu }G_{\ell }\left( W,t\right)
G_{\ell }\left( W,t\right) ^{\func{tr}},\ \ \ \ \ \left( W,t\right) \in
\Omega .
\end{equation*}

Now let $\varphi :\left( 0,1\right) \rightarrow \left( 0,1\right) $ be a
strictly increasing elliptical flat smooth function on $\left( 0,1\right) $,
and define the matrix function%
\begin{eqnarray}
\mathbf{F}_{\varphi ,\psi }\left( W,t\right) &\equiv &\varphi \left(
t\right) \mathbf{L}\left( W\right) +\left( \psi \left( t\right) +\eta \left(
t,r\right) \right) \mathbb{I}_{3},  \label{def f} \\
\text{for }\ \left( W,t\right) &\in &\Omega \equiv B_{\mathbb{R}^{3}}\left(
0,1\right) \times \left( -1,1\right) ,  \notag
\end{eqnarray}%
where $\mathbb{I}_{3}$ is the $3\times 3$ identity matrix, $r=\left\vert
W\right\vert =\sqrt{x^{2}+y^{2}+z^{2}}$, and $\psi \left( t\right) $ and $%
\eta \left( t,r\right) $ are smooth nonnegative functions constructed as
follows. The function $\eta \left( t,r\right) $ is chosen to have the form $%
\eta \left( t,r\right) =\varphi \left( r\right) h\left( \frac{t}{r}\right) $
where $h$ is a smooth nonnegative function supported in $\left( -1,1\right) $
with $h\left( 0\right) =1$. With these constructions completed, we see that $%
\mathbf{F}_{\varphi ,\psi }$ is a diagonally elliptical flat smooth $3\times
3$ matrix function on $B_{\mathbb{R}^{3}}\left( 0,1\right) \times \left(
-1,1\right) $.

\begin{lemma}
\label{failure}Suppose $0<\beta <1$ and let $\mathbf{F}_{\varphi ,\psi
}\left( W,t\right) $ be as in (\ref{def f}). If 
\begin{equation}
\lim_{t\rightarrow 0}\frac{\psi \left( t\right) }{\varphi \left( t\right) ^{%
\frac{2}{\beta }}t^{\frac{4}{\beta }}}=0,  \label{contra}
\end{equation}%
then $\mathbf{F}_{\varphi ,\psi }$ fails to satisfy $\mathcal{SOS}_{1,\omega
_{\beta }}^{\nu }$ for any $\nu \in \mathbb{N}$. Note in particular we may
even take both $\varphi $ and $\psi $ to be nearly monotone functions on $%
\left( -1,1\right) $.
\end{lemma}

\begin{proof}
Assume that $\mathbf{F}_{\varphi ,\psi }\left( W,t\right) $ has the property 
$\mathcal{SOS}_{1,\omega _{\beta }}^{\nu }$ for some $\nu \in \mathbb{N}$,
i.e. $\mathbf{F}_{\varphi ,\psi }=\sum_{\ell =1}^{\nu }G_{\ell }^{2}$ where $%
G_{\ell }\in C^{1,\omega }\left( \Omega \right) $, i.e.%
\begin{eqnarray*}
&&\varphi \left( t\right) \mathbf{L}\left( x,y,z,t\right) +\left[ \psi
\left( t\right) +\sigma \left( r\right) h\left( \frac{t}{r}\right) \right] 
\mathbb{I}_{3}=\sum_{\ell =1}^{\nu }G_{\ell }\left( x,y,z,t\right) ^{2}, \\
&&\text{for }\left( x,y,z,t\right) \in \Omega =B_{\mathbb{R}^{3}}\left(
0,1\right) \times \left( -1,1\right) .
\end{eqnarray*}%
Then since $h\left( \frac{t}{r}\right) $ vanishes for $r\leq \left\vert
t\right\vert $, we have with $W\equiv \left( x,y,z\right) $, and without
loss of generality $t>0$, that%
\begin{equation*}
\varphi \left( t\right) \mathbf{L}\left( W\right) +\psi \left( t\right) 
\mathbb{I}_{3}=\sum_{\ell =1}^{\nu }G_{\ell }\left( W,t\right) ^{2},\ \ \ \
\ \text{for }r\leq t,
\end{equation*}%
and replacing $W$ by $tW$ we have, 
\begin{eqnarray*}
\varphi \left( t\right) \mathbf{L}\left( tW\right) +\psi \left( t\right) 
\mathbb{I}_{3} &=&\sum_{\ell =1}^{\nu }G_{\ell }\left( tW,t\right) ^{2}, \\
\text{for }\left\vert W\right\vert &\leq &1,t\in \left( 0,1\right) .
\end{eqnarray*}%
Multiplying by $\frac{1}{\varphi \left( t\right) t^{2}}$, and using that $%
\mathbf{L}$ is homogeneous of degree two, we obtain%
\begin{eqnarray*}
\mathbf{L}\left( W\right) +\frac{\psi \left( t\right) }{\varphi \left(
t\right) t^{2}}\mathbb{I}_{3} &=&\sum_{\ell =1}^{\nu }\left( \frac{G_{\ell
}\left( tW,t\right) }{\sqrt{\varphi \left( t\right) }t}\right) ^{2}, \\
\text{for }\left\vert W\right\vert &\leq &1,t\in \left( 0,1\right) .
\end{eqnarray*}

Since $G_{\ell }\in C^{1,\omega }\left( B_{\mathbb{R}^{4}}\left( 0,1\right)
\times \left( -1,1\right) \right) $, the functions $W\rightarrow G_{\ell
}\left( W,t\right) $ lie in a bounded set in $C^{1,\omega }\left( B_{\mathbb{%
R}^{3}}\left( 0,1\right) \right) $ independent of $t$ and $j$, and hence
also the collection of functions%
\begin{equation*}
H_{\ell }^{t}\left( W\right) \equiv G_{\ell }\left( tW,t\right) ,\ \ \ \ \
1\leq \ell \leq \nu ,t\in \left( 0,1\right) ,
\end{equation*}%
is bounded in $C^{1,\omega }\left( B_{\mathbb{R}^{4}}\left( 0,1\right)
\right) $, say 
\begin{equation}
\sum_{\ell =1}^{\nu }\left\Vert H_{\ell }^{t}\right\Vert _{C^{1,\omega
}\left( B_{\mathbb{R}^{3}}\left( 0,1\right) \right) }\leq \mathfrak{N}_{\nu
},\ \ \ \ \ t\in \left( 0,1\right) .  \label{say bound}
\end{equation}%
Thus with $\tau =\tau \left( t\right) \equiv \frac{\psi \left( t\right) }{%
\varphi \left( t\right) t^{2}}$, we have from (\ref{say bound}) and (\ref%
{crucial low}) that%
\begin{eqnarray*}
&&\frac{\mathfrak{N}_{\nu }}{\sqrt{\varphi \left( t\right) t^{2}}}\geq
\sum_{\ell =1}^{\nu }\left\Vert \frac{H_{\ell }^{t}}{\sqrt{\varphi \left(
t\right) }t^{2}}\right\Vert _{C^{1,\omega }\left( B_{\mathbb{R}^{4}}\left(
0,1\right) \right) }\geq \mathcal{C}_{\nu }^{\omega }\left( \tau \left(
t\right) \right) \\
&\geq &\left( \frac{\delta _{\nu }}{C}\right) ^{\frac{1}{2-\beta }}\tau
\left( t\right) ^{-\frac{\beta }{4-2\beta }}=\left( \frac{\delta _{\nu }}{C}%
\right) ^{\frac{1}{2-\beta }}\left( \frac{\psi \left( t\right) }{\varphi
\left( t\right) t^{2}}\right) ^{-\frac{\beta }{4-2\beta }},
\end{eqnarray*}%
and hence%
\begin{equation*}
\left( \frac{\delta _{\nu }}{C}\right) ^{\frac{1}{2-\beta }}\leq
\liminf_{t\rightarrow 0}\frac{\mathfrak{N}_{\nu }}{\sqrt{\varphi \left(
t\right) t^{2}}}\left( \frac{\psi \left( t\right) }{\varphi \left( t\right)
t^{2}}\right) ^{\frac{\beta }{4-2\beta }}=\mathfrak{N}_{\nu
}\liminf_{t\rightarrow 0}\left( \frac{\psi \left( t\right) }{\varphi \left(
t\right) ^{\frac{2}{\beta }}t^{\frac{4}{\beta }}}\right) ^{\frac{\beta }{%
4-2\beta }},
\end{equation*}%
contradicting (\ref{contra}) as required. This completes the proof of Lemma %
\ref{failure}.
\end{proof}

\subsubsection{Sharpness}

In this subsubsection we take 
\begin{equation*}
\psi \left( t\right) \leq C\varphi \left( t\right) ^{\frac{2}{\beta }}t^{%
\frac{4}{\beta }},\ \ \ \ \ \text{for some }\beta <1,
\end{equation*}%
where both $\varphi $ and $\psi $ are nearly monotone on $\left( -1,1\right) 
$. Then by Lemma \ref{failure}, the matrix function $\mathbf{F}_{\varphi
,\psi }$ as in (\ref{def f}) fails to be a finite sum of squares of $%
C^{1,\beta }$ vector fields. On the other hand, we now show that (\ref{off
diag hyp}) holds with $\varepsilon <\frac{1}{4}$.

\begin{lemma}
Let $\varphi $ be nearly monotone on $\left( -1,1\right) $. The off diagonal
entries of $\mathbf{F}_{\varphi ,\psi }$ satisfy (\ref{off diag hyp}) for
some $\delta ,\delta ^{\prime },\delta ^{\prime \prime }>0$ if $%
0<\varepsilon <\frac{1}{4}$.
\end{lemma}

\begin{proof}
The three off diagonal entries of $\mathbf{F}_{\varphi ,\psi _{s}}\left(
x,y,z,t\right) =\left[ a_{k,j}\right] _{k,j=1}^{3}$ are%
\begin{eqnarray*}
a_{1,2}\left( x,y,z,t\right) &=&xy\varphi \left( t\right) , \\
a_{2,3}\left( x,y,z,t\right) &=&yz\varphi \left( t\right) , \\
a_{1,3}\left( x,y,z,t\right) &=&zx\varphi \left( t\right) ,
\end{eqnarray*}%
and for $\left\vert \mu \right\vert \leq 4$, we have%
\begin{equation*}
\left\vert D^{\mu }\left( xy\varphi \left( t\right) \right) \right\vert
\lesssim \left\vert W\right\vert ^{2}\left\vert D^{\alpha }\varphi \left(
t\right) \right\vert +\left\vert W\right\vert \left\vert D^{\beta }\varphi
\left( t\right) \right\vert +\left\vert D^{\gamma }\varphi \left( t\right)
\right\vert ,
\end{equation*}%
where $\left\vert \alpha \right\vert \leq 4$, $\left\vert \beta \right\vert
\leq 3$ and $\left\vert \gamma \right\vert \leq 2$. Since $\varphi \left(
t\right) $ nearly monotone implies $\left\vert D^{\nu }\varphi \left(
t\right) \right\vert \leq C_{\nu ,\eta }\varphi \left( t\right) ^{1-\eta }$
for any $\eta >0$, we have%
\begin{equation*}
\left\vert D^{\mu }\left( xy\varphi \left( t\right) \right) \right\vert
\lesssim \left( 1+\left\vert W\right\vert ^{2}\right) \varphi \left(
t\right) ^{1-\eta },\ \ \ \text{for }\left\vert \mu \right\vert \leq 4.
\end{equation*}%
Now the diagonal entries are all comparable to $\left\vert W\right\vert
^{2}\varphi \left( t\right) $, and thus we obtain%
\begin{equation*}
\left\vert D^{\mu }\left( xy\varphi \left( t\right) \right) \right\vert
\lesssim \left( 1+\left\vert W\right\vert ^{2}\right) \varphi \left(
t\right) ^{1-\eta }\lesssim \left( \left\vert W\right\vert ^{2}\varphi
\left( t\right) \right) ^{\left[ \frac{1}{2}+\left( 2-\left\vert \mu
\right\vert \right) \varepsilon \right] _{+}+\delta ^{\prime \prime }},
\end{equation*}%
since%
\begin{equation*}
1-\eta >\left[ \frac{1}{2}+\left( 2-\left\vert \mu \right\vert \right)
\varepsilon \right] _{+}+\delta ^{\prime \prime }\text{ for }\left\vert \mu
\right\vert \leq 4\text{ and }\varepsilon <\frac{1}{4},
\end{equation*}%
provided $\delta ^{\prime \prime }>0$ is sufficiently small.
\end{proof}

\begin{lemma}
Let $\varphi $ and $\psi $ be nearly monotone on $\left( -1,1\right) $. The
diagonal entries of $\mathbf{F}_{\varphi ,\psi }=\left[ a_{k,j}\right]
_{k,j=1}^{3}$ satisfy (\ref{diag hyp}) for $\delta >0$ and $\varepsilon >%
\frac{1}{4}$.
\end{lemma}

\begin{proof}
We note that the diagonal entries are%
\begin{eqnarray*}
a_{1,1}\left( x,y,z,t\right) &=&\varphi \left( t\right) \left( x^{2}+\lambda
y^{2}+2z^{2}\right) +\psi \left( t\right) +\varphi \left( r\right) h\left( 
\frac{t}{r}\right) , \\
a_{2,2}\left( x,y,z,t\right) &=&\varphi \left( t\right) \left( y^{2}+\lambda
z^{2}+2x^{2}\right) +\psi \left( t\right) +\varphi \left( r\right) h\left( 
\frac{t}{r}\right) , \\
a_{3,3}\left( x,y,z,t\right) &=&\varphi \left( t\right) \left( z^{2}+\lambda
x^{2}+2y^{2}\right) +\psi \left( t\right) +\varphi \left( r\right) h\left( 
\frac{t}{r}\right) ,
\end{eqnarray*}%
which are each comparable to 
\begin{equation*}
\varphi \left( t\right) \left\vert W\right\vert ^{2}+\psi \left( t\right)
+\varphi \left( r\right) h\left( \frac{t}{r}\right) .
\end{equation*}%
Recall that for any $0<\eta <1$, we have $\left\vert D_{t}^{\mu }\varphi
\left( t\right) \right\vert \leq C_{\eta }\varphi \left( t\right) ^{1-\eta }$%
. Thus for $\left\vert \mu \right\vert =1$ we have%
\begin{eqnarray*}
\left\vert D^{\mu }a_{1,1}\left( x,y,z,t\right) \right\vert &\lesssim
&\varphi \left( t\right) ^{1-\eta }\left\vert W\right\vert ^{2}+\varphi
\left( t\right) \left\vert W\right\vert +\left\vert \psi ^{\prime }\left(
t\right) \right\vert +\left\vert \varphi ^{\prime }\left( r\right)
\right\vert h\left( \frac{t}{r}\right) +\varphi \left( r\right) \frac{1}{r}%
\left( 1+\frac{t}{r}\right) \\
&\lesssim &\left( \varphi \left( t\right) r^{2}+\psi \left( t\right)
+\varphi \left( r\right) h\left( \frac{t}{r}\right) \right) ^{1-\varepsilon
+\delta ^{\prime }}\approx a_{1,1}\left( x,y,z,t\right) ^{\left[
1-\left\vert \mu \right\vert \varepsilon \right] _{+}+\delta ^{\prime }}
\end{eqnarray*}%
provided we choose $\eta $ so that $1-\eta >1-\varepsilon +\delta ^{\prime }$%
, and provided 
\begin{equation}
\varphi \left( t\right) r\lesssim \left( \varphi \left( t\right)
r^{2}+\varphi \left( t\right) ^{2}t^{4}+\varphi \left( r\right) h\left( 
\frac{t}{r}\right) \right) ^{1-\varepsilon +\delta ^{\prime }}.
\label{what we need}
\end{equation}

But this latter inequality holds since\newline
(i) if $t\leq \frac{3}{4}r$, then%
\begin{equation*}
\varphi \left( t\right) r=\varphi \left( t\right) ^{1-\varepsilon +\delta
^{\prime }}\varphi \left( t\right) ^{\varepsilon -\delta ^{\prime
}}r\lesssim \varphi \left( t\right) ^{1-\varepsilon +\delta ^{\prime
}}\varphi \left( \frac{3}{4}r\right) ^{\varepsilon -\delta ^{\prime
}}r\lesssim \left( \varphi \left( t\right) r^{2}\right) ^{1-\varepsilon
+\delta ^{\prime }},
\end{equation*}%
since $\varphi $ is flat at the origin; while\newline
(ii) if $t>\frac{3}{4}r$, then%
\begin{eqnarray*}
\left( \varphi \left( t\right) r^{2}+\varphi \left( t\right)
^{2}t^{4}\right) ^{1-\varepsilon +\delta ^{\prime }} &\approx &\varphi
\left( t\right) ^{1-\varepsilon +\delta ^{\prime }}r^{2-2\varepsilon
+2\delta ^{\prime }}+\varphi \left( t\right) ^{2-2\varepsilon +2\delta
^{\prime }}t^{4-4\varepsilon +4\delta ^{\prime }} \\
&\gtrsim &\left\{ 
\begin{array}{ccc}
\varphi \left( t\right) r & \text{ if } & r\geq \varphi \left( t\right) ^{%
\frac{\varepsilon -\delta ^{\prime }}{1-2\varepsilon +2\delta ^{\prime }}}
\\ 
\varphi \left( t\right) r & \text{ if } & r<\varphi \left( t\right) ^{\frac{%
\varepsilon -\delta ^{\prime }}{1-2\varepsilon +2\delta ^{\prime }}}%
\end{array}%
\right.
\end{eqnarray*}%
since if $r<\varphi \left( t\right) ^{\frac{\varepsilon -\delta ^{\prime }}{%
1-2\varepsilon +2\delta ^{\prime }}}$, then%
\begin{equation*}
\varphi \left( t\right) r<\varphi \left( t\right) ^{1+\frac{\varepsilon
-\delta ^{\prime }}{1-2\varepsilon +2\delta ^{\prime }}}=\varphi \left(
t\right) ^{\frac{1-\varepsilon +\delta ^{\prime }}{1-2\varepsilon +2\delta
^{\prime }}}\lesssim \varphi \left( t\right) ^{2-2\varepsilon +2\delta
^{\prime }}t^{4-4\varepsilon +4\delta ^{\prime }}
\end{equation*}%
since%
\begin{equation*}
\frac{1-\varepsilon +\delta ^{\prime }}{1-2\varepsilon +2\delta ^{\prime }}%
>2-2\varepsilon +2\delta ^{\prime },\ \ \ \text{i.e. }
\end{equation*}%
for some $\delta ^{\prime }>0$, i.e. $\frac{1-\varepsilon }{1-2\varepsilon }%
>2-2\varepsilon $, i.e. 
\begin{equation*}
1-\varepsilon >\left( 1-2\varepsilon \right) \left( 2-2\varepsilon \right)
,\ \ \ \ \ \text{i.e. }\frac{1}{4}<\varepsilon <1.
\end{equation*}%
The case $r\geq \varphi \left( t\right) ^{\frac{\varepsilon -\delta ^{\prime
}}{1-2\varepsilon +2\delta ^{\prime }}}$ is straightforward.

Now we turn to the case $\left\vert \mu \right\vert =2$ where using that
both $\varphi $ and $\psi $ are nearly monontone on $\left( -1,1\right) $,
and from Theorem \ref{main intro}, we have%
\begin{eqnarray*}
\left\vert D^{\mu }a_{1,1}\left( x,y,z,t\right) \right\vert &\lesssim
&\varphi \left( t\right) ^{1-\eta }\left\vert W\right\vert ^{2}+\varphi
\left( t\right) ^{1-\eta }\left\vert W\right\vert +\varphi \left( t\right)
+\left\vert \psi ^{\prime }\left( t\right) \right\vert +\left\vert \psi
^{\prime \prime }\left( t\right) \right\vert \\
&&+\left\vert \varphi ^{\prime \prime }\left( r\right) \right\vert h\left( 
\frac{t}{r}\right) +\varphi ^{\prime }\left( r\right) \left\vert h^{\prime
}\left( \frac{t}{r}\right) \right\vert \frac{1}{r}\left( 1+\frac{t}{r}\right)
\\
&&+\varphi \left( r\right) \left\vert h^{\prime \prime }\left( \frac{t}{r}%
\right) \right\vert \frac{1}{r^{2}}\left( 1+\frac{t}{r}\right) ^{2} \\
&\lesssim &\left( \varphi \left( t\right) r^{2}+\psi \left( t\right)
+\varphi \left( r\right) h\left( \frac{t}{r}\right) \right) ^{1-2\varepsilon
+\delta ^{\prime }}\approx a_{1,1}\left( x,y,z,t\right) ^{\left[
1-\left\vert \mu \right\vert \varepsilon \right] _{+}+\delta ^{\prime }}
\end{eqnarray*}%
provided we choose $\eta $ so that $1-\eta >1-2\varepsilon +\delta ^{\prime
} $, and provided 
\begin{equation}
\varphi \left( t\right) ^{1-\eta }r\lesssim \left( \varphi \left( t\right)
r^{2}+\psi \left( t\right) +\varphi \left( r\right) h\left( \frac{t}{r}%
\right) \right) ^{1-2\varepsilon +\delta ^{\prime }},  \label{what we need'}
\end{equation}%
which of course holds for $\varepsilon >\frac{1}{4}$. Similar calculations
hold for $\left\vert \mu \right\vert =3,4$.
\end{proof}

We have thus demonstrated sharpness of Theorem \ref{final n Grushin} when $%
\varepsilon =\frac{1}{4}$ in the following sense. We do not know if similar
sharpness holds for $\frac{1}{4}<\varepsilon <1$.

\begin{theorem}
\label{flat counter}Let $0<\beta <1$. The diagonally elliptic smooth flat
matrix function $\mathbf{F}_{\varphi ,\psi }$ constructed above satisfies
the diagonal estimates (\ref{diag hyp}) for all $\varepsilon >\frac{1}{4}$
and $\delta >0$, and the off diagonal estimates (\ref{off diag hyp}) for all 
$\varepsilon <\frac{1}{4}$ and $\delta >0$, yet is \emph{not} $SOS_{1,\omega
_{\beta }}$, hence not $SOS_{1,\omega _{1}}$. Of course, if for a diagonally
elliptic smooth flat matrix function $\mathbf{F}$, both (\ref{diag hyp}) and
(\ref{off diag hyp}) hold for $\varepsilon =\frac{1}{4}$ and some $\delta >0$%
, then Theorem \ref{final n Grushin} shows that $\mathbf{F}$ \emph{is} $%
SOS_{2,\omega _{\delta }}$, hence $SOS_{1,\omega _{1}}$.
\end{theorem}

\subsubsection{A Grushin type subordinate matrix function, not a finite sum
of squares plus a subordinate quasiconformal block}

It is not hard to modify the above example to obtain a Grushin type
subordinate matrix function, with diagonal entries that are finite sums of
squares of $C^{2,\delta }$ functions, and that cannot be decomposed as a
finite sum of squares of vector fields plus a quasiconformal block\footnote{%
The hypoellipticity theorem in \cite{KoSa3} doesn't apply to the operator $%
\mathbf{L}=\nabla ^{\func{tr}}\mathbf{A}\left( x\right) \nabla $ with a
matrix function $A\left( x\right) $ of this form.}. Consider first the $%
7\times 7$ matrix function in block form,%
\begin{equation*}
\mathbf{M}\left( x,y,z,t,u,v,w,s\right) \equiv \left[ 
\begin{array}{cc}
\mathbb{I}_{4} & \mathbf{0}_{4\times 3} \\ 
\mathbf{0}_{3\times 4} & \mathbf{F}_{\varphi ,\psi }\left( x,y,z,t\right)%
\end{array}%
\right] ,
\end{equation*}%
where $\mathbb{I}_{4}$ is the $4\times 4$ identity matrix, and $\mathbf{0}%
_{m\times n}$ is the $m\times n$ zero matrix. Then if $f_{\varphi ,\psi
}\left( x,y,z,t\right) \equiv \func{trace}\mathbf{F}_{\varphi ,\psi }\left(
x,y,z,t\right) $, we have using $\mathbf{L}\left( x,y,z\right) \approx
\left\vert \left( x,y,z\right) \right\vert ^{2}\mathbb{I}_{3}$ that%
\begin{equation*}
\mathbf{M}\left( x,y,z,t,u,v,w\right) \approx \left[ 
\begin{array}{cc}
\mathbb{I}_{4} & \mathbf{0}_{4\times 3} \\ 
\mathbf{0}_{3\times 4} & f_{\varphi ,\psi }\left( x,y,z,t\right) \mathbb{I}%
_{3}%
\end{array}%
\right] ,
\end{equation*}%
where $f_{\varphi ,\psi }\left( x,y,z,t\right) $ is $\omega _{s}$-monotone
for all $0<s<s_{0}$ by \cite[Theorem 37]{KoSa1}, and hence is a finite sum
of squares of $C^{2,\delta }$ functions by Theorem \ref{efs eps}, yet $%
\mathbf{M}\left( x,y,z,t,u,v,w\right) $ is not a finite sum of squares of $%
C^{1,1}$ vector functions. This example shows in a striking way that
additional conditions must be assumed on the off-diagonal entries of the
matrix function $\mathbf{M}\left( x,y,z,t,u,v,w\right) $ in order for $%
\mathbf{M}$ to be a finite sum of squares of $C^{2,\delta }$ vector
functions.

However, by Theorem \ref{final n Grushin}, the matrix function $\mathbf{M}$ 
\emph{can} be decomposed as a sum of squares plus the subordinate
quasiconformal block $\mathbf{F}_{\varphi ,\psi }\left( x,y,z,t\right) $,
and we must work just a bit harder to prevent this. We consider instead the
example%
\begin{equation*}
\mathbf{N}\left( x,y,z,t,u,v,w,s\right) \equiv \left[ 
\begin{array}{cc}
\mathbf{M}\left( x,y,z,t,u,v,w\right) & \mathbf{0}_{7\times 1} \\ 
\mathbf{0}_{1\times 7} & \mathbf{G}\left( x,y,z,t\right)%
\end{array}%
\right] \sim \left[ 
\begin{array}{ccc}
\mathbb{I}_{4} & \mathbf{0}_{4\times 3} & \mathbf{0}_{4\times 1} \\ 
\mathbf{0}_{3\times 4} & f_{\varphi ,\psi }\left( x,y,z,t\right) \mathbb{I}%
_{3} & \mathbf{0}_{3\times 1} \\ 
\mathbf{0}_{1\times 4} & \mathbf{0}_{1\times 3} & g_{\varphi ,\psi }\left(
x,y,z,t\right) \mathbb{I}_{3}%
\end{array}%
\right] ,
\end{equation*}%
where 
\begin{eqnarray*}
\mathbf{F}\left( x,y,z,t\right) &=&\varphi \left( t\right) \mathbf{L}\left(
W\right) +\left\{ \left( \varphi \left( t\right) t^{2}\right) ^{4}+\varphi
\left( r\right) h\left( \frac{t}{r}\right) \right\} \mathbb{I}_{3}, \\
\mathbf{G}\left( x,y,z,t\right) &=&\rho \left( t\right) \mathbf{L}\left(
W\right) +\left\{ \left( \rho \left( t\right) t^{2}\right) ^{4}+\rho \left(
r\right) h\left( \frac{t}{r}\right) \right\} \mathbb{I}_{3},
\end{eqnarray*}%
are both examples of a $3\times 3$ matrix function that cannot be written as
a finite sum of squares of $C^{2}$ vector fields, and where $\varphi \left(
t\right) $ and $\rho \left( t\right) $ incomparable.

Note that the hypotheses of Theorem \ref{final n Grushin} fail here since
the final block is $\mathbf{G}\left( x,y,z,t\right) $, and Theorem \ref%
{final n Grushin} then requires $\mathbf{M}\left( x,y,z,t,u,v,w\right) $ to
be a sum of squares of $C^{2,\delta }$ vector functions, which it is not
since $\mathbf{F}$ is embedded in $\mathbf{M}$. The same observation holds
even if we permute rows and columns of $\mathbf{N}$ and declare a final
block of the permuted matrix to be the Grushin block. Indeed, the Grushin
block will be comparable to $\lambda \left( x,y,z,t\right) \mathbb{I}_{p}$
where $\lambda \left( x,y,z,t\right) \in \left\{ 1,f_{\varphi ,\psi }\left(
x,y,z,t\right) ,g_{\varphi ,\psi }\left( x,y,z,t\right) \right\} $, and then
the remaining block will have either $\mathbf{F}$ or $\mathbf{G}$ embedd in
it, hence cannot be a sum of squares.

In another direction, suppose that $f\left( x,y,z,w,t\right) $ and $g\left(
x,y,z,w,t\right) $ are two elliptical flat smooth functions that cannot be
written as a finite sum of squares of $C^{2,\delta }$ functions, such as can
be found in \cite{KoSa1}. Then a diagonally elliptical $7\times 7$ matrix
function $\mathbf{P}\left( x,y,z,w,t,u,v\right) $, whose diagonal elements $%
\left\{ p_{1,1},p_{2,2},p_{3,3},p_{4,4},p_{5,5},p_{6,6},p_{7,7}\right\} $
are comparable to $\left\{ 1,1,1,1,1,f,g\right\} $, cannot be decomposed as
a finite sum of squares of $C^{2,\delta }$ vector fields plus a subordinate
quasiconformal block.

\end{document}